\documentclass[reqno,11pt]{amsart}
\usepackage{amsmath,amssymb,mathrsfs,color}
\usepackage{graphicx}
\usepackage{epstopdf}
\usepackage{soul}
\usepackage{cases}
\usepackage{hyperref}
\hypersetup{hypertex=true,
	colorlinks=true,
	linkcolor=blue,
	anchorcolor=blue,
	citecolor=blue}
\usepackage[margin=1in]{geometry}
\usepackage{
indentfirst, latexsym, bm, enumerate,
wrapfig, adjustbox, subfigure, appendix,
float, bbding, xcolor, multirow, framed, lipsum, epsdice,
url, amsthm, booktabs, makecell, caption, algpseudocode, pifont,
balance, dcolumn, fontenc, extarrows, arydshln % els, titlesec, mathptmx, algorithm
}

\allowdisplaybreaks[4]
\usepackage{tikz}
\usetikzlibrary{arrows}
\usepackage{esint}%%%  $\fint$
\usepackage{mathtools}

\newtheorem{definition}{Definition}[section]
\newtheorem{theorem}{Theorem}[section]
\newtheorem{lemma}{Lemma}[section]
\newtheorem{proposition}{Proposition}[section]

\newtheorem{remark}{\bf{Remark}}[section]

\makeatletter
\newcommand{\rmnum}[1]{\romannumeral #1}
\newcommand{\Rmnum}[1]{\expandafter\@slowromancap\romannumeral #1@}
\makeatother

\numberwithin{equation}{section}

\renewcommand{\hl}[1]{#1}

\begin{document}

\title[Dissipative weak solutions to the 3D stochastic Euler equations]{Dissipative weak solutions to the three-dimensional stochastic Euler equations}

\author{Guanggan Chen}
\address{G. Chen: School of Mathematical Sciences,
Sichuan Normal University, Chengdu 610068, P. R. China}
\email{chenguanggan@hotmail.com}

\author{Wenhu Zhong}
\address{W. Zhong (Corresponding author): School of Mathematical Sciences,
Dalian University of Technology, Dalian 116024, P. R. China}
\email{zhongwenhu@dlut.edu.cn}

%\author{Zhenxin Liu}
%\address{Z. Liu (Corresponding author): School of Mathematical Sciences,
%Dalian University of Technology, Dalian 116024, P. R. China}
%\email{zxliu@dlut.edu.cn}

\date{\today}
%\date{August 31, 2021}

\begin{abstract}
This paper is concerned with the dissipative weak solutions to the three-dimensional stochastic Euler equations. For any $\gamma \in \left[ {0,\frac{1}{5}} \right)$, we establish the existence of $\gamma$-H\"{o}lder continuous weak solutions that exhibit a strict energy dissipation in expectation. For any $\gamma > \frac{1}{3}$, we prove that the three-dimensional stochastic Euler equations do not possess dissipative solutions satisfying the corresponding spatial $\gamma$-H\"{o}lder regularity. It is well known that the stochastic forcing inevitably gives rise to additional stochastic effects to the Reynolds stresses, which cannot be treated by the differential estimates in the classical convex integration scheme introduced by Buckmaster, De Lellis, Isett and Sz\'{e}kelyhidi ({\it{Ann. Math.}} \textbf{182}(1):127--172, 2015). To overcome this difficulty, we develop a novel material integration technique into the stochastic convex integration scheme, so that those stochastic Reynolds stresses are quantified.
\end{abstract}

\maketitle
\tableofcontents  % 在此处生成目录

\section{Introduction}
In this paper, we investigate the three-dimensional stochastic incompressible Euler equations
\begin{equation}
\left\{ {\begin{array}{*{20}{l}}
\mathrm{d} {u}  + \left( {{u} \cdot \nabla } \right){u} \ \mathrm{d} t + \nabla p \ \mathrm{d} t =  \mathrm{d} {\mathcal{W}},   & \mbox{on}\ \left[ {0,T} \right] \times {\mathbb{T}^3}, \\
\mathrm{div}\, {u} = 0 ,
\end{array}} \right.  \label{BouE1}
\end{equation}
where $T$ is a positive real number, and $\mathbb{T}^3$ is the three-dimensional torus. The velocity ${u}\left( {t,x} \right)$ is a $\mathbb{R}^3$-valued function on $\left[ {0,T} \right] \times {\mathbb{T}^3}$, and the pressure $p \left( {t , x} \right)$ is a $\mathbb{R}$-valued function on $ \left[ {0, T } \right] \times \mathbb{T}^3$. The stochastic Euler equations \eqref{BouE1} characterize the dynamic behavior of turbulent flows, where the Wiener process $\mathcal{W}\left( {t} \right)$ describes the stochastic effects providing an additional energy injection (see \cite{Zbl1537.35273,Zbl1555.60124}).
\par
An fundamental mechanism of turbulent flows is anomalous dissipation, which captures a striking phenomenon that the energy dissipation rate of flows remains non-zero even in the inviscid regime (see \cite{MR4595604,arXiv:2207.06833,Zbl1508.35069,Zbl1009.35062}). To understand the anomalous dissipation, Onsager \cite{MR0036116} investigated the mechanism of energy transfer across scales and heuristically proposed the celebrated conjecture: the three-dimensional deterministic Euler equations exhibit anomalous dissipation when the H\"{o}lder exponent of solutions is less than 1/3. The Onsager's conjecture identifies the critical regularity threshold for anomalous dissipation and reveals a connection between the spatial regularity of Euler flows and the persistence of energy transfer across small scales in turbulence (see \cite{arXiv:2212.08413,Zbl0818.35085}). This small-scale energy transfer is quantitatively characterized by the Kolmogorov 4/5 law.
\par
To construct dissipative weak solutions of the deterministic Euler equations, Shnirelman \cite{ShnirelmanA-1,ShnirelmanA-2} introduced a classical convex integration scheme, which was developed into the Nash iterative scheme by De Lellis and Székelyhidi \cite{Zbl1350.35146,Zbl1280.35103,Zbl1307.35205}. Combining the Nash iterative scheme with the Beltrami flows, Buckmaster, De Lellis, Isett and Sz\'{e}kelyhidi \cite{MR3302631,MR3374958} established the existence of dissipative weak solutions to the three-dimensional deterministic Euler equations. Furthermore, Buckmaster, De Lellis, Sz\'{e}kelyhidi, Vicol \cite{Zbl1480.35317}, Daneri \cite{Zbl1372.35221}, Isett \cite{Zbl1416.35194} and Novack, Vicol \cite{MR4601999} constructed the dissipative Onsager-subcritical solutions to the three-dimensional deterministic Euler equations. Together with the results of Constantin, E, Titi \cite{Zbl0818.35085} and Giri, Radu \cite{arXiv:2305.18105}, the Onsager's conjecture has been resolved for the deterministic Euler equations.
\par
However, the three-dimensional stochastic Euler equations \eqref{BouE1} are driven by the stochastic forcing whose temporal irregularity makes the material derivative of Reynolds stress ill-defined, so that the classical Nash iteration scheme is not directly applicable to Eq.\eqref {BouE1}. To overcome this difficulty, Hofmanov\'{a}, Lange, Pappalettera \cite{Zbl1537.35273}, L\"{u}, L\"{u}, Zhu \cite{arXiv:2505.06915} and Pappalettera, Triggiano \cite{arXiv:2511.21616} regularized the stochastic forcing through a Wong–Zakai approximation, so that their gluing procedure admits the key differential estimates analogous to those in the classical Nash iteration scheme. Nevertheless, their construction requires additional weighted trace-class conditions to control the stochastic contributions in the iterative procedure. This stochastic convex integration scheme does not directly address the challenges arising from the temporal irregularity and low spatial regularity of the stochastic forcing. To date, the existence of dissipative weak solutions to Eq.\eqref{BouE1} without altering the fundamental properties of the stochastic forcing remains an open problem.
\par
In this paper, our purpose is to construct dissipative weak solutions to the three-dimensional stochastic Euler equations \eqref{BouE1} without altering the fundamental properties of the Wiener process $\mathcal{W}\left( {t} \right)$. To overcome the difficulties of temporal irregularity, we develop a novel material integration technique into the stochastic convex integration scheme. Then, the stochastic contributions of $\mathcal{W}\left( {t} \right)$ to the Reynolds stress are quantified by It\^{o} stochastic integrals. Our principal result is presented as follows.
\par
\begin{theorem}\label{CTSE-O-C}
Let the Wiener process ${\mathcal{W}}\left( {t} \right)$ be in the collection $\mathscr{W}$ and satisfy
\begin{equation}\label{ansjf-SFFGG}
\mathrm{Tr}\, Q + \mathbb{E}\mathop {\sup }\limits_{t \in \left[ {0,T} \right]} \left\| {\mathcal{W}\left( t \right)} \right\|_{{\mathcal{C}^{\frac{2}{3}} }}^2 < + \infty ,
\end{equation}
where $Q$ is the covariance operator of ${\mathcal{W}}\left( {t} \right)$. Then, the following results hold.
\par
\vspace{1em}
\noindent$(\rmnum{1})$. For any $\gamma  \in \left[ {0,\frac{1}{5}} \right)$, there exists a stopping time $\mathfrak{t}$ such that Eq.\eqref{BouE1} possesses a probabilistically strong and analytically weak solution ${u }\left( t ,x \right) \in {L^\infty }\left( {\left[ {0, T \wedge {\mathfrak{t}}} \right];{\mathcal{H}}} \right)$ satisfying 
\begin{equation}
\mathop {\sup }\limits_{t \in \left[ {0,T \wedge \mathfrak{t}} \right]} {\left\| {u\left( t \right)} \right\|_{{\mathcal{C}^\gamma }}} \leqslant C\left( {\gamma } \right)   , \ \ \mathbb{P}\mbox{-a.s.} ,  \label{djfkshhgss}
\end{equation}
and the strict energy inequality
\begin{equation}\label{SHUDJ-NENH-B-SIDIS}
\mathbb{E} \left\| {u\left( t \wedge  \mathfrak{t} \right)} \right\|_\mathcal{H}^2 <  \mathbb{E}  \left\| {u\left( s  \wedge  \mathfrak{t} \right)} \right\|_\mathcal{H}^2  + \mathrm{Tr}\, Q \cdot \mathbb{E} \left( {t \wedge  \mathfrak{t} - s \wedge  \mathfrak{t}} \right)  , \ \ \forall \, 0 \leqslant s < t \leqslant T ;
\end{equation}
\par
\par
\noindent$(\rmnum{2})$. For any $\gamma  > \frac{1}{3}$, Eq.\eqref{BouE1} does not possess analytically weak solutions satisfying the regularity \eqref{djfkshhgss} and the strict energy inequality \eqref{SHUDJ-NENH-B-SIDIS}.
\end{theorem}
\par
The proof of the first result in Theorem \ref{CTSE-O-C} is based on the stochastic convex integration scheme built upon the classical Beltrami flow. The overlap of perturbations stemming from the Beltrami flow at adjacent levels gives rise to additional cross-interaction errors. To control these interactions, we adopt the pressure compensation mechanism to absorb the gradient components and incorporate the residual terms into the Reynolds stress. However, the stochastic forcing introduces It\^{o}-integral contributions to the Reynolds stress. To preserve their stochastic structure, we control these terms through a material integration technique. Indeed, the increments induced by the temporal roughness of Brownian motions can be absorbed by the rapidly decaying inductive scale. This allows us to find a uniform, almost surely finite random constant controlling the material integration of Reynolds stress.
\par
The material integration technique provides a novel stochastic convex integration scheme that enables us to handle the temporal roughness of Brownian motions directly. In contrast to the Mikado-flow-based convex integration scheme of L\"{u}, L\"{u}, Zhu \cite{arXiv:2505.06915} and the temporal regularization of Kinra, Koley \cite{Zbl08194795}, our material integration technique avoids the sophisticated gluing procedure and Wong--Zakai approximations. As a consequence, the iterative framework permits a weaker spatial regularity assumption \eqref{ansjf-SFFGG} on the Wiener process $\mathcal{W}\left( {t} \right)$.
\par
To prove the second result in Theorem \ref{CTSE-O-C}, the key procedure is to establish the Kolmogorov 4/5 law for the three-dimensional stochastic Euler equations \eqref{BouE1}. We combine the K\'{a}rm\'{a}n-Howarth-Monin relationship and the strict energy inequality \eqref{SHUDJ-NENH-B-SIDIS} to obtain the Kolmogorov 4/5 law in the finite-time setting. In fact, the K\'{a}rm\'{a}n-Howarth-Monin relationship connects the averaged third-order longitudinal structure functions with the energy balance across scales. The strict energy inequality \eqref{SHUDJ-NENH-B-SIDIS} ensures a strictly positive dissipation rate. However, the finite-time formulation produces an additional term stemming from the two-point correlation functions at distinct times. We control this temporal term in the small-scale limit through the spatial translation continuity of the velocity. Finally, combining the Kolmogorov 4/5 law with a contradiction argument, we demonstrate that Eq.\eqref{BouE1} does not possess weak solutions with the regularity \eqref{djfkshhgss} and the dissipative property \eqref{SHUDJ-NENH-B-SIDIS}.
\par
We remark that Theorem \ref{CTSE-O-C} establishes the existence of dissipative weak solutions to Eq.\eqref{BouE1} with the spatial H\"{o}lder exponent $\gamma  \in \left[ {0,\frac{1}{5}} \right)$ and shows that no dissipative weak solutions to Eq.\eqref{BouE1} with the spatial H\"{o}lder exponent $\gamma  > \frac{1}{3}$ can be constructed. However, the range $\gamma\in \left[ {\frac{1}{5} , \frac{1}{3}} \right] $ is not attained by our construction, since the inductive estimate of Nash error fails to satisfy the parameter requirements of iteration. This limitation can be overcome by combining Mikado flows with gluing procedure, but the stochastic forcing requires additional regularity assumptions (see \cite{Zbl1537.35273,arXiv:2505.06915}). Our construction may provide insights into relaxing the regularity requirements on the stochastic forcing in the stochastic Mikado-flow convex integration scheme.
\par
The remainder of this paper is organized as follows. Section 2 introduces the functional setting, the stochastic forcing, and the definition of solutions to Eq. \eqref{BouE1}. Section 3 presents the stochastic Nash iteration scheme. Section 4 proves the existence of dissipative, probabilistically strong and analytically weak solutions in Theorem \ref{CTSE-O-C}. Section 5 establishes the Kolmogorov 4/5 law and proves the nonexistence of dissipative weak solutions in Theorem \ref{CTSE-O-C}.
%%%%%%%%%%%%%%%%%%%%%%%%%%%%%%%%%%%%%%%%%%%%%%
%%%%%%%%%%%%%%%%%%%%%%%%%%%%%%%%%%%%%%%%%%%%%%
%%%%%%%%%%%%%%%%%%%%%%%%%%%%%%%%%%%%%%%%%%%%%%
\section{Preliminaries}\label{Pre-w}
This section states basic setups of spaces, noises, stopping times and the definition of solutions.
%%%%%%%%%%%%%%%%%%%%%%%%%%%%%%%%%%%%%%%%%%%%%%
%%%%%%%%%%%%%%%%%%%%%%%%%%%%%%%%%%%%%%%%%%%%%%
%%%%%%%%%%%%%%%%%%%%%%%%%%%%%%%%%%%%%%%%%%%%%%
\subsection{Spaces and noises}\label{NOis-W}
For $m \in \mathbb{Z}^+ \cup \left\{ 0 \right\}$ and $k \geqslant 1$, let ${W^{m,k}}\left( {{\mathbb{T}^3}} \right)$ be the Sobolev space. Define the Hilbert space $\mathcal{H} \triangleq \left\{ {u\left( x \right) \in {W^{0,2}}\left( {{\mathbb{T}^3}} \right) \ \big{|}\  \nabla  \cdot u\left( x \right) = 0,\ \int_{{\mathbb{T}^3}} {u\left( {t,x} \right)} \, \mathrm{d}x = \mathbf{0}} \right\}$ with the inner product
\begin{equation*}
\left( {u,v} \right)_{\mathbb{T}^3} = \int_{{\mathbb{T}^3}} {u \left( x \right) \cdot v \left( x \right)} \, \mathrm{d}x,\ \ u\left( x \right) , v\left( x \right) \in \mathcal{H},
\end{equation*}
whose corresponding norm is denoted by $\left\|  \cdot  \right\|_{\mathcal{H}}$. Further define the space ${\mathcal{H}^1} \triangleq \mathcal{H} \cap {W^{1,2}}\left( {{\mathbb{T}^3}} \right)$ with the norm
\begin{equation*}
{\left\| u \right\|_{\mathcal{H}^1}} = {\left( {\int_{{\mathbb{T}^3}} {\sum\limits_{i = 1}^3 {{{\left| {\nabla {u^i}\left( x \right)} \right|}^2}} } \, \mathrm{d}x} \right)^{\frac{1}{2}}},\ \ u\left( x \right) \triangleq \left( {{u^1}\left( x \right),{u^2}\left( x \right),{u^3}\left( x \right)} \right) \in {\mathcal{H}^1} .
\end{equation*}
\par
Let ${\mathfrak{D}^{\sigma}}\left(  \cdot  \right)$ mean the differential derivative with order ${\sigma}$ of functions regarding the spatial variable $x$. For $N \in \mathbb{N}$, define the continuous function space ${\mathcal{C}^N}\left( {{\mathbb{T}^3}} \right)$, whose norm is
\begin{equation*}
{\left\| {u} \right\|_{{\mathcal{C}^N}}} = \sum\limits_{j = 0}^N {\mathop {\max }\limits_{\left| \sigma  \right| = j} \mathop {\sup }\limits_{x \in {\mathbb{T}^3}} {\left| {{\mathfrak{D}^{\sigma} }u\left( x \right)} \right|}},\ \ u\left( x \right) \in {\mathcal{C}^N}\left( {{\mathbb{T}^3}} \right) .
\end{equation*}
Let ${\mathcal{C}^{N + \xi }}\left( {{\mathbb{T}^3}} \right)$ with $\xi  \in \left( {0,1} \right)$ be the H\"{o}lder space consisting of continuous functions $u \left( x \right) \in  {\mathcal{C}^N}\left( {{\mathbb{T}^3}} \right)$ such that
\begin{equation*}
{\left\| {u} \right\|_{{\mathcal{C}^{N + \xi }}}} \triangleq {\left\| {u} \right\|_{{\mathcal{C}^N}}} +  {\left[ u \right]_{N + \xi }}  < + \infty ,
\end{equation*}
where ${\left[ u \right]_{N + \xi }} = \mathop {\max }\limits_{\left| \sigma  \right| = N} \mathop {\sup }\limits_{x \ne y } \frac{{\left| {{\mathfrak{D}^\sigma }u\left( x \right) - {\mathfrak{D}^\sigma }u\left( y \right)} \right|}}{{{{\left| {x - y} \right|}^\xi }}}$.
\par
\begin{lemma}[\cite{MR3302631}]\label{TSP-E-L-A-1}
Suppose $\left| {t - {t_0}} \right|{\left\| u \right\|_1} \lesssim 1$. Let $w{\left( t ,x \right)}$ be the solution of the transport equation
\begin{equation*}
\left\{ {\begin{array}{*{20}{l}}
{\partial _t}w + \left( {v \cdot \nabla } \right)w = u ,   &\mbox{on}\ \left[ {{t_0}, T } \right]  \times  \mathbb{T}^3, \\
w\left( {{t_0}} \right) = {w_0} ,
\end{array}} \right.
\end{equation*}
where $v\left( t, x \right)$ and $u\left( t, x \right)$ are the vector functions on $\left[ {{t_0}, T } \right]  \times   \mathbb{T}^3$. Then, for any $\alpha \in \left[ {0,1} \right)$ and $N \in {\mathbb{Z}^ + } $, we have
\begin{equation*}
{\left\| {w\left( t \right)} \right\|_{\mathcal{C}^ \alpha }} \lesssim {\left\| {{w_0}} \right\|_{\mathcal{C}^ \alpha } } + \int_{{t_0}}^t {{{\left\| {u\left( s  \right)} \right\|}_{\mathcal{C}^ \alpha } }} \, \mathrm{d}s  ,
\end{equation*}
and
\begin{equation*}
{\left[ {w\left( t \right)} \right]_{N + \alpha }} \lesssim  {\left[ {{w_0}} \right]_{N + \alpha }} + \left| {t - {t_0}} \right|{\left[ v \right]_{N + \alpha }}{\left[ {{w_0}} \right]_1}  + \int_{{t_0}}^t {{{\left[ {u\left( s  \right)} \right]}_{N + \alpha }} + \left| {t - s } \right|{{\left[ v \right]}_{N + \alpha }}{{\left[ {u\left( s  \right)} \right]}_1}} \, \mathrm{d}s  .
\end{equation*}
\end{lemma}
\par
\begin{lemma}[\cite{MR3374958}]\label{STN-HCS-A-sjdij}
Assume that $f\left( x \right)$ and $g\left( x \right)$ are in ${\mathcal{C}^{N + \alpha}}\left( {{\mathbb{T}^3}} \right)$ for $N \in {\mathbb{Z}^ + } \cup \left\{ 0 \right\}$ and $\alpha \in \left[ {0,1} \right)$. If $f\left( x \right)g\left( x \right)$ is in ${\mathcal{C}^{N + \alpha}}\left( {{\mathbb{T}^3}} \right)$, we then have
\begin{equation*}
{\left\| {fg} \right\|_{\mathcal{C}^{N + \alpha }}} \lesssim {\left\| f \right\|_{\mathcal{C}^{N + \alpha }}}{\left\| g \right\|_{\mathcal{C}^0}} + {\left\| f \right\|_{\mathcal{C}^0}}{\left\| g \right\|_{\mathcal{C}^{N + \alpha }}}  .
\end{equation*}
\end{lemma}
\par
\begin{lemma}[\cite{Zbl1480.35317}]\label{INSKDINEQU-afski-a-1}
Let $\Psi :{\mathbb{R}^3} \to \mathbb{R}$ and $u:{\mathbb{T}^3} \to {\mathbb{R}^3}$ be two smooth functions. Then
\begin{equation*}
{\left\| {\Psi \left( u \right)} \right\|_{{\mathcal{C}^N}}} \lesssim  {{\left\| \Psi  \right\|}_{{\mathcal{C}^1}}}{{\left\| u \right\|}_{{\mathcal{C}^N}}} + {{\left\| \Psi  \right\|}_{{\mathcal{C}^N}}}\left\| u \right\|_{{\mathcal{C}^0}}^{N - 1}{{\left\| u \right\|}_{{\mathcal{C}^N}}},\ \ \forall  \,  N \in {\mathbb{Z}^+} ,
\end{equation*}
where the relevant constant just depends on $N$.
\end{lemma}
\par
\begin{lemma}[\cite{arXiv:2305.18105}]\label{TSsdfP-sE-L-A-1}
Let $a \in \mathcal{C}^{N + \alpha}$ be smooth vector field. Then, for any smooth function $b \in \mathcal{C}^{N + \alpha}$, we have
\begin{equation*}
{\left\| {\left[ {a \cdot \nabla ,\nabla } \right]b} \right\|_{{\mathcal{C}^{N + \alpha }}}} \lesssim {\left\| a \right\|_{{\mathcal{C}^{N + 1 + \alpha }}}}{\left\| b \right\|_{{\mathcal{C}^1}}} + {\left\| a \right\|_{{\mathcal{C}^1}}}{\left\| b \right\|_{{\mathcal{C}^{N + 1 + \alpha }}}} ,
\end{equation*}
where $\left[ {a \cdot \nabla ,\nabla } \right]b \triangleq a \cdot \nabla \left( {\nabla b} \right) - \nabla \left( {a \cdot \nabla b} \right)$. Furthermore, let $\mathfrak{T}$ be a Calder\'{o}n-Zygmund operator. Then
\begin{equation*}
{\left\| {\left[ {a \cdot \nabla ,\mathfrak{T}} \right]b} \right\|_{{\mathcal{C}^{N + \alpha }}}} \lesssim {\left\| a \right\|_{{\mathcal{C}^{1 + \alpha }}}}{\left\| b \right\|_{{\mathcal{C}^{N + \alpha }}}} + {\left\| a \right\|_{{\mathcal{C}^\alpha }}}{\left\| b \right\|_{{\mathcal{C}^{N + 1 + \alpha }}}} .
\end{equation*}
\end{lemma}
\par
Choose two parameters $b > 1$ and $c \in \left( {\frac{5}{2},\frac{1}{\gamma }} \right]$ satisfying $2 cb \gamma < 1$. For $a > 1$, define the inductive parameter
\begin{equation*}
{\delta _q} = {a^{ - {b^q}}}  .
\end{equation*}
Let ${ \lambda _q}$ be the positive integer satisfying 
\begin{equation*}
{a^{c{b^{q + 1}}}} \leqslant { \lambda _q} \leqslant 2{a^{c{b^{q + 1}}}} .
\end{equation*}
By the exponential decay of $\delta _q$, we have
\begin{equation}
\sum\limits_{j < q} {\delta _{j + n}^{\frac{1}{2}}{{ \lambda }_j}}  \lesssim \delta _{q + n}^{\frac{1}{2}}{{ \lambda }_q}, \ \ \ \  \sum\limits_{j < q} {\delta _{j + n} {{ \lambda }_j ^{1 - m \alpha }}}  \lesssim \delta _{q + n} \lambda _q^{1 - m \alpha } , \ \ \forall \,   m ,n  \in \left\{ {0,1} \right\}. \label{OS-SBE-C-2}
\end{equation}
\par
Assume that $\mathscr{W}$ is the collection of Wiener processes $\mathcal{W} \left( {t,x} \right)$ satisfying the decomposition
\begin{equation}
{\mathcal{W}}\left( {t,x} \right) = \sum\limits_{q \in \mathbb{Z}^+} { {{b _q}{\mathfrak{B}_q}\left( t \right) \lambda _q^{ - \frac{2}{3}}  \left( {\cos \left( {{\lambda _q}{x_3}} \right),\sin \left( {{\lambda _q}{x_3}} \right),0} \right)}},\ \ {b_q} \in \mathbb{R}  ,  \label{BDHFYSJF-SD}
\end{equation}
where ${{\mathfrak{B} _q}\left( t \right)}$ are mutually independent $\mathbb{R}$-valued Brownian motions on the complete probability space $\left( {\Omega ,\mathcal{F},\mathbb{P}} \right)$. Thus, the Wiener process $\mathcal{W} \left( {t} \right)$ generates temporal white noise $\mathrm{d} {\mathcal{W}}$.
Suppose that
\begin{equation*}
\mathfrak{G} \triangleq \sum\limits_{q \in {\mathbb{Z}^ + }} {b_q^2}   <  + \infty .
\end{equation*}
We then have
\begin{equation*}
\mathbb{E}\left\| {\mathcal{W}\left( t \right)} \right\|_\mathcal{H}^2 = t \cdot \sum\limits_{q \in {\mathbb{Z}^ + }} {b_q^2}\lambda _q^{ - \frac{4}{3} }   = t \cdot \mathrm{Tr}\, Q  < + \infty  , \ \ \forall \ t \in  \left[ {0,T} \right] ,
\end{equation*}
and
\begin{equation}
\mathbb{E}\mathop {\sup }\limits_{t \in \left[ {0,T} \right]} \left\| {\mathcal{W}\left( t \right)} \right\|_{{\mathcal{C}^{\frac{2}{3}} }}^2 \leqslant 4T\mathfrak{G} <  + \infty  . \label{DJIFSTOP-A-1}
\end{equation}
For some $\mathfrak{S} > 1 $, define the stopping time
\begin{equation}\label{IDJIFHISDF-ADS-1}
{\mathfrak{t}_\mathfrak{S}} \triangleq \inf \left\{ {t \geqslant 0  \ \big{|}\  \mathop {\sup }\limits_{\tau  \in \left[ {0,t} \right]} \left\| {{\mathcal{W}}\left( \tau  \right)} \right\|_{{\mathcal{C}^{\frac{2}{3}} }}^2  > \mathfrak{S}  T \mathfrak{G}} \right\}  .
\end{equation}
It follows that $\mathbb{E} \left( { {\mathfrak{t}_\mathfrak{S}} } \right)  <  + \infty $ and $\mathbb{P}\left\{ {{\mathfrak{t}_\mathfrak{S}} > 0} \right\} = 1$, which implies that the stopping time ${\mathfrak{t}_\mathfrak{S}}$ is well-defined.
\par
\begin{remark}
\rm{The Wiener process $\mathcal{W} \left( {t} \right)$ in \eqref{BDHFYSJF-SD} should be regarded as a representative realization of white-in-time Gaussian forcing, which provides an explicit divergence-free spatial decomposition to simplify interactions among the stochastic modes. The essential requirement of our construction is the spatial regularity \eqref{DJIFSTOP-A-1} and the orthogonality $\left( {\mathcal{W} \cdot \nabla } \right)\mathcal{W} = 0$. For instance, Bedrossian, Coti Zelati, Punshon-Smith and Weber \cite{MR3943488} illustrated a similar decomposition of stochastic forcing in studying the stationary martingale solutions and the Kolmogorov 4/5 law for the three-dimensional stochastic Navier--Stokes equations.}
\end{remark}
%%%%%%%%%%%%%%%%%%%%%%%%%%%%%%%%%%%%%%%%%%%%%%
\subsection{Probabilistically strong and analytically weak solution}\label{dfsg}
%%%%%%%%%%%%%%%%%%%%%%%%%%%%%%%%%%%%%%%%%%%%%%
For $q  \in \mathbb{Z}^+$, define the stopping times
\begin{equation*}
{\mathfrak{t}_q} = \inf \left\{ {t > 0  \ \big{|}\  \mathop {\sup }\limits_{\tau  \in \left[ {0,t} \right]} \left| {{\mathfrak{B}_q}{\left( \tau  \right)}} \right| > C \delta _q^{\frac{1}{2}} {\delta _{q + 1}}\lambda _q^{\frac{2}{3} - 2\alpha }} \right\}   .
\end{equation*}
We invoke all the stopping times to define the terminal stopping time
\begin{equation*}
{\mathfrak{t}} \triangleq \mathop {\inf }\limits_{q \in {\mathbb{Z}^ + }} \left\{ {{\mathfrak{t}_q},{\mathfrak{t}_\mathfrak{S}}} \right\},
\end{equation*}
where the truncated stopping time ${\mathfrak{t}_\mathfrak{S}}$ is given by \eqref{IDJIFHISDF-ADS-1}.
\par
\begin{lemma}\label{LEM-ST-A-1-TT}
The terminal stopping time ${\mathfrak{t}}$ admits $\mathbb{P}\left( { {\mathfrak{t}} > 0} \right) = 1$ and $\left[ {0,\mathfrak{t}} \right] = \mathop  \cap \limits_{q \in {\mathbb{Z}^ + }} \left[ {0,{\mathfrak{t}_q} \wedge {\mathfrak{t}_\mathfrak{S}}} \right] $.
\end{lemma}
\par
\noindent{\textbf{Proof}}. 
For any $T_0 > 0$, it follows from Chebyshev inequality and Doob maximal inequality that
\begin{equation*}
\mathbb{P}\left\{ {\mathop {\sup }\limits_{\tau  \in \left[ {0,{T_0}} \right]} \left| {{\mathfrak{B}_q}\left( \tau  \right)} \right| > C  \delta _q^{\frac{1}{2}}  {\delta _{q + 1}}\lambda _q^{\frac{2}{3} - 2\alpha }} \right\} \leqslant \frac{{2{T_0}}}{{{C^2}}}{a^  {\left( {1 + 2b + 4\alpha cb - \frac{4}{3}cb} \right)c{b^q}}} ,
\end{equation*}
which implies that
\begin{equation*}
\sum\limits_{q \in {\mathbb{Z}^ + }} \mathbb{P}\left\{ {\mathop {\sup }\limits_{\tau  \in \left[ {0,{T_0}} \right]} \left| {{\mathfrak{B}_q}\left( \tau  \right)} \right| > C \delta _q^{\frac{1}{2}} {\delta _{q + 1}}\lambda _q^{\frac{2}{3} - 2\alpha }} \right\} < + \infty .
\end{equation*}
Using Borel--Cantelli lemma, there exists a random integer ${N^*}\left( \omega  \right)$ such that
\begin{equation*}
\mathbb{P}\left\{ {{\mathfrak{t}_q} > {T_0}  \ \big{|}\  q \geqslant {N^*}} \right\} = 1 .
\end{equation*}
By the continuity of Brownian paths at $t = 0$, we have
\begin{equation*}
\mathbb{P}\left\{ {{\mathfrak{t}_q} > 0  \ \big{|}\  q < {N^*}} \right\} = 1 .
\end{equation*}
Therefore,
\begin{equation*}
\mathfrak{t} = \mathop {\inf }\limits_{q \in {\mathbb{Z}^ + }} \left\{ {{\mathfrak{t}_q},{\mathfrak{t}_\mathfrak{S}}} \right\} \geqslant \mathop {\min }\limits_{q < {N^*}} \left\{ {{\mathfrak{t}_q},{\mathfrak{t}_\mathfrak{S}},{T_0}} \right\} > 0 , \ \ \mathbb{P}\mbox{-a.s.}
\end{equation*}
\par
Employing Chebyshev inequality and Doob maximal inequality again, we obtain
\begin{equation*}
\mathbb{P}\left\{ {{\mathfrak{t}_{q + 1}} < {\mathfrak{t}_q}} \right\} = {\text{ }}\mathbb{P}\left\{ {\mathop {\sup }\limits_{\tau  \in \left[ {0,{\mathfrak{t}_q}} \right]} \left| {{\mathfrak{B}_{q + 1}}\left( \tau  \right)} \right| > C \delta _q^{\frac{1}{2}} {\delta _{q + 1}}\lambda _q^{\frac{2}{3} - 2\alpha }} \right\} \lesssim {a^{\left( {1 + 2b + 4\alpha cb - \frac{4}{3}cb} \right)c{b^q}}} ,
\end{equation*}
which implies that
\begin{equation*}
\sum\limits_{q \in {\mathbb{Z}^ + }} \mathbb{P}\left\{ {{\mathfrak{t}_{q + 1}} < {\mathfrak{t}_q}} \right\}  < + \infty , \ \ \forall \, N \in {\mathbb{Z}^ + } \cup \left\{ 0 \right\} .
\end{equation*}
Thus, there exists a sufficiently large integer ${N^*}\left( \omega  \right)$ such that
\begin{equation*}
\mathbb{P}\left\{ {{\mathfrak{t}_{q + 1}} \geqslant {\mathfrak{t}_q} \ \big{|}\  q > {N^*}} \right\} = 1 .
\end{equation*}
We then obtain
\begin{equation*}
\mathop {\inf }\limits_{q \in {\mathbb{Z}^ + }} \left\{ {{\mathfrak{t}_q} \wedge {\mathfrak{t}_\mathfrak{S}}} \right\} = \mathop {\min }\limits_{q \leqslant {N^*} + 1} \left\{ {{\mathfrak{t}_q} \wedge {\mathfrak{t}_\mathfrak{S}}} \right\}   , \ \ \mathbb{P}\mbox{-a.s.},
\end{equation*}
and
\begin{equation*}
\mathop  \cap \limits_{j \leqslant {N^* + 1}} \left[ {0,{\mathfrak{t}_j}} \right] = \mathop  \cap \limits_{j \leqslant q} \left[ {0,{\mathfrak{t}_j}} \right] , \ \ \forall \ q > {N^*} , \ \ \mathbb{P}\mbox{-a.s.}
\end{equation*}
Consequently,
\begin{align*}
\left[ {0,\mathfrak{t}} \right] = & \left[ {0,\mathop {\inf }\limits_{q \in {\mathbb{Z}^ + }} \left\{ {{\mathfrak{t}_q} \wedge {\mathfrak{t}_\mathfrak{S}}} \right\}} \right] = \left[ {0,\mathop {\min }\limits_{q{N^*} + 1} \left\{ {{\mathfrak{t}_q} \wedge {\mathfrak{t}_\mathfrak{S}}} \right\}} \right] \\
= & \mathop  \cap \limits_{q \leqslant {N^*} + 1} \left[ {0,{\mathfrak{t}_q} \wedge {\mathfrak{t}_\mathfrak{S}}} \right] = \mathop  \cap \limits_{q \in {\mathbb{Z}^ + }} \left[ {0,{\mathfrak{t}_q} \wedge {\mathfrak{t}_\mathfrak{S}}} \right] , \ \ \mathbb{P}\mbox{-a.s.}
\end{align*}
The proof is complete. \qed
\par
According to Lemma \ref{LEM-ST-A-1-TT}, it is easy to see that
\begin{equation}
\left| {{\mathfrak{B}_q}\left( t \right)} \right| \lesssim \delta _q^{\frac{1}{2}} {\delta _{q + 1}}\lambda _q^{\frac{2}{3} - 2\alpha }  , \ \  \forall \, s, t \in \left[ {0, T \wedge {\mathfrak{t}}} \right],\ \mathbb{P}\mbox{-a.s.} \label{BCBSJFHIG-SFG}
\end{equation}
Furthermore, we can choose two large parameters $C $ and $\mathfrak{S}$ such that $\mathbb{P}\left\{ {\mathfrak{t} > T} \right\} > 1 - \rho$ for any $\rho > 0$. In fact, it follows from Chebyshev inequality that 
\begin{equation*}
\sum\limits_{q \in {\mathbb{Z}^ + }} {\mathbb{P}\left\{ {{\mathfrak{t}_q} \leqslant T} \right\}}  = \sum\limits_{q \in {\mathbb{Z}^ + }} {\mathbb{P}\left\{ {\mathop {\sup }\limits_{\tau  \in \left[ {0,T} \right]} \left| {{\mathfrak{B}_q}\left( \tau  \right)} \right| > C{\delta _{q + 1}}\lambda _q^{\frac{2}{3} - 2\alpha }} \right\}}  \leqslant {C^{ - 2}}T\sum\limits_{q \in {\mathbb{Z}^ + }} {{a^{\left( {2 + 4\alpha c - \frac{4}{3}c} \right)c{b^{q + 1}}}}}  ,
\end{equation*}
and
\begin{equation*}
\mathbb{P}\left\{ {{\mathfrak{t}_\mathfrak{S}} \leqslant T} \right\} = \mathbb{P}\left\{ {\mathop {\sup }\limits_{\tau  \in \left[ {0,T} \right]} \left\| {\mathcal{W}\left( \tau  \right)} \right\|_{{C^\gamma }}^2 > \mathfrak{S}T\mathfrak{G}} \right\} \leqslant C{\mathfrak{S}^{ - 1}} .
\end{equation*}
Therefore,
\begin{equation*}
\mathbb{P}\left\{ {\mathfrak{t} > T} \right\} > 1 - \sum\limits_{q \in {\mathbb{Z}^ + }} {\mathbb{P}\left\{ {\mathfrak{t}_q^* \leqslant T} \right\}} -  \mathbb{P}\left\{ {{\mathfrak{t}_\mathfrak{S}} \leqslant T} \right\} \geqslant 1 - \rho .
\end{equation*}
From this perspective, the Kolmogorov 4/5 law established by Proposition \ref{LHS-3-4-and-5-4-L-L} also holds on the entire time interval $\left[ {0,T} \right]$ with arbitrarily high probability.
\par
\begin{definition}[\cite{Zbl1555.60124}]\label{WMStoBous}
Let $\left( {\Omega ,\mathcal{F},\mathbb{P}} \right)$ be the complete probability space with the right-continuous filtration ${{\left\{ {{\mathcal{F}_t}} \right\}}_{t \in \left[ {0,T} \right]}}$. The stochastic process $u\left( t \right)$ is said to be a probabilistically strong and analytically weak solution to the stochastic incompressible Euler equation \eqref{BouE1}, if it satisfies the following conditions. 
\par
\vspace{1em}
\noindent$(\rmnum{1})$. The process $u\left( t \right)$ is $\mathcal{F}_t$-adapted;
\par
\noindent$(\rmnum{2})$. There exists a stopping time ${\mathfrak{t}}$ such that 
\begin{equation*}
{{u }\left( {t,x} \right)} \in {L^\infty }\left( {\left[ {0, T \wedge {\mathfrak{t}}} \right];{\mathcal{H}}} \right)  , \ \ \mathbb{P}\mbox{-a.s} . ;
\end{equation*}
\par
\noindent$(\rmnum{3})$. For any $\psi \left( x \right) \in {\mathcal{C}^\infty }\left( {{\mathbb{T}^3}} \right) \cap \mathcal{H}$ and $t \in \left[ {0,T \wedge {\mathfrak{t}}} \right]$,
\begin{equation*}
{\left( {u\left( t \right),\psi } \right)_{{\mathbb{T}^3}}} + \int_0^t {{{\left( {\left( {u\left( \tau  \right) \cdot \nabla } \right)u\left( \tau  \right) + \nabla p\left( \tau  \right),\psi } \right)}_{{\mathbb{T}^3}}}} \, \mathrm{d}\tau =  {\left( {u\left( 0 \right),\psi } \right)_{{\mathbb{T}^3}}}  + {\left( {\mathcal{W}\left( t \right),\psi } \right)_{{\mathbb{T}^3}}} , \ \ \mathbb{P}\mbox{-a.s.}
\end{equation*}
\end{definition}
%%%%%%%%%%%%%%%%%%%%%%%%%%%%%%%%%%%%%%%%%%%%%%
%%%%%%%%%%%%%%%%%%%%%%%%%%%%%%%%%%%%%%%%%%%%%%
%%%%%%%%%%%%%%%%%%%%%%%%%%%%%%%%%%%%%%%%%%%%%%
\section{Stochastic Nash iteration scheme}\label{OSG-C-u}
In this section, we employ the stochastic Nash iteration scheme to construct the approximate Euler--Reynolds system and establish the main iterative estimates. The key procedure is to introduce high-frequency perturbations built from classical Beltrami flows together with the designed energy gap, thereby reducing the Reynolds stress and simultaneously preserving the prescribed energy structure. To avoid repetition, we omit the qualifier ``$\mathbb{P}\mbox{-a.s.}$" throughout some proofs whenever the almost-sure interpretation is clear from the context.
%%%%%%%%%%%%%%%%%%%%%%%%%%%%%%%%%%%%%%%%%%%%%%
\subsection{Euler--Reynolds system}\label{s3-1}
Suppose that ${\chi}\left( t \right)$ is the smooth, non-negative and decreasing function with compact support $\left[ { 0, T} \right]$. Take the integer ${\mu _q} = \left\lceil {\delta _{q - 1}^{\frac{1}{4}}\delta _q^{\frac{1}{4}}\lambda _{q - 1}^{\frac{1}{2}}\lambda _q^{\frac{1}{2} + \alpha }} \right\rceil $ with $0 < \alpha < 1$. Making $\alpha$ sufficiently small, it is easy to see that
\begin{equation}
\delta _{q - 1}^{\frac{1}{2}}{{ \lambda }_{q - 1}} \lambda _q^{\alpha } \lesssim \mu_q  \lesssim \delta _q^{\frac{1}{2}}{{ \lambda }_q} . \label{OS-SBE-C-1}
\end{equation}
Denote ${t_l} \triangleq \frac{{lT}}{{\left\lceil {{\mu _q}T} \right\rceil }}$ with the integer $l \in \left[ {0,\left\lceil {{\mu _q}T} \right\rceil } \right] $. Define another cutoff function ${\chi _l}\left( t \right)$ with the compact support $\left[ {{t_l},{t_l} + \frac{5}{4}\frac{T}{{\left\lceil {{\mu _q}T} \right\rceil }}} \right] \cap \left[ {0,T} \right]$ such that
\begin{equation*}
\sum\limits_{{l \in  \left[ {0,\left\lceil {{\mu _q}T} \right\rceil } \right]}} {{\chi ^2_l}\left( { t } \right)}  = 1 , \ \ \ \ \  {\left\| {\partial _t^m{\chi _l}} \right\|_{{L^\infty }\left( {[0,T]} \right)}} \leqslant {C}\mu _q^{ m}, \ \ m \in \mathbb{Z}^+  .
\end{equation*}
Denote ${{\rm T}_l} \triangleq \left[ {{t_l},{t_l} + \frac{5}{4}\frac{T}{{\left\lceil {{\mu _q}T} \right\rceil }}} \right] \cap \left[ {0,T \wedge \mathfrak{t}} \right]$. Define the operator $\mathcal{R}$ on ${{C}\left( {{\mathbb{T}^3}} \right)}$ by
\begin{equation*}
\mathcal{R} v \left( {x} \right) \triangleq  \frac{1}{4}\left( {\nabla \Pi u \left( {x} \right) + {{\left( {\nabla \Pi u \left( {x} \right)} \right)}^{\top}}} \right) - \frac{1}{2}\left( {\mathrm{div}\, u \left( {x} \right)} \right) {I}  + \frac{3}{4}\left( {\nabla u \left( {x} \right) + {{\left( {\nabla u \left( {x} \right)} \right)}^{\top}}} \right) ,
\end{equation*}
where the function $u \left( {x} \right)$ admits $\Delta u\left( x \right) = v\left( x \right) - \int_{{\mathbb{T}^3}} {v\left( x \right)} \, \mathrm{d}x$ and $\int_{{\mathbb{T}^3}} {u\left( x \right)} \, \mathrm{d}x = \mathbf{0}$, the superscript $\top$ means the transposition of vectors, and $I$ denotes the $\left( {3 \times 3} \right)$-order identity matrix. Buckmaster \cite{MR3302631} demonstrated that $\mathrm{div} \mathcal{R} v = v$. 
\par
Define the initial auxiliary function
\begin{equation*}
{v_1}\left( {t,x} \right) = \lambda _1^{ - \frac{1}{3} + \alpha }\chi \left( t \right)\left( {\cos \left( {{\lambda _1}{x_3}} \right),\sin \left( {{\lambda _1}{x_3}} \right),0} \right) , 
\end{equation*}
and the initial velocity
\begin{equation*}
{u_1}\left( {t,x} \right) = {v_1}\left( {t,x} \right) + {\mathcal{W}_1}\left( {t,x} \right) , 
\end{equation*}
where the Wiener process ${\mathcal{W}_1}$ is given by \eqref{BDHFYSJF-SD} with $q = 1$. Then, we have $\left( {{u_1} \cdot \nabla } \right){u_1}  =  0$. Define the initial Reynolds stresses
\begin{align*}
\mathring{R}_1\left( {t,x} \right) = & \lambda _1^{ - \frac{1}{3} + \alpha }{\partial _t}\chi \left( t \right)\mathcal{R}\left( {\cos \left( {{\lambda _1}{x_3}} \right),\sin \left( {{\lambda _1}{x_3}} \right), 0} \right)  \\
= & \lambda _1^{ - \frac{4}{3} + \alpha }{\partial _t}\chi \left( t \right)\left( {\begin{array}{*{20}{c}}
  0&0&{\sin \left( {{\lambda _1}{x_3}} \right)} \\ 
  0&0&{ - \cos \left( {{\lambda _1}{x_3}} \right)} \\ 
  {\sin \left( {{\lambda _1}{x_3}} \right)}&{ - \cos \left( {{\lambda _1}{x_3}} \right)}&{ 0} 
\end{array}} \right) ,
\end{align*}
which implies that ${\partial _t}{v_1} =  \mathrm{div}\, \mathring{R}_1$. Taking the initial pressure ${p_1}\left( {t,x} \right) = 0$, we construct the initial Euler--Reynolds system
\begin{equation*}
\left\{ {\begin{array}{*{20}{l}}
{\partial _t}{v_1} + \left( {{u_1} \cdot \nabla } \right){u_1} + \nabla {p_1} =  \mathrm{div}\, \mathring{R}_1, \\ 
{u_1} = {v_1} + {\mathcal{W}_1}, \\ 
\mathrm{div}\, {u_1} = 0, 
\end{array}} \right.    \ \  t \in \left[ {0, T \wedge \mathfrak{t}} \right] ,\  \mathbb{P}\mbox{-a.s.}
\end{equation*}
\par
Choose the mollifier scales ${\kappa _{q - 1}} = \delta _q^{ - \frac{1}{8}}\delta _{q - 1}^{\frac{1}{8}}\lambda _{q - 1}^{ - \frac{1}{4}}\lambda _q^{ - \frac{3}{4}}$ and ${{\hat \kappa }_{q - 1}} = {\kappa _{q - 1}}\delta _{q}^{\frac{1}{2}}\lambda _q^{ - \varpi }$, where $\varpi \in \left[ {\alpha ,1} \right)$. Direct computation shows that 
\begin{equation}
\lambda _q^{ \alpha } \lesssim {\kappa _{q - 1}} {{ \lambda }_q} , \ \ \ \ \ \ \ {{\hat \kappa }_{q - 1}}\delta _q^{\frac{1}{2}}{\lambda _q} \lesssim {\delta _{q + 1}} . \label{SJKFHASJJFN}
\end{equation}
Denote the convolutions $\bar u \triangleq u*{\varsigma _{{\kappa _{q - 1}}}}$ and $\hat u \triangleq u*{\varsigma _{{{\hat \kappa }_{q - 1}}}}$, where ${\varsigma _{\left(  \cdot  \right)}}$ is the symmetric spatial mollifier. From the definitions of inductive parameters $\delta _q$ and $\lambda _ q$ in Section \ref{dfsg}, we have $\lambda _1^{ - \frac{1}{3} + \alpha } \lesssim \delta _1^{\frac{1}{2}}$ and $\delta _2^{\frac{1}{2}}\lambda _1^{\alpha  - 1} +  \lambda _1^{ - \frac{4}{3} + 2\alpha } \lesssim {\delta _2}$ for some parameters $0 < \alpha \ll 1$. Thus, it is easy to verify that components of the initial stochastic Euler--Reynolds system admit prior estimates as follows.
\par
\begin{lemma}\label{IDUNHU-A-1}
For any $t \in \left[ {0, T \wedge \mathfrak{t}} \right]$ and $N \in \mathbb{Z}^+ \cup \left\{ 0 \right\}$, we have
\begin{equation*}
{\left\| {{v_1}} \right\|_{{\mathcal{C}^N}}} \lesssim \delta _1^{\frac{1}{2}}\lambda _1^N , \ \ \ \ {\left\| {{p_1}} \right\|_{{\mathcal{C}^N}}} \lesssim {\delta _1}\lambda _q^N , \ \ \ \ \mathbb{E}\left\| {\mathcal{W}_1} \right\|_{{\mathcal{C}^N}}^2  \lesssim \lambda _1^{2\left( {N - \frac{2}{3} } \right)} ,
\end{equation*}
\begin{equation*}
{\left\| {{u_1}} \right\|_{{\mathcal{C}^N}}} \lesssim \delta _1^{\frac{1}{2}}\lambda _1^N  , \ \ \mathbb{P}\mbox{-a.s.} 
\end{equation*}
Furthermore, there exists a real number $\alpha  \in \left( {0,1} \right)$ such that
\begin{equation*}
{\left\| {\mathring{R}_1} \right\|_{{\mathcal{C}^0}}} + \lambda _1^{ - 1} {\left\| {\mathring{R}_1} \right\|_{{\mathcal{C}^1}}} + \lambda _1^{ - 2} {\left\| {\mathring{R}_1} \right\|_{{\mathcal{C}^2}}}  \lesssim {\delta _2}\lambda _1^{ - \alpha } , \ \ \forall \, t \in \left[ {0, T} \right] . 
\end{equation*}
and
\begin{equation*}
{\left\| {{\mathfrak{I}_{t,l,1}}\mathring{R}_1} \right\|_{{\mathcal{C}^0}}}  \lesssim  \mu _{2}^{ - 1} \delta _1^{\frac{1}{2}}{\delta _2}\lambda _1^{1 - \alpha }  , \ \  \mathbb{P}\mbox{-a.s.} ,
\end{equation*}
where $\mathfrak{I}_{t,l,q} \left(  \cdot  \right) \triangleq \int_{{t_l}}^t  {\mathrm{d}\left(  \cdot  \right)}  + \int_{{t_l}}^t  {{\bar u_q} \cdot \nabla \left(  \cdot  \right)\mathrm{d}t}  $, $ t \in {{\rm T}_l} $, is the material integration corresponding to ${\bar u_q}$.
\end{lemma}
\par
\par
Without loss of generality, assume that the $\left( {q - 1} \right)$-step Euler--Reynolds system takes the form
\begin{equation}
\left\{ {\begin{array}{*{20}{l}}
{\partial _t} {v_{q - 1}} + \left( {{u_{q - 1}} \cdot \nabla } \right){u_{q - 1}} + \nabla {p_{q - 1}}  = \mathrm{div}\, \mathring{R}_{q - 1} , \\ 
{u_{q - 1}} = {v_{q - 1}} + {\mathcal{W}_{q - 1}}, \\ 
\mathrm{div}\, {u_{q - 1}} = 0, 
\end{array}} \right.  \ \  t \in \left[ {0, T \wedge \mathfrak{t}} \right] ,\  \mathbb{P}\mbox{-a.s.} \label{Xin-h-s-2-TT-u-q-1-a}
\end{equation}
For any $t \in \left[ {0, T \wedge \mathfrak{t}} \right]$, the components in \eqref{Xin-h-s-2-TT-u-q-1-a} satisfy the $\left( {q - 1} \right)$-step inductive estimates,
\begin{equation}
{\left\| {{u_{q - 1}}} \right\|_{{\mathcal{C}^0}}} + {\left\| {{v_{q - 1}}} \right\|_{{\mathcal{C}^0}}} + {\left\| {{p_{q - 1}}} \right\|_{{\mathcal{C}^0}}} \leqslant C , \ \ \mathbb{P}\mbox{-a.s.} ,  \label{IJSUHD-A-1-KDJI}
\end{equation}
\begin{equation}
{\left\| {{u_{q - 1}}} \right\|_{{\mathcal{C}^N}}} + {\left\| {{v_{q - 1}}} \right\|_{{\mathcal{C}^N}}}  \lesssim \delta _{q - 1}^{\frac{1}{2}}\lambda _{q - 1}^N , \ \ \ \  {\left\| {{p_{q - 1}}} \right\|_{{\mathcal{C}^N}}} \lesssim \delta _{q - 1}\lambda _{q - 1}^N , \ \ \forall \, N \in \mathbb{Z}^+ , \ \mathbb{P}\mbox{-a.s.} ,   \label{IJSUHD-A-3}
\end{equation}
\begin{equation}
\mathbb{E} \left\| {{\mathcal{W}_{q - 1}}  -   {\mathcal{W}_{q - 2}} } \right\|_{{\mathcal{C}^{N - 1}}}^2 \lesssim \lambda _{q - 1}^{2\left( {N - \frac{5}{3} } \right)} , \ \ \forall \, N \in \mathbb{Z}^+  ,   \label{IJSUHD-A-2}
\end{equation}
Moreover, there exist a real number $\alpha  \in \left( {0,1} \right)$ and a random constant $\mathbb{P}\left\{ {C\left( \omega  \right) <  + \infty } \right\} = 1$ such that
\begin{equation}
{\left\| {\mathring{R}_{q - 1}} \right\|_{{\mathcal{C}^0}}}  + \lambda _{q - 1}^{ - 1} {\left\| {\mathring{R}_{q - 1}} \right\|_{{\mathcal{C}^1}}}  + \lambda _{q - 2}^{ - 1} {\left\| {\mathring{R}_{q - 1}} \right\|_{{\mathcal{C}^2}}} \lesssim C\left( \omega  \right) {\delta _q}\lambda _{q - 1}^{ - \alpha } , \ \ \forall \, t \in \left[ {0, T \wedge \mathfrak{t}} \right], \ \mathbb{P}\mbox{-a.s.} ,  \label{IJSUHD-A-4}
\end{equation}
\begin{equation}
{\left\| {{\mathfrak{I}_{t,l,{q - 1}}}\mathring{R}_{q - 1}} \right\|_{{\mathcal{C}^0}}} \lesssim C\left( \omega  \right) \mu _{q}^{ - 1}  \delta _{q - 1}^{\frac{1}{2}}{\delta _q}\lambda _{q - 1}^{1 - \alpha }  , \ \ \mathbb{P}\mbox{-a.s.}    \label{IJSUHD-A-5}
\end{equation}
Throughout the remainder of this paper, we will omit the random constant $C\left( \omega  \right)$ in some estimates for notational simplicity whenever no ambiguity arises.
%%%%%%%%%%%%%%%%%%%%%%%%%%%%%%%%%%%%%%%%%%%%%%
\subsection{Iterative framework} We use the symbols ``$\otimes$" and ``$\mathtt{i}$" to denote the tensor product and the imaginary unit, respectively. Assume that the vector ${\mathbf{A}_{\mathbf{k}}} \in \mathbb{R}^3$ satisfies
\begin{equation*}
{\mathbf{A}_{\mathbf{k}}} \cdot {\mathbf{k}} = 0,\ \left| {{\mathbf{A}_{\mathbf{k}}}} \right| = {2^{ - \frac{1}{2}}},\ {\mathbf{A}_{ - {\mathbf{k}}}} = {\mathbf{A}_{\mathbf{k}}},\ \  \forall \, {\mathbf{k}} \in \mathbb{Z}^3,\ \left| {\mathbf{k}} \right| = \bar \lambda ,\ \bar \lambda  \geqslant 1 .
\end{equation*}
Denote ${\mathbf{B}_{\mathbf{k}}} \triangleq {\mathbf{A}_{\mathbf{k}}} + \mathtt{i}\frac{{\mathbf{k}}}{{\left| {\mathbf{k}} \right|}} \times {\mathbf{A}_{\mathbf{k}}} \in {\mathbb{C}^3}$. Let ${{\mathbb{B}_{{r_0}}}\left( I \right)}$ with ${r_0} > 0$ be the metric ball centered at $I$ in the space of symmetric $3 \times 3$ matrices.
\par
\begin{lemma}[\cite{MR3302631}]\label{BFL-1}
For any ${a_{\mathbf{k}}} \in \mathbb{C}$ satisfying ${a_{ - {\mathbf{k}}}} = {{\bar a}_{\mathbf{k}}}$, the vector field $w\left( x \right) = \sum\limits_{\left| {\mathbf{k}} \right| = \bar \lambda } {{a_{\mathbf{k}}} {\mathbf{B}_{\mathbf{k}}}\exp \left( {\mathtt{i}{\mathbf{k}} \cdot x} \right)} $ is real-valued and divergence-free. Furthermore, we have
\begin{equation*}
\mathrm{div}\left( {w \left( x \right) \otimes w \left( x \right)} \right) = \frac{1}{2}\nabla \left( {{{\left| w \left( x \right) \right|}^2}} \right),
\end{equation*}
and
\begin{equation*}
\int_{{\mathbb{T}^3}} {w \left( x \right) \otimes w \left( x \right)} \, \mathrm{d}x = \frac{1}{2 \left| {{\mathbb{T}^3}} \right|}\sum\limits_{\left| {\mathbf{k}} \right| = \bar \lambda } {{a_{\mathbf{k}}^2}\left( {I - \frac{{\mathbf{k}}}{{\left| {\mathbf{k}} \right|}} \otimes \frac{{\mathbf{k}}}{{\left| {\mathbf{k}} \right|}}} \right)}.
\end{equation*}
\end{lemma}
\par
\begin{lemma}[\cite{MR3374958}]\label{BFL-2}
For some ${\bar {\lambda}}  > 1$, there exist pairwise disjoint collections ${\Xi  _j} \subset \left\{ {{\mathbf{k}} \in {\mathbb{Z}^3} \mid \left| {\mathbf{k}} \right| = \bar \lambda } \right\}$ and smooth positive functions $\psi _{\mathbf{k}} \in {\mathcal{C}^\infty }\left( {{\mathbb{B}_{{r_0}}}\left( I \right)} \right)$ such that
\begin{equation*}
- {\mathbf{k}} \in {\Xi _j},\ \psi _{{\mathbf{k}}} = \psi _{{ - {\mathbf{k}}}},\ \  \forall \, {\mathbf{k}} \in {\Xi _j},  \ j \in \left\{ {1,2, \cdots ,\mathfrak{N}} \right\}, \ \mathfrak{N} \in \mathbb{Z}^+ ,
\end{equation*}
and
\begin{equation*}
X = \frac{1}{2}\sum\limits_{{\mathbf{k}} \in {\Xi _j}} {{{\left( {\psi _{\mathbf{k}}\left( X \right)} \right)}^2}\left( {I - \frac{{\mathbf{k}}}{{\left| {\mathbf{k}} \right|}} \otimes \frac{{\mathbf{k}}}{{\left| {\mathbf{k}} \right|}}} \right)} ,\ \  \forall \, X \in {{\mathbb{B}_{{r_0}}}\left( I \right)} .
\end{equation*}
\end{lemma}
\par
Suppose that ${\Phi _l}\left( {t,x} \right)$ is the solution of the transport equation
\begin{equation*}
\left\{ {\begin{array}{*{20}{l}}
{\partial _t}{\Phi _l} + \left( {{{\bar u}_{q - 1}} \cdot \nabla } \right){\Phi _l} = 0, & \mbox{on}\ {{\rm T}_l} \times {\mathbb{T}^3},\\
{\Phi _l}\left( {{t_l},x} \right) = x ,
\end{array}} \right. 
\end{equation*}
and ${{\bar R}_l}\left( {t,x} \right)$ is the solution of the transport equation
\begin{equation*}
\left\{ {\begin{array}{*{20}{l}}
{\partial _t}{{\bar R}_l} + \left( {{{\bar u}_{q - 1}} \cdot \nabla } \right){{\bar R}_l} = 0, & \mbox{on}\ {{\rm T}_l} \times {\mathbb{T}^3},\\
{{\bar R}_l}\left( {{t_l},x} \right) =  \bar R_{q - 1} \left( {{{t_l},x} } \right) .
\end{array}} \right.  
\end{equation*}
Let the energy gap read
\begin{align}
{\mathfrak{E}_{q - 1 , l}}  = & \frac{1}{3 \left| {{\mathbb{T}^3}} \right|}\left( {\mathfrak{e}\left( t_l \right) - \left\| {{{\hat u} _{q - 1}}\left( t_l \wedge  \mathfrak{t} \right)} \right\|_{{L^2}}^2 + \left\| {{{ \mathcal{\hat W}}_{q - 1}\left( t_l \wedge  \mathfrak{t} \right)}} \right\|_{{L^2}}^2 } \right) \nonumber \\
& + \frac{1}{3 \left| {{\mathbb{T}^3}} \right|}\left( { 2\int_0^{t_l \wedge  \mathfrak{t}} {\left( {{{\hat v}_{q - 1}},{\text{d}}{\mathcal{\hat W}_{q - 1}}} \right)_{\mathbb{T}^3}}  - \frac{1}{2} \mathfrak{M}{\delta _{q + 1}}} \right) ,  \label{SJKFJdafj-sagfg}
\end{align}
and satisfy 
\begin{equation}
0 < \mathbb{E} \mathfrak{E}_{q - 1 , l}  \sim {\delta _q} , \ \ \forall \, l \in \left[ {0,\left\lceil {{\mu _q}T} \right\rceil } \right] , \label{SJKFJIFJSKJGG-ASD-1}
\end{equation}
where $\mathfrak{M}$ is a positive parameter, and $\mathfrak{e}\left( t \right) \in \mathcal{C}^1 $ is a decreasing function. It is straightforward to verify that \eqref{SJKFJIFJSKJGG-ASD-1} holds for the first-level energy gap. In view of \eqref{SJKFJIFJSKJGG-ASD-1}, we have
\begin{equation*}
{\left\| {\frac{{{{\bar R}_l}}}{ \mathbb{E}{\mathfrak{E}_{q - 1,l}} }} \right\|_{{\mathcal{C}^0}}}  \lesssim \delta _q^{ - 1}\left\| {\bar R_{q - 1}} \right\|_{{\mathcal{C}^0}} \lesssim \lambda _{q - 1}^{ - \alpha }  \ll 1 , \ \ \mathbb{P}\mbox{-a.s.} 
\end{equation*}
\par
Define the amplitude function $a_{{\mathbf{k}}l} \triangleq \sqrt {{\rho _{q - 1,l}}} {\psi _{{\mathbf{k}}}}\left( {I - \frac{{{{\bar R}_l}}}{{\rho _{q - 1,l}}}} \right) $, where ${\rho _{q - 1 ,l}} = \mathbb{E}{\mathfrak{E}_{q - 1 ,l}} $. Construct the Beltrami flow
\begin{equation*}
w_{{\mathbf{k}}l}\left( {t,x} \right) =  {{a_{{\mathbf{k}}l}}\left( {t,x} \right)}  {\mathbf{B}_{\mathbf{k}}}\exp \left( {\mathtt{i}{{\lambda} _q}{\mathbf{k}} \cdot {\Phi _l}\left( {t,x} \right)} \right) .  
\end{equation*}
There exist disjoint collections ${\Xi _{1}}$ and ${\Xi _{2}}$ such that the principal perturbation reads
\begin{equation}
{w_{o,q-1}}\left( {t,x} \right) \triangleq \sum\limits_{l \ \text{odd}, \ {\mathbf{k}} \in {\Xi _1}} {{\chi _l}\left( t \right)w_{{\mathbf{k}}l}\left( {t,x} \right)}  + \sum\limits_{l \ \text{even}, \ {\mathbf{k}} \in {\Xi _2}} {{\chi _l}\left( t \right)w_{{\mathbf{k}}l}\left( {t,x} \right)} .  \label{PP-A-1}
\end{equation}
The corrector of ${w_{o,q-1}}\left( {t,x} \right)$ is
\begin{equation*}
{w_{c,q-1}}\left( {t,x} \right) \triangleq \sum\limits_{\left( {{\mathbf{k}},l} \right)} {\frac{{{\chi _l}\left( t \right)}}{{{{\lambda} _q}}}} \nabla  \times \left( {\mathtt{i}{a_{{\mathbf{k}}l}}\left( {t,x} \right){\phi _{{\mathbf{k}}l }}\left( {t,x} \right)\frac{{{\mathbf{k}} \times {\mathbf{B}_{\mathbf{k}}}}}{{{{\left| {\mathbf{k}} \right|}^2}}}} \right) \exp \left( {\mathtt{i}{{\lambda} _q}{\mathbf{k}} \cdot x} \right) ,
\end{equation*}
where ${\phi _{{\mathbf{k}}l}}\left( {t,x} \right) = \exp \left( {\mathtt{i}{{\lambda} _q}{\mathbf{k}} \cdot \left( {{\Phi _l}\left( {t,x} \right) - x} \right)} \right)$. Then, the random perturbation is defined by
\begin{equation}
w_{q - 1}\left( {t,x} \right) = {w_{o,q-1}}\left( {t,x} \right) + {w_{c,q-1}}\left( {t,x} \right) , \label{P-IE-Q}
\end{equation}
and we have $\mathrm{div}\, \left( {{w_{o,q-1}} + {w_{c,q-1}}} \right) = 0$. 
\par
Denote the stochastic component by
\begin{equation}
\mathcal{W}_q^{loc}\left( {t,x} \right) \triangleq {b_q}{\mathfrak{B}_q}\left( t \right) \lambda _q^{ - \frac{2}{3} }  \left( {\cos \left( {{\lambda _q}{x_3}} \right),\sin \left( {{\lambda _q}{x_3}} \right),0} \right) . \label{P-Ie-AAA-F-2}
\end{equation}
Construct the $q$-step velocities
\begin{equation}
{u_q}  = {u_{q - 1}} + w_{q - 1}  + \mathcal{W}_q^{loc} , \ \ \ \ {v_q} = {v_{q - 1}} + {w_{q - 1}} ,\label{P-Ie}
\end{equation}
and the pressure
\begin{equation}
{p_q} =  {p_{q - 1}}  - \frac{2}{3} {w_{o,q-1}} \cdot {w_{c,q-1}}  - \frac{2}{3}  \left( {{u_{q - 1}} - {\bar u_{q - 1}}} \right) \cdot w_{q-1} - \frac{1}{2}{\left| {{w_{o,q-1}}} \right|^2} - \frac{1}{3}{\left| {{w_{c,q-1}}} \right|^2}  . \label{Ie-SBE-C-1}
\end{equation}
Using the inductive estimates \eqref{IJSUHD-A-1-KDJI}-\eqref{IJSUHD-A-4}, we derive that for any $t \in \left[ {0, T \wedge \mathfrak{t}} \right]$,
\begin{equation}
{\left\| {{{\bar u}_{q - 1}}} \right\|_{{\mathcal{C}^N}}} \lesssim \kappa _{q - 1}^{1 - N}{\left\| {{u_{q - 1}}} \right\|_{{\mathcal{C}^1}}} \lesssim \kappa _{q - 1}^{1 - N}\delta _{q-1}^{\frac{1}{2}}{\lambda _{q-1}} , \ \ \forall \, N \in \mathbb{Z}^+  , \ \mathbb{P}\mbox{-a.s.}, \label{IJUHCJD-D-D-A-1}
\end{equation}
\begin{equation}
{\left\| {{u_{q - 1}} - {{\bar u}_{q - 1}}} \right\|_{{\mathcal{C}^N}}} \lesssim {\kappa _{q - 1}}{\left\| {{u_{q - 1}}} \right\|_{{C^{N + 1}}}}  \lesssim {\kappa _{q - 1}}\delta _{q - 1}^{\frac{1}{2}}\lambda _{q - 1}^{N + 1} , \ \ \forall \, N \in \mathbb{Z}^+ \cup \left\{ 0 \right\}  , \ \mathbb{P}\mbox{-a.s.} , \label{JUIDHI-AS-A}
\end{equation}
and 
\begin{equation}
{\left\| {{\partial _t}\left( {v_{q - 1}^u - \bar v_{q - 1}^u} \right)} \right\|_{{\mathcal{C}^N}}} \lesssim {\kappa _{q - 1}}\delta _{q - 1}^{\frac{1}{2}}\lambda _{q - 1}^{N + 2} , \ \ \forall \, N \in \mathbb{Z}^+ \cup \left\{ 0 \right\}  , \ \mathbb{P}\mbox{-a.s.} \label{IJUHCJD-D-D-A-1-JU-2}
\end{equation}
It follows from \eqref{IJUHCJD-D-D-A-1} and \eqref{JUIDHI-AS-A} that for any $t \in {{\rm T}_l}$,
\begin{equation}
{\left\| {\mathfrak{D}{\Phi _l}} \right\|_{{\mathcal{C}^0}}} + {\left\| {{\phi _{\mathbf{k}l}}} \right\|_{{\mathcal{C}^0}}} \leqslant C , \ \mathbb{P}\mbox{-a.s.} , \label{IJUHCJD-D-D-A-2-dfjgtyu-3e}
\end{equation}
\begin{equation}
{\left\| {\mathfrak{D}{\Phi _l} - I} \right\|_{{\mathcal{C}^0}}} \lesssim \mu _q^{ - 1}\delta _{q - 1}^{\frac{1}{2}}{\lambda _{q - 1}} , \ \mathbb{P}\mbox{-a.s.} , \label{IJUHCJD-D-D-A-2-dfjgtyu-3e-djfjgjas}
\end{equation}
\begin{equation}
{\left\| {\mathfrak{D}{\Phi _l} - I} \right\|_{{\mathcal{C}^N}}} + {\left\| {\mathfrak{D}{\Phi _l}} \right\|_{{\mathcal{C}^N}}} \lesssim {\mu _q ^{ - 1}}\kappa _{q - 1}^{ - N}\delta _{q - 1}^{\frac{1}{2}}{\lambda _{q - 1}} , \ \ \forall \, N \in \mathbb{Z}^+ , \ \mathbb{P}\mbox{-a.s.} , \label{IJUHCJD-D-D-A-2}
\end{equation}
\begin{equation}
{\left\| {a_{{\mathbf{k}}l}} \right\|_{{\mathcal{C}^N}}} \lesssim \delta _q^{\frac{1}{2}}{\lambda _{q - 1}}\kappa _{q - 1}^{1 - N}, \ \ \forall \, N \in \mathbb{Z}^+ , \ \mathbb{P}\mbox{-a.s.} , \label{IJUHCJD-D-D-A-3}
\end{equation}
and
\begin{equation}
{\left\| {{\phi _{\mathbf{k}l}}} \right\|_{{\mathcal{C}^N}}} \lesssim \lambda _q^{N\left( {1 - \varpi } \right)} , \ \ \forall \, N \in \mathbb{Z}^+ , \ \mathbb{P}\mbox{-a.s.} \label{IJUHCJD-D-D-A-4}
\end{equation}
\par
Our construction is inspired by the framework of Buckmaster, Lellis, Isett and Sz\'{e}kelyhidi \cite[Section $2$]{MR3374958}. By virtue of the similar argument, it follows from \eqref{IJUHCJD-D-D-A-1}-\eqref{IJUHCJD-D-D-A-4} that for any $t \in  \left[ {0, T \wedge \mathfrak{t}} \right]$ and $N \in \mathbb{Z}^+ \cup \left\{ 0 \right\} $,
\begin{equation}
{\left\| {{w_{o,q-1}}} \right\|_{{\mathcal{C}^N}}} \lesssim \delta _q^{\frac{1}{2}}\lambda _q^N , \ \mathbb{P}\mbox{-a.s.} , \label{E-OS-SBE-1}
\end{equation}
\begin{equation}
{\left\| {{w_{c,q-1}}} \right\|_{{\mathcal{C}^N}}} \lesssim {\mu _q ^{ - 1}}\delta _{q - 1}^{\frac{1}{2}}\delta _q^{\frac{1}{2}}{{ \lambda }_{q - 1}}\lambda _q^N , \ \mathbb{P}\mbox{-a.s.} , \label{E-OS-SBE-2}
\end{equation}
\begin{equation}
{\left\| {{\mathfrak{D}_{t,q - 1}} {w_{o,q-1}}} \right\|_{{\mathcal{C}^N}}} \lesssim \mu _q \delta _q^{\frac{1}{2}}\lambda _q^N , \ \mathbb{P}\mbox{-a.s.} , \label{E-OS-SBE-3}
\end{equation}
and
\begin{equation}
{\left\| {{\mathfrak{D}_{t,q - 1}} {w_{c,q-1}}} \right\|_{{\mathcal{C}^N}}} \lesssim \delta _{q - 1}^{\frac{1}{2}}\delta _q^{\frac{1}{2}}{{ \lambda }_{q - 1}}\lambda _q^N , \ \mathbb{P}\mbox{-a.s.},  \label{E-OS-SBE-4}
\end{equation}
where ${{ \mathfrak{D}}_{t,q - 1}} = {\partial _t} + {{\bar u}_{q - 1}} \cdot \nabla $ is the material derivative corresponding to ${\bar u_{q - 1}}$.
\par
\begin{lemma}\label{Ie-SBE}
There exists the $q$-step Reynolds stress $\mathring{R}_q\left( {t,x} \right)$ such that the combination $\left( {{u_{q }}, {v_{q }} ,{p_{q }}} \right) $ fulfills the $q$-step Euler--Reynolds system
\begin{equation}
\left\{ {\begin{array}{*{20}{l}}
{\partial _t}{v_q} + \left( {{u_q} \cdot \nabla } \right){u_q} + \nabla {p_q} = \mathrm{div}\, \mathring{R}_q, \\ 
{u_q} = {v_q} + {\mathcal{W}_q}  , \\ 
\mathrm{div}\, {u_q} = 0, 
\end{array}} \right. \ \  t \in \left[ {0, T \wedge \mathfrak{t}} \right], \  \mathbb{P}\mbox{-a.s.} \label{ie-sbe-c}
\end{equation}
\end{lemma}
\par
{\noindent}{\bf{Proof}}.
Construct the $q$-step Reynolds stress
\begin{equation}
\mathring{R}_q\left( {t,x} \right) = R_{q - 1}^{Nash}\left( {t,x} \right) + R_{q - 1}^{osc}\left( {t,x} \right) + R_{q - 1}^{com}\left( {t,x} \right) + R_{q - 1}^{rem}\left( {t,x} \right) + R_{q - 1}^{sto}\left( {t,x} \right) , \label{rsyduf-a-1}
\end{equation}
where the Nash error is
\begin{equation*}
R_{q - 1}^{Nash} = \mathcal{R}  {\partial _t}{w_{q - 1}}  + \mathcal{R} \mathrm{div}\left( {{{\bar u}_{q - 1}} \otimes {w_{q - 1}} + {w_{q - 1}} \otimes {{\bar u}_{q - 1}}} \right),
\end{equation*}
the oscillation error is
\begin{equation*}
R_{q - 1}^{osc} =  \mathcal{R} \mathrm{div}\left( {{w_{o,q - 1}} \otimes {w_{o,q - 1}} - \sum\limits_{{l \in  \left[ {0,\left\lceil {{\mu _q}T} \right\rceil } \right]}} {\chi _l^2 \left( {{\rho _{q - 1}}I - {{\bar R}_l}} \right)} - \frac{1}{2}{\left| {{w_{o,q - 1}}} \right|^2}I } \right) ,
\end{equation*}
the compensative corrector is
\begin{equation*}
R_{q - 1}^{com} = \mathring{R}_{q - 1} - \bar R_{q - 1} + \sum\limits_{{l \in  \left[ {0,\left\lceil {{\mu _q}T} \right\rceil } \right]}} {\chi _l^2\left( {\bar R_{q - 1} -  {{\bar R}_l}} \right)}  ,
\end{equation*}
the stochastic error is
\begin{equation*}
R_{q - 1}^{sto} =  \mathcal{R}{\text{div}}\left( { {w_{q - 1}} \otimes \mathcal{W}_q^{loc} +   \mathcal{W}_q^{loc} \otimes {w_{q - 1}} + {v_{q - 1}} \otimes \mathcal{W}_q^{loc} + \mathcal{W}_q^{loc} \otimes {v_{q - 1}}} \right) ,
\end{equation*}
and the remaining term is
\begin{align*}
R_{q - 1}^{rem} = & {w_{q - 1}} \otimes \left( {{u_{q - 1}} - {{\bar u}_{q - 1}}} \right) + \left( {{u_{q - 1}} - {{\bar u}_{q - 1}}} \right) \otimes {w_{q - 1}} - \frac{2}{3}\left( {{u_{q - 1}} - {{\bar u}_{q - 1}}} \right) \cdot {w_{q - 1}}I \\
& + {w_{c,q - 1}} \otimes {w_{q - 1}} + {w_{o,q - 1}} \otimes {w_{c,q - 1}} - \frac{1}{3}\left( {{{\left| {{w_{c,q - 1}}} \right|}^2} + 2{w_{o,q - 1}} \cdot {w_{c,q - 1}}} \right)I .
\end{align*}
\par
Owing to $\sum\limits_{{l \in  \left[ {0,\left\lceil {{\mu _q}T} \right\rceil } \right]}} {\chi _l^2}  = 1$, ${\text{div}} \sum\limits_{l \in \left[ {0,{\mu _q}} \right]} {\chi _l^2{{\rho _{q - 1}}}\left( t \right)I}  = \mathbf{0}$ and ${\text{div}}\left( {\mathcal{W}_q \otimes \mathcal{W}_q^{loc} } \right) = \mathbf{0}$, we derive that
\begin{equation*}
\mathrm{d}{v_q} + \mathrm{div}\, \left( {{u_q} \otimes {u_q}} \right) \mathrm{d}t + \nabla {p_q}\mathrm{d}t  - \mathrm{div} \, \mathring{R}_q \mathrm{d}t = 0 , \ \ \mathbb{P}\mbox{-a.s.} 
\end{equation*}
Since the perturbation $w_{q - 1} \left( {t,x} \right) $ is divergence-free, we have
\begin{equation*}
\mathrm{div}\, {u_q} = \mathrm{div}\, {u_{q - 1}} + \mathrm{div}\, {w_{q - 1}} + \mathrm{div}\, \mathcal{W}_q^{loc} = 0 ,
\end{equation*}
which implies that $\mathrm{div}\, \left( {{u_q} \otimes {u_q}} \right) = \left( {{u_q} \cdot \nabla } \right){u_q}$. The proof is complete. \qed
%%%%%%%%%%%%%%%%%%%%%%%%%%%%%%%%%%%%%%%%%%%%%%%%%%%%%%%%%%%%%%%%%%%%%%%
\subsection{Inductive estimates}\label{S-3-3e}
\begin{proposition}
For any $t \in \left[ {0, T \wedge \mathfrak{t}} \right]$ and $N \in \mathbb{Z}^+ $, the $q$-step Euler--Reynolds system \eqref{ie-sbe-c} possesses inductive estimates
\begin{equation}
{\left\| {{v_q}} \right\|_{{\mathcal{C}^0}}}  + {\left\| {{u_q}} \right\|_{{\mathcal{C}^0}}}  + {\left\| {{p_q}} \right\|_{{\mathcal{C}^0}}}  \leqslant  C , \ \ \mathbb{P}\mbox{-a.s.} ,  \label{IJSUHD-A-B-6}
\end{equation}
\begin{equation}
{\left\| {{u_q}} \right\|_{{\mathcal{C}^N}}}  + {\left\| {{v_q}} \right\|_{{\mathcal{C}^N}}}  \lesssim \delta _q^{\frac{1}{2}}\lambda _q^N  , \ \ \mathbb{P}\mbox{-a.s.} ,   \label{IJSUHD-A-B-3}
\end{equation}
\begin{equation}
{\left\| {{p_q}} \right\|_{{\mathcal{C}^N}}} \lesssim \delta _q \lambda _q^N  , \ \ \mathbb{P}\mbox{-a.s.} ,   \label{IJSUHD-A-B-3dfg}
\end{equation}
and
\begin{equation}
\mathbb{E} \left\| {{\mathcal{W}_q}  -   {\mathcal{W}_{q - 1}} } \right\|_{{\mathcal{C}^{N-1}}}^2 \lesssim \lambda _q^{2\left( {N - \frac{5}{3}  } \right)} .  \label{IJSUHD-A-B-2}
\end{equation}
\end{proposition}
\par
{\noindent}{\bf{Proof}}. 
According to \eqref{P-Ie-AAA-F-2} and \eqref{P-Ie}, it is easy to check that \eqref{IJSUHD-A-B-6} and \eqref{IJSUHD-A-B-2} hold. For $N \geqslant 1$, it follows from \eqref{OS-SBE-C-2}, \eqref{BCBSJFHIG-SFG}, \eqref{IJSUHD-A-3}, \eqref{E-OS-SBE-1} and \eqref{E-OS-SBE-2} that
\begin{equation*}
{\left\| {{v_q}} \right\|_{{\mathcal{C}^N}}} \leqslant {\left\| {{v_{q - 1}}} \right\|_{{\mathcal{C}^N}}} + {\left\| {{w_{q - 1}}} \right\|_{{\mathcal{C}^N}}}  \lesssim \delta _q^{\frac{1}{2}}\lambda _q^N  ,
\end{equation*}
and
\begin{equation*}
{\left\| {{u_q}} \right\|_{{\mathcal{C}^N}}}  \leqslant  {\left\| {{u_{q - 1}}} \right\|_{{\mathcal{C}^N}}} + {\left\| {{w_{q - 1}}} \right\|_{{\mathcal{C}^N}}} + {\left\| {\mathcal{W}_q^{loc}} \right\|_{{\mathcal{C}^N}}} \lesssim  \delta _q^{\frac{1}{2}}\lambda _q^N .
\end{equation*}
Employing \eqref{OS-SBE-C-2}, \eqref{JUIDHI-AS-A}, \eqref{Ie-SBE-C-1} and Lemma \ref{STN-HCS-A-sjdij}, we obtain
\begin{align*}
\left\| {{p_q}} \right\|_{{\mathcal{C}^N}} \leqslant & \left\| {{p_{q - 1}}} \right\|_{{\mathcal{C}^N}} + {\left\| {{w_{q - 1}}} \right\|_{{\mathcal{C}^N}}}{\left\| {{w_{q - 1}}} \right\|_{{\mathcal{C}^0}}} \\
& + {\left\| {{v_{q - 1}} - {{\bar v}_{q - 1}}} \right\|_{{\mathcal{C}^0}}}{\left\| {{w_{q - 1}}} \right\|_{{\mathcal{C}^N}}} + {\left\| {{v_{q - 1}} - {{\bar v}_{q - 1}}} \right\|_{{\mathcal{C}^N}}}{\left\| {{w_{q - 1}}} \right\|_{{\mathcal{C}^0}}} \\
\lesssim & \delta _q \lambda _q^N .
\end{align*}
The proof is complete. \qed
\par
\begin{proposition}
For any $t \in \left[ {0, T \wedge \mathfrak{t}} \right]$, there exist a real number $\alpha  \in \left( {0,1} \right)$ and a random constant $\mathbb{P}\left\{ {C\left( \omega  \right) <  + \infty } \right\} = 1$ such that the $q$-step Reynolds stress $\mathring{R}_q\left( {t,x} \right)$ satisfies the inductive estimates
\begin{equation}
{\left\| {\mathring{R}_q} \right\|_{{\mathcal{C}^0}}} +  \lambda _q^{ - 1} {\left\| {\mathring{R}_q} \right\|_{{\mathcal{C}^1}}}  +   \lambda _q^{ - 2} {\left\| {\mathring{R}_q} \right\|_{{\mathcal{C}^2}}} \lesssim C\left( \omega  \right) {\delta _{q+1}}\lambda _q^{ - \alpha } , \ \ \mathbb{P}\mbox{-a.s.} ,  \label{IJSUHD-A-B-4}
\end{equation}
and
\begin{equation}
{\left\| {{\mathfrak{I}_{t,l,{q}}}\mathring{R}_q} \right\|_{{\mathcal{C}^0}}} \lesssim C\left( \omega  \right) \mu _{q + 1}^{ - 1} \delta _q^{\frac{1}{2}}{\delta _{q + 1}}\lambda _q^{1 - \alpha } , \ \ \mathbb{P}\mbox{-a.s.}   \label{IJSUHD-A-B-5}
\end{equation}
\end{proposition}
\par
{\noindent}{\bf{Proof}}.
We estimate each component of the $q$-step Reynolds stress $\mathring{R}_q\left( {t,x} \right)$ defined by \eqref{rsyduf-a-1}, respectively.
\par
\vspace{1em}
{\textbf{Step 1. Estimates of the Nash error $R_{q - 1}^{Nash}$}}. 
\par
Denote $L_{{\mathbf{k}}l} \triangleq a_{{\mathbf{k}}l}{{\mathbf{B}}_{\mathbf{k}}} + \left( {\mathtt{i} \lambda _q^{ - 1}\nabla a_{{\mathbf{k}}l} - a_{{\mathbf{k}}l}\left( {\mathfrak{D}{\Phi _l} - I} \right){\mathbf{k}}} \right) \times {\left| {\mathbf{k}} \right|^{ - 2}}\left( {{\mathbf{k}} \times {{\mathbf{B}}_{\mathbf{k}}}} \right)$. Then, the perturbation $w_{q - 1}$ is written as
\begin{equation*}
{w_{q - 1}} = \sum\limits_{\left( {{\mathbf{k}},l} \right)} {{\chi _l}L_{{\mathbf{k}}l}\exp \left( {\mathtt{i} {\lambda _q}{\mathbf{k}} \cdot {\Phi _l}} \right)} .
\end{equation*}
Thanks to ${\mathfrak{D}_{t,q - 1}}{a_{{\mathbf{k}}l}} = 0$ and ${\mathfrak{D}_{t,q - 1}}\nabla {a_{{\mathbf{k}}l}} = \rho _{q - 1}^{ - \frac{1}{2}} \left[ {{{\bar u}_{q - 1}} \cdot \nabla ,\nabla {{\bar R}_l}} \right] \mathfrak{D}\psi $, it follows from \eqref{IJUHCJD-D-D-A-2-dfjgtyu-3e-djfjgjas}-\eqref{IJUHCJD-D-D-A-3} and Lemma \ref{TSsdfP-sE-L-A-1} that
\begin{equation}
{\left\| {L_{{\mathbf{k}}l}} \right\|_{{\mathcal{C}^0}}}  \lesssim \delta _q^{\frac{1}{2}}, \ \ \ \ {\left\| {L_{{\mathbf{k}}l}} \right\|_{{\mathcal{C}^N}}}  \lesssim \delta _q^{\frac{1}{2}}{\lambda _{q - 1}}\kappa _{q - 1}^{1 - N} ,  \label{IJSUHD-A-B-5-proof-a-1}
\end{equation}
\begin{align}
{\left\| {{\mathfrak{D}_{t,q - 1}}{L_{{\mathbf{k}}l}}} \right\|_{{\mathcal{C}^0}}} \lesssim  & \lambda _q^{ - 1} \delta _{q}^{- \frac{1}{2}} {\left\| {\left[ {{{\bar u}_{q - 1}} \cdot \nabla ,\nabla } \right]{{\bar R}_l}} \right\|_{_{{\mathcal{C}^0}}}} + \delta _q^{\frac{1}{2}}{\left\| {\mathfrak{D}{\Phi _l}} \right\|_{{\mathcal{C}^0}}}{\left\| {{{\bar u}_{q - 1}}} \right\|_{{C^1}}} \nonumber \\
\lesssim & \delta _{q - 1}^{\frac{1}{2}} \delta _{q}^{\frac{1}{2}} \lambda _{q - 1}^{2 - \alpha }\lambda _q^{ - 1} + \delta _{q - 1}^{\frac{1}{2}}\delta _q^{\frac{1}{2}}{\lambda _{q - 1}} ,  \label{IJSUHD-A-B-5-proof-a-2}
\end{align}
and
\begin{align}
\mu _{q + 1} {\left\| {{\mathfrak{I} _{t,l,q}}{\mathfrak{D}_{t,q - 1}}{L_{{\mathbf{k}}l}}} \right\|_{{\mathcal{C}^0}}} \lesssim  \lambda _q^{  - 1} \delta _{q}^{- \frac{1}{2}}  {\left\| {\left[ {{{\bar u}_q} \cdot \nabla ,\left[ {{{\bar u}_{q - 1}} \cdot \nabla ,\nabla } \right]} \right]{{\bar R}_l}} \right\|_{_{{\mathcal{C}^0}}}} \lesssim  \delta _{q - 1}^{\frac{1}{2}}{\delta _q}\lambda _{q - 1}^{2 - \alpha } .  \label{IJSUHD-A-B-5-proof-a-3}
\end{align}
\par
Notice that
\begin{equation*}
R_{q - 1}^{Nash} =  \mathcal{R}\sum\limits_{\left( {{\mathbf{k}},l} \right)} {\left( {L_{{\mathbf{k}}l} {\partial _t}{\chi _l} + {\chi _l}{{ \mathfrak{D}}_{t,q - 1}}L_{{\mathbf{k}}l}  + {\chi _l}L_{{\mathbf{k}}l}  \cdot \nabla {{\bar u}_{q - 1}}} \right){\phi _{{\mathbf{k}}l}}\exp \left( {{\text{i}}{\lambda _q}{\mathbf{k}} \cdot x} \right)} .
\end{equation*}
Using the same argument as Buckmaster, De Lellis, Isett, Sz\'{e}kelyhidi \cite[Proposition $5.1$]{MR3374958}, it follows from \eqref{IJSUHD-A-B-5-proof-a-1}, \eqref{IJSUHD-A-B-5-proof-a-2} and the stationary phase estimate (see \cite[Appendix G]{MR3374958}) that
\begin{equation*}
{\left\| {R_{q - 1}^{Nash}} \right\|_{{\mathcal{C}^0}}} + \lambda _q^{ - 1}{\left\| {R_{q - 1}^{Nash}} \right\|_{{\mathcal{C}^1}}} + \lambda _q^{ - 2}{\left\| {R_{q - 1}^{Nash}} \right\|_{{\mathcal{C}^2}}}  \lesssim {\mu _q}\delta _q^{\frac{1}{2}}\lambda _q^{\alpha  - 1} + \delta _{q - 1}^{\frac{1}{2}}\delta _q^{\frac{1}{2}}{\lambda _{q - 1}}\lambda _q^{\alpha  - 1} . 
\end{equation*}
According to the definition of $\mu _q$, we have
\begin{equation}
{\mu _q}\delta _q^{\frac{1}{2}}\delta _{q + 1}^{ - 1}\lambda _q^{2\alpha  - 1} \sim \delta _{q - 1}^{\frac{1}{4}}\delta _q^{\frac{3}{4}}\delta _{q + 1}^{ - 1}\lambda _{q - 1}^{\frac{1}{2}}\lambda _q^{3\alpha  - \frac{1}{2}} \lesssim {a^{{b^{q - 1}}f\left( {\alpha ,b,c} \right)}}  , \label{JKAJHF-SDGDAF}
\end{equation}
where $f\left( {\alpha ,b,c} \right) = \left[ {\left( {3\alpha  - \frac{1}{2}} \right)c + 1} \right]{b^2} + \left( {\frac{1}{2}c - \frac{3}{4}} \right)b - \frac{1}{4}$. To achieve the inductive estimate \eqref{IJSUHD-A-4}, we choose $c \to {\frac{5}{2}}^+$, $b \to {1^ + }$ and $\alpha \ll 1$ to arrive at ${\partial _b}f\left( {\alpha ,1,c} \right) < 0$ and $f\left( {\alpha ,b,c} \right) < 0$. Therefore,
\begin{equation*}
{\left\| {R_{q - 1}^{Nash}} \right\|_{{\mathcal{C}^0}}} + \lambda _q^{ - 1}{\left\| {R_{q - 1}^{Nash}} \right\|_{{\mathcal{C}^1}}} + \lambda _q^{ - 2}{\left\| {R_{q - 1}^{Nash}} \right\|_{{\mathcal{C}^2}}}  \lesssim {\delta _{q + 1}}\lambda _q^{ - \alpha }  .
\end{equation*}
\par
For the material integration of ${R_{q - 1}^{Nash}}$, we have
\begin{align}
{\mathfrak{I} _{t,l,q}}R_{q - 1}^{Nash} = & \mathcal{R}{\mathfrak{I} _{t,l,q}}{\mathfrak{D}_{t,q - 1}}{w_{q - 1}}  + \mathcal{R}{\mathfrak{I} _{t,l,q}}{\text{div}}\left( {{w_{q - 1}} \otimes {{\bar u}_{q - 1}}} \right) \nonumber \\
& + \int_{{t_l}}^t {\left[ {{{\bar u}_q} \cdot \nabla ,\mathcal{R}} \right]{\mathfrak{D}_{t,q - 1}}{w_{q - 1}}} {\text{d}}\tau  + \int_{{t_l}}^t {\left[ {{{\bar u}_q} \cdot \nabla ,\mathcal{R}} \right]{\text{div}}\left( {{w_{q - 1}} \otimes {{\bar u}_{q - 1}}} \right)} {\text{d}}\tau  .  \label{IJSUHD-A-B-5-prosdafof-a-3s1}
\end{align}
Considering ${\mathfrak{I} _{t,l,q}}{\text{div}}\left( {{w_{q - 1}} \otimes {{\bar u}_{q - 1}}} \right)$, we have
\begin{align}
{\mathfrak{I} _{t,l,q}}{\text{div}}\left( {{w_{q - 1}} \otimes {{\bar u}_{q - 1}}} \right) = & {\mathfrak{I} _{t,l,q}}{\text{div}}\left( {{w_{q - 1}} \otimes {{\bar v}_{q - 1}}} \right) + {\mathfrak{I} _{t,l,q}}{\text{div}}\left( {{w_{q - 1}} \otimes {{\mathcal{\hat W}}_{q - 1}}} \right) \nonumber \\
= & \int_{{t_l}}^t {{w_{q - 1}} \cdot \nabla } {\text{d}}{{\mathcal{\hat W}}_{q - 1}} + \int_{{t_l}}^t {{\mathfrak{D}_{\tau ,q - 1}}{w_{q - 1}} \cdot \nabla {{\bar v}_{q - 1}} + {w_{q - 1}} \cdot {\mathfrak{D}_{\tau ,q - 1}}\nabla {{\bar v}_{q - 1}}} {\text{d}}\tau  \nonumber \\
& + \int_{{t_l}}^t {\left( {{{\bar u}_q} - {{\bar u}_{q - 1}}} \right) \cdot \left( {\nabla {w_{q - 1}}} \right) \cdot \nabla {{\bar v}_{q - 1}}} {\text{d}}\tau  + \int_{{t_l}}^t {{\mathfrak{D}_{\tau ,q - 1}}{w_{q - 1}} \cdot \nabla {{\mathcal{\hat W}}_{q - 1}}} {\text{d}}\tau  \nonumber \\
& + \int_{{t_l}}^t {{w_{q - 1}} \cdot \left( {\left( {{{\bar u}_q} - {{\bar u}_{q - 1}}} \right) \cdot \nabla } \right)\nabla {{\bar v}_{q - 1}}} {\text{d}}\tau  \nonumber \\
&  + \int_{{t_l}}^t {{w_{q - 1}} \cdot \left( {{{\bar u}_q} \cdot \nabla } \right)\nabla {{\mathcal{\hat W}}_{q - 1}}} {\text{d}}\tau .  \label{IJSUHD-A-B-5-prosdafof-a-3}
\end{align}
For the It\^{o} integral $\int_{{t_l}}^t {{w_{q - 1}} \cdot \nabla } {\text{d}}{{{\mathcal{\bar W}}}_{q - 1}}$, we infer that
\begin{equation*}
{\mu _{q + 1}}{\left\| {\mathcal{R}\int_{{t_l}}^t {{w_{q - 1}} \cdot \nabla } {\text{d}}{{ \mathcal{\bar W}}_{q - 1}}} \right\|_{{\mathcal{C}^\alpha }}} \lesssim \lambda _q^{\alpha  - 1}\mu _{q + 1}^{\frac{1}{2} + \alpha }{\left\| {{w_{q - 1}}} \right\|_{{\mathcal{C}^{\alpha}}}}\sum\limits_{j \leqslant q - 1} {\lambda _j^{\frac{1}{3}}{{\left[ {{\mathfrak{B}_j}} \right]}_{\frac{1}{2} - \alpha }}}  .
\end{equation*}
Choose a random constant $C\left( \omega  \right) = {r_q} \mathop {\sup }\limits_{q \in {\mathbb{Z}^ + }} {\left[ {{\mathfrak{B}_q}\left( \tau  \right)} \right]_{\frac{1}{2} - \alpha }}$ with ${r_q} = {a^{ - \left( {\frac{{5}}{3} - \varpi } \right){b^{q - 1}}}}$, $\varpi  \in \left[ {\alpha ,1} \right)$. Using the fact that $\mathbb{E}\left[ {{\mathfrak{B}_q}\left( \tau  \right)} \right]_{\frac{1}{2} - \alpha }^2 <  + \infty $ for any $q \in {{\mathbb{Z}^ + }}$, we have
\begin{equation*}
\mathbb{E}{\left( {C\left( \omega  \right)} \right)^2} \lesssim \sum\limits_{q \in {\mathbb{Z}^ + }} {r_q^2}  <  + \infty ,
\end{equation*}
which implies that
\begin{equation*}
\mathbb{P}\left\{ {C\left( \omega  \right) <  + \infty } \right\} = 1 .
\end{equation*}
Then, we rewrite the estimate of $\int_{{t_l}}^t {{w_{q - 1}} \cdot \nabla } {\text{d}}{{{\mathcal{\bar W}}}_{q - 1}}$ as
\begin{align}
{\mu _{q + 1}}{\left\| {\mathcal{R}\int_{{t_l}}^t {{w_{q - 1}} \cdot \nabla } {\text{d}}{{ \mathcal{\bar W}}_{q - 1}}} \right\|_{{\mathcal{C}^\alpha }}} \lesssim & C\left( \omega  \right)\lambda _q^{2\alpha  - 1}   \mu _{q + 1}^{\frac{1}{2} + \alpha } r_q^{ - 1}  \delta _q^{\frac{1}{2}}\lambda _{q - 1}^{\frac{1}{3}} \nonumber \\
\lesssim & C\left( \omega  \right) \delta _q^{\frac{1}{2}}{\delta _{q + 1}}\lambda _q^{1 - \alpha }  .  \label{IJSUHD-A-B-5-prosdafof-a-4}
\end{align}
\par
We next examine the transport term $\int_{{t_l}}^t {{w_{q - 1}} \cdot {\mathfrak{D}_{\tau ,q - 1}}\nabla {{\bar v}_{q - 1}}} {\text{d}}\tau $ in \eqref{IJSUHD-A-B-5-prosdafof-a-3}. Since
\begin{align*}
{\mathfrak{D}_{\tau ,q - 1}}{{\bar v}_{q - 1}} = & {\text{div}}{{\bar R}_{q - 1}} - \nabla {{\bar p}_{q - 1}} - {{\bar u}_{q - 1}} \cdot \nabla {{\mathcal{\hat W}}_{q - 1}}  \\
& + {\text{div}}\left( {{{\bar u}_{q - 1}} \otimes {{\bar u}_{q - 1}} - \left( {{u_{q - 1}} \otimes {u_{q - 1}}} \right)*{\varsigma _{{{\bar \kappa }_{q - 1}}}}} \right) ,
\end{align*}
\begin{equation*}
{\mathfrak{D}_{\tau ,q - 1}}\nabla {{\bar v}_{q - 1}} = \nabla {\mathfrak{D}_{\tau ,q - 1}}{{\bar v}_{q - 1}} + \left[ {{{\bar u}_{q - 1}} \cdot \nabla ,\nabla } \right]{{\bar v}_{q - 1}} ,
\end{equation*}
and
\begin{equation*}
{\left\| {\left( {{u_{q - 1}} \otimes {u_{q - 1}}} \right)*{\varsigma _{{{\hat \kappa }_{q - 1}}}} - {{\hat u}_{q - 1}} \otimes {{\hat u}_{q - 1}}} \right\|_{{\mathcal{C}^{2 + \alpha }}}} \lesssim \kappa _{q - 1}^2\left( {{{\left\| {{u_{q - 1}}} \right\|}_{{\mathcal{C}^1}}}{{\left\| {{u_{q - 1}}} \right\|}_{{\mathcal{C}^{3 + \alpha }}}} + \left\| {{u_{q - 1}}} \right\|_{{\mathcal{C}^{2 + \alpha }}}^2} \right) ,
\end{equation*}
we infer that
\begin{align*}
{\left\| {{\mathfrak{D}_{\tau ,q - 1}}\nabla {{\bar v}_{q - 1}}} \right\|_{{\mathcal{C}^\alpha }}} \lesssim & {\left\| {{{\bar R}_{q - 1}} + {{\bar p}_{q - 1}}} \right\|_{{\mathcal{C}^{2 + \alpha }}}} + {\left\| {{{\bar u}_{q - 1}} \cdot \nabla {{\mathcal{\hat W}}_{q - 1}}} \right\|_{{\mathcal{C}^{1 + \alpha }}}} + {\left\| {{{\bar u}_{q - 1}}} \right\|_{{\mathcal{C}^{1 + \alpha }}}}{\left\| {{{\bar v}_{q - 1}}} \right\|_{{\mathcal{C}^1}}}  \\
& + {\left\| {{{\bar u}_{q - 1}} \otimes {{\bar u}_{q - 1}} - \left( {{u_{q - 1}} \otimes {u_{q - 1}}} \right)*{\varsigma _{{{\bar \kappa }_{q - 1}}}}} \right\|_{{\mathcal{C}^{2 + \alpha }}}} \\
\lesssim &  \kappa _{q - 1}^2{\delta _{q - 1}}\lambda _{q - 1}^{4 + \alpha } + {\delta _{q - 1}}\lambda _{q - 1}^{2 + \alpha }  + \delta _{q - 1}^{\frac{1}{2}}\delta _q^{\frac{1}{2}}{\delta _{q + 1}}\lambda _{q - 1}^{\frac{4}{3} + \alpha }\lambda _q^{\frac{2}{3} - 2\alpha } . 
\end{align*}
Then, we obtain
\begin{equation}
{\mu _{q + 1}}{\left\| {\mathcal{R}\int_{{t_l}}^t {{w_{q - 1}} \cdot {\mathfrak{D}_{\tau ,q - 1}}\nabla {{\bar v}_{q - 1}}} {\text{d}}\tau } \right\|_{{\mathcal{C}^\alpha }}} \lesssim  \delta _q^{\frac{1}{2}}{\delta _{q + 1}}\lambda _q^{1 - \alpha } .  \label{IJSUsdafof-a-4}
\end{equation}
\par
The remaining terms in \eqref{IJSUHD-A-B-5-prosdafof-a-3} can be estimated immediately by using the existing inductive estimates \eqref{BCBSJFHIG-SFG}, \eqref{IJSUHD-A-1-KDJI}, \eqref{IJSUHD-A-3} and \eqref{E-OS-SBE-1}-\eqref{E-OS-SBE-4}, and all of them are bounded by the scale $\mu _{q + 1}^{ - 1}\delta _q^{\frac{1}{2}}{\delta _{q + 1}}\lambda _q^{1 - \alpha }$. Therefore, combining \eqref{IJSUHD-A-B-5-prosdafof-a-3}, \eqref{IJSUHD-A-B-5-prosdafof-a-4} and \eqref{IJSUsdafof-a-4}, we deduce that
\begin{equation}
{\left\| {\mathcal{R}{\mathfrak{I} _{t,l,q}}{\text{div}}\left( {{w_{q - 1}} \otimes {{\bar u}_{q - 1}}} \right)} \right\|_{{\mathcal{C}^0}}} \lesssim \mu _{q + 1}^{ - 1}\delta _q^{\frac{1}{2}}{\delta _{q + 1}}\lambda _q^{1 - \alpha }  .  \label{IJSUsdasdfof-a-4da}
\end{equation}
\par
It follows from \eqref{IJSUHD-A-B-5-proof-a-1}-\eqref{IJSUHD-A-B-5-proof-a-3} that
\begin{align}
{\mu _{q + 1}}{\left\| {\mathcal{R}{\mathfrak{I} _{t,l,q}}{\mathfrak{D}_{t,q - 1}}{w_{q - 1}}} \right\|_{{\mathcal{C}^\alpha }}} \lesssim & {\mu _q}\lambda _q^{2\alpha  - 1}\left( {{\mu _q}{{\left\| {{L_{{\mathbf{k}}l}}} \right\|}_{{\mathcal{C}^0}}} + {{\left\| {{\mathfrak{D}_{t,q - 1}}{L_{{\mathbf{k}}l}}} \right\|}_{{\mathcal{C}^0}}} + \delta _q^{\frac{1}{2}}{{\left\| {{L_{{\mathbf{k}}l}}} \right\|}_{{\mathcal{C}^1}}}} \right) \nonumber \\
& + \delta _q^{\frac{1}{2}}\lambda _q^{2\alpha }\left( {{\mu _q}{{\left\| {{L_{{\mathbf{k}}l}}} \right\|}_{{\mathcal{C}^0}}} + {{\left\| {{\mathfrak{D}_{t,q - 1}}{L_{{\mathbf{k}}l}}} \right\|}_{{\mathcal{C}^0}}}} \right)  \nonumber \\
& + {\mu _{q + 1}}\lambda _q^{2\alpha  - 1}{\left\| {{\mathfrak{I} _{t,l,q}}{\mathfrak{D}_{t,q - 1}}{L_{{\mathbf{k}}l}}} \right\|_{{\mathcal{C}^0}}} \nonumber \\
\lesssim & \delta _q^{\frac{1}{2}}{\delta _{q + 1}}\lambda _q^{1 - \alpha }  .  \label{IJSUHD-A-B-5-psdflugroof-a-3}
\end{align}
Moreover, using the commutator estimate (see \cite[Appendix H]{MR3374958}), we obtain
\begin{align}
{\mu _{q + 1}} {\left\| {\int_{{t_l}}^t {\left[ {{{\bar u}_{q - 1}} \cdot \nabla ,\mathcal{R}} \right]{\mathfrak{D}_{t,q - 1}}{w_{q - 1}}} {\text{d}}\tau } \right\|_{{\mathcal{C}^\alpha }}} \lesssim & \lambda _q^{\alpha  - 1} {\left\| {{{\bar u}_{q - 1}}} \right\|_{{\mathcal{C}^{1 + \alpha }}}}{\left\| {{\mathfrak{D}_{t,q - 1}}{w_{q - 1}}} \right\|_{{\mathcal{C}^\alpha }}} \nonumber \\
\lesssim &  {\mu _q} \delta _{q - 1}^{\frac{1}{2}}\delta _q^{\frac{1}{2}}\lambda _{q - 1}^{1 + \alpha }\lambda _q^{2\alpha  - 1} \nonumber  \\
\lesssim & \delta _q^{\frac{1}{2}}{\delta _{q + 1}}\lambda _q^{1 - \alpha }  , \label{IJSUsdadfga}
\end{align}
and
\begin{align}
& {\mu _{q + 1}}{\left\| {\mathcal{R}\int_{{t_l}}^t {\left[ {{{\bar u}_q} \cdot \nabla ,\mathcal{R}} \right]{\text{div}}\left( {{w_{q - 1}} \otimes {{\bar u}_{q - 1}}} \right)} {\text{d}}\tau } \right\|_{{\mathcal{C}^\alpha }}} \nonumber \\
\lesssim & \lambda _q^{\alpha  - 1}{\left\| {{{\bar u}_q}} \right\|_{{\mathcal{C}^{1 + \alpha }}}}{\left\| {{{\bar u}_q}} \right\|_{{\mathcal{C}^1}}}{\left\| {{w_{q - 1}}} \right\|_{{\mathcal{C}^\alpha }}} \lesssim {\delta _{q - 1}}\delta _q^{\frac{1}{2}}\lambda _{q - 1}^{2 + \alpha }\lambda _q^{2\alpha  - 1}  \nonumber \\
\lesssim & \delta _q^{\frac{1}{2}}{\delta _{q + 1}}\lambda _q^{1 - \alpha }  . \label{IJSUsdadd-sf-fga}
\end{align}
Therefore, it follows from \eqref{IJSUHD-A-B-5-prosdafof-a-3s1}, \eqref{IJSUHD-A-B-5-psdflugroof-a-3}, \eqref{IJSUHD-A-B-5-prosdafof-a-4}, \eqref{IJSUsdadfga} and \eqref{IJSUsdadd-sf-fga} that
\begin{equation*}
{\left\| {{{\mathfrak{I} _{t,l,q - 1}}}R_{q - 1}^{Nash}} \right\|_{{\mathcal{C}^0}}} \lesssim \mu _{q+1}^{ - 1}  \delta _q^{\frac{1}{2}}{\delta _{q + 1}}\lambda _q^{1 - \alpha } . 
\end{equation*}
\par
\vspace{1em}
{\textbf{Step 2. Estimates of the oscillation error $R_{q - 1}^{osc}$}}. 
\par
Denote ${f_{\mathbf{k}l\mathbf{k}'l'}}\left( {t,x} \right) \triangleq \sum\limits_{\left( {{\mathbf{k}},l} \right)} {{\chi _l}\left( t \right){\chi _{l'}}\left( t \right)  {a_{\mathbf{k}l}}\left( {t,x} \right){a_{\mathbf{k}'l'}}\left( {t,x} \right)  {\phi _{{\mathbf{k}}l}}\left( {t,x} \right){\phi _{{\mathbf{k}}'l'}}\left( {t,x} \right)}$. It follows from \eqref{IJUHCJD-D-D-A-3} and \eqref{IJUHCJD-D-D-A-4} that for some $\varpi \in \left[ {\alpha,1} \right)$,
\begin{equation}
{\left\| {{f_{\mathbf{k}l\mathbf{k}'l'}}} \right\|_{{\mathcal{C}^N}}}  \lesssim {\delta _q}{\lambda _{q - 1}}\kappa _{q - 1}^{1 - N} + {\delta _q}\lambda _q^{N\left( {1 - \varpi } \right)} , \ \ N \geqslant 1 .  \label{AIUJBHTGsdfghS-AB-2}
\end{equation}
Notice that ${w_{o, {q-1}}} \otimes {w_{o, {q-1}}} - \sum\limits_{{l \in  \left[ {0,\left\lceil {{\mu _q}T} \right\rceil } \right]}} {\chi _l^2 \left( {{\rho _{q - 1}} I - {{\bar R}_l}} \right)}  = \sum\limits_{\left( {\mathbf{k},l} \right),\left( {\mathbf{k}',l'} \right),\mathbf{k} + \mathbf{k}' \ne 0} {{\chi _l}{\chi _{l'}}{w_{\mathbf{k}l}} \otimes {w_{\mathbf{k}'l'}}} $ . We then have
\begin{align*}
& \mathrm{div} \left( {{w_{o,q - 1}} \otimes {w_{o,q - 1}} - \sum\limits_{{l \in  \left[ {0,\left\lceil {{\mu _q}T} \right\rceil } \right]}} {\chi _l^2 \left( {{\rho _{q - 1}} I - {{\bar R}_l}} \right)}  - \frac{1}{2}\left| {{w_{o,q - 1}}} \right|^ 2 I} \right) \\
= & \sum\limits_{\left( {\mathbf{k},l} \right),\left( {\mathbf{k}',l'} \right),\mathbf{k} + \mathbf{k}' \ne 0} {\left( {{\mathbf{B}_\mathbf{k}} \otimes {\mathbf{B}_{\mathbf{k}'}} - \frac{1}{2}{\mathbf{B}_\mathbf{k}} \cdot {\mathbf{B}_{\mathbf{k}'}}I} \right) \left( {\nabla   {f_{\mathbf{k}l\mathbf{k}'l'}}} \right) \exp \left( {\mathtt{i}{\lambda _q}\left( {\mathbf{k} + \mathbf{k}'} \right) \cdot x} \right)}  \\
& + \mathtt{i}{\lambda _q}\sum\limits_{\left( {\mathbf{k},l} \right),\left( {\mathbf{k}',l'} \right),\mathbf{k} + \mathbf{k}' \ne 0} {\left( {{\mathbf{B}_\mathbf{k}} \otimes {\mathbf{B}_{\mathbf{k}'}} - \frac{1}{2}{\mathbf{B}_\mathbf{k}} \cdot {\mathbf{B}_{\mathbf{k}'}}I} \right)\left( {\mathbf{k} + \mathbf{k}'} \right){f_{\mathbf{k}l\mathbf{k}'l'}}\exp \left( {\mathtt{i}{\lambda _q}\left( {\mathbf{k} + \mathbf{k}'} \right) \cdot x} \right)} .
\end{align*}
Thanks to $\left( {{\mathbf{B}_\mathbf{k}} \otimes {\mathbf{B}_{\mathbf{k}'}} + {\mathbf{B}_{\mathbf{k}'}} \otimes {\mathbf{B}_\mathbf{k}}} \right)\left( {\mathbf{k} + \mathbf{k}'} \right) = {\mathbf{B}_\mathbf{k}} \cdot {\mathbf{B}_{\mathbf{k}'}}\left( {\mathbf{k} + \mathbf{k}'} \right)$, we infer that
\begin{equation*}
R_{q - 1}^{osc} = \mathcal{R}\sum\limits_{\left( {\mathbf{k},l} \right),\left( {\mathbf{k}',l'} \right),\mathbf{k} + \mathbf{k}' \ne 0} {\left( {{\mathbf{B}_\mathbf{k}} \otimes {\mathbf{B}_{\mathbf{k}'}} - \frac{1}{2}{\mathbf{B}_\mathbf{k}} \cdot {\mathbf{B}_{\mathbf{k}'}}I} \right) \left( {\nabla   {f_{\mathbf{k}l\mathbf{k}'l'}}} \right)  \exp \left( {\mathtt{i}{\lambda _q}\left( {\mathbf{k} + \mathbf{k}'} \right) \cdot x} \right)}   .  
\end{equation*}
\par
Using Lemma \ref{STN-HCS-A-sjdij} and the stationary phase lemma, we deduce that
\begin{equation}
{\left\| {R_{q - 1}^{osc}} \right\|_{{\mathcal{C}^N}}} \lesssim \lambda _q^{ - 1}\mathop {\sup }\limits_{\left( {\mathbf{k},l} \right),\left( {\mathbf{k}',l'} \right)} \left( {{{\left\| {{f_{\mathbf{k}l\mathbf{k}'l'}} } \right\|}_{{C^{N + 1}}}} + \lambda _q^N{{\left\| {{f_{\mathbf{k}l\mathbf{k}'l'}}} \right\|}_{{\mathcal{C}^1}}}} \right) . \label{ddmcjffa}
\end{equation}
It follows from \eqref{SJKFHASJJFN} and \eqref{AIUJBHTGsdfghS-AB-2} that
\begin{align*}
{\left\| {R_{q - 1}^{osc}} \right\|_{{\mathcal{C}^0}}} + \lambda _q^{ - 1}{\left\| {R_{q - 1}^{osc}} \right\|_{{\mathcal{C}^1}}} + \lambda _q^{ - 2}{\left\| {R_{q - 1}^{osc}} \right\|_{{\mathcal{C}^2}}} \lesssim & {\delta _q}\lambda _q^{ - 3}\left( {{\lambda _{q - 1}}\lambda _q^2 + {\lambda _{q - 1}}\kappa _{q - 1}^{ - 1}{\lambda _q} + {\lambda _{q - 1}}\kappa _{q - 1}^{ - 2}} \right) \\
\lesssim & {\delta _q}\left( {\lambda _q^{ - \varpi } + \lambda _q^{ - 2\varpi } + \lambda _q^{ - 3\varpi }} \right).
\end{align*}
Consequently, there exists a real number $\alpha  \in \left( {0,1} \right)$ such that
\begin{equation*}
{\left\| {R_{q - 1}^{osc}} \right\|_{{\mathcal{C}^0}}} + \lambda _q^{ - 1}{\left\| {R_{q - 1}^{osc}} \right\|_{{\mathcal{C}^1}}} + \lambda _q^{ - 2}{\left\| {R_{q - 1}^{osc}} \right\|_{{\mathcal{C}^2}}} \lesssim {\delta _{q + 1}}\lambda _q^{ - \alpha }  .  
\end{equation*}
\par
Applying the material derivative operator to $R_{q - 1}^{osc}$, we have
\begin{equation*}
{\mathfrak{D}_{t,q}}R_{q - 1}^{osc} = \mathcal{R}{\text{div}}{\mathfrak{D}_{t,q - 1}}F + \left[ {{{\bar u}_{q - 1}} \cdot \nabla ,\mathcal{R}{\text{div}}} \right]F + \left( {{{\bar u}_q} - {{\bar u}_{q - 1}}} \right) \cdot \nabla R_{q - 1}^{osc} ,
\end{equation*}
where
\begin{equation*}
F \triangleq {w_{o,q - 1}} \otimes {w_{o,q - 1}} - \sum\limits_{l \in \left[ {0,\left\lceil {{\mu _q}T} \right\rceil } \right]} {\chi _l^2\left( {{\rho _{q - 1}}I - {{\bar R}_l}} \right)}  - \frac{1}{2}{\left| {{w_{o,q - 1}}} \right|^2}I  .
\end{equation*}
Notice that Lemma \ref{TSP-E-L-A-1} implies
\begin{equation}
{\left\| {{{\bar R}_l}} \right\|_{{\mathcal{C}^0}}} \lesssim  {\left\| {\bar R_{q - 1}} \right\|_{{\mathcal{C}^0}}} , \ \ \ \ {\left\| {{{\bar R}_l}} \right\|_{{\mathcal{C}^1}}} \lesssim \left( {1 + \mu _q^{ - 1}{{\left\| {{{\bar u}_{q - 1}}} \right\|}_{{\mathcal{C}^1}}}} \right)  {\left\| {\bar R_{q - 1}} \right\|_{{\mathcal{C}^1}}} .  \label{SJFHJKSA-DAFGH}
\end{equation}
Then, it follows from \eqref{IJSUHD-A-4}, \eqref{E-OS-SBE-1}, \eqref{E-OS-SBE-3} and \eqref{SJFHJKSA-DAFGH} that
\begin{equation*}
{\left\| F \right\|_{{\mathcal{C}^0}}} + \lambda _q^{ - 1}{\left\| F \right\|_{{\mathcal{C}^1}}} + \mu _q^{ - 1}{\left\| {{\mathfrak{D}_{t,q - 1}}F} \right\|_{{\mathcal{C}^0}}} + \lambda _q^{ - 1}\mu _q^{ - 1}{\left\| {{\mathfrak{D}_{t,q - 1}}F} \right\|_{{\mathcal{C}^1}}} \lesssim {\delta _q}
\end{equation*}
Employing Schauder estimate (see \cite[Appendix F]{MR3374958}) and Lemma \ref{TSsdfP-sE-L-A-1}, we deduce that
\begin{equation*}
{\left\| {\mathcal{R}{\text{div}}{\mathfrak{D}_{t,q - 1}}F} \right\|_{{\mathcal{C}^{\alpha}}}} \lesssim \left\| {{\mathfrak{D}_{t,q - 1}}F} \right\|_{{\mathcal{C}^0}}^{1 - \alpha }\left\| {{\mathfrak{D}_{t,q - 1}}F} \right\|_{{\mathcal{C}^1}}^\alpha  \lesssim {\mu _q}{\delta _q}\lambda _q^\alpha  \lesssim \delta _q^{\frac{1}{2}}{\delta _{q + 1}}\lambda _q^{1 - \alpha } ,
\end{equation*}
and
\begin{equation*}
{\left\| {\left[ {{{\bar u}_{q - 1}} \cdot \nabla ,\mathcal{R}{\text{div}}} \right]F} \right\|_{{\mathcal{C}^{\alpha}}}} \lesssim {\left\| {{{\bar u}_{q - 1}}} \right\|_{{\mathcal{C}^{1 + \alpha}}}}\left\| F \right\|_{{\mathcal{C}^0}}^{1 - \alpha }\left\| F \right\|_{{\mathcal{C}^1}}^\alpha  \lesssim \delta _{q - 1}^{\frac{1}{2}}{\delta _q}{\lambda _{q - 1}^{1 + \alpha}}\lambda _q^\alpha  \lesssim \delta _q^{\frac{1}{2}}{\delta _{q + 1}}\lambda _q^{1 - \alpha } .
\end{equation*}
By the previous estimates \eqref{JUIDHI-AS-A} and \eqref{ddmcjffa}, we have
\begin{align*}
& {\left\| {\left( {{{\bar u}_q} - {{\bar u}_{q - 1}}} \right) \cdot \nabla R_{q - 1}^{osc}} \right\|_{{\mathcal{C}^0}}} \\
\lesssim & {\left\| {\left( {{u_{q - 1}} - {{\bar u}_{q - 1}}} \right) \cdot \nabla R_{q - 1}^{osc}} \right\|_{{\mathcal{C}^0}}} + {\left\| {\left( {{u_{q}} - {{\bar u}_{q}}} \right) \cdot \nabla R_{q - 1}^{osc}} \right\|_{{\mathcal{C}^0}}} + {\left\| {\left( {{w_{q - 1}} \cdot \nabla } \right)R_{q - 1}^{osc}} \right\|_{{\mathcal{C}^0}}} \\
\lesssim & \left( {{{\left\| {{u_{q - 1}} - {{\bar u}_{q - 1}}} \right\|}_{{\mathcal{C}^0}}} + {{\left\| {{u_{q}} - {{\bar u}_{q}}} \right\|}_{{\mathcal{C}^0}}} + {\left\| {{w_{q - 1}} + \mathcal{W}_q^{loc}} \right\|_{{\mathcal{C}^0}}}} \right){\left\| {R_{q - 1}^{osc}} \right\|_{{\mathcal{C}^1}}} \\
\lesssim & \left( {{\delta _q}\lambda _q^{1 - \varpi } + {\delta _q}{\lambda _{q - 1}}} \right)\left( {\delta _{q - 1}^{\frac{1}{2}} + {\kappa _{q}} \delta _{q}^{\frac{1}{2}} {\lambda _{q }} + {\kappa _{q - 1}}\delta _{q - 1}^{\frac{1}{2}}{\lambda _{q - 1}}} \right) .
\end{align*}
Therefore,
\begin{equation*}
{\left\| { {{\mathfrak{I} _{t,l,q}}} R_{q - 1}^{osc}} \right\|_{{\mathcal{C}^0}}}  \lesssim \mu _{q + 1}^{ - 1}  \delta _q^{\frac{1}{2}}{\delta _{q + 1}}\lambda _q^{1 - \alpha }  .
\end{equation*}
\par
\vspace{1em}
{\textbf{Step 3. Estimates of the remaining term $R_{q - 1}^{rem}$}}. 
\par
Invoking \eqref{JUIDHI-AS-A}, \eqref{E-OS-SBE-1}, \eqref{E-OS-SBE-2} and Lemma \ref{STN-HCS-A-sjdij}, we derive that
\begin{align*}
{\left\| {R_{q - 1}^{rem}} \right\|_{{\mathcal{C}^0}}} \lesssim & {\left\| {u_{q - 1} - \bar u_{q - 1}} \right\|_{{\mathcal{C}^0}}}{\left\| {{w_{q - 1}}} \right\|_{{\mathcal{C}^0}}} + \lambda _q^{ - 1}{\left\| {{w_{c,q - 1}}} \right\|_{{\mathcal{C}^0}}}{\left\| {{w_{o,q - 1}}} \right\|_{{\mathcal{C}^1}}} + {\left\| {{w_{c,q - 1}}} \right\|_{{\mathcal{C}^0}}}{\left\| {{w_{c,q - 1}}} \right\|_{{\mathcal{C}^0}}} \\
\lesssim & {\kappa _{q - 1}}\delta _{q - 1}^{\frac{1}{2}}\delta _q^{\frac{1}{2}}{\lambda _{q - 1}} + {\mu _q ^{ - 1}}\delta _{q - 1}^{\frac{1}{2}}{\delta _q}{\lambda _{q - 1}} + {\mu_q ^{ - 2}}{\delta _{q - 1}}{\delta _q}\lambda _{q - 1}^2 , 
\end{align*}
and
\begin{align*}
{\left\| {R_{q - 1}^{rem}} \right\|_{{\mathcal{C}^N}}} \lesssim & {\left\| {u_{q - 1} - \bar u_{q - 1}} \right\|_{{\mathcal{C}^0}}}{\left\| {{w_{q - 1}}} \right\|_{{\mathcal{C}^N}}} + {\left\| {u_{q - 1} - \bar u_{q - 1}} \right\|_{{\mathcal{C}^N}}}{\left\| {{w_{q - 1}}} \right\|_{{\mathcal{C}^0}}} \\
& + {\left\| {{w_{c,q - 1}}} \right\|_{{\mathcal{C}^0}}}\left( {{{\left\| {{w_{o,q - 1}}} \right\|}_{{\mathcal{C}^N}}} + {{\left\| {{w_{c,q - 1}}} \right\|}_{{\mathcal{C}^N}}}} \right) + \lambda _q^{ - 1}{\left\| {{w_{c,q - 1}}} \right\|_{{\mathcal{C}^0}}}{\left\| {{w_{o,q - 1}}} \right\|_{{C^{N + 1}}}} \\
& + {\left\| {{w_{c,q - 1}}} \right\|_{{\mathcal{C}^N}}}\left( {{{\left\| {{w_{o,q - 1}}} \right\|}_{{\mathcal{C}^N}}} + {{\left\| {{w_{c,q - 1}}} \right\|}_{{\mathcal{C}^1}}}} \right) + \lambda _q^{ - 1}{\left\| {{w_{c,q - 1}}} \right\|_{{\mathcal{C}^N}}}{\left\| {{w_{o,q - 1}}} \right\|_{{\mathcal{C}^N}}}  \\
\lesssim & {\kappa _{q - 1}}\delta _{q - 1}^{\frac{1}{2}}\delta _q^{\frac{1}{2}}{\lambda _{q - 1}}{\lambda _q^N} + {\mu _q ^{ - 1}}\delta _{q - 1}^{\frac{1}{2}}{\delta _q}{\lambda _{q - 1}}{\lambda _q^N}  .
\end{align*}
Then
\begin{equation}
{\left\| {R_{q - 1}^{rem}} \right\|_{{\mathcal{C}^0}}} + \lambda _q^{ - 1}{\left\| {R_{q - 1}^{rem}} \right\|_{{\mathcal{C}^1}}}  + \lambda _q^{ - 2}{\left\| {R_{q - 1}^{rem}} \right\|_{{\mathcal{C}^2}}} \lesssim {\delta _{q + 1}}\lambda _q^{ - \alpha } . \label{AIUJBHTGS-AB-17}  
\end{equation}
\par
Considering the It\^{o} integral $\int_{{t_l}}^t { {{w_{q - 1}} \cdot \mathrm{d}\left( {{\mathcal{W}_{q - 1}} - {\mathcal{\bar W}_{q - 1}}} \right)}} $, it follows from \eqref{OS-SBE-C-2}, \eqref{BCBSJFHIG-SFG} and \eqref{E-OS-SBE-1}-\eqref{E-OS-SBE-4} that
\begin{align*}
& {\mu _{q + 1}} {\left\| {\int_{{t_l}}^t { {{w_{q - 1}} \cdot \mathrm{d}\left( {{\mathcal{W}_{q - 1}} - {\mathcal{\bar W}_{q - 1}}} \right)}} } \right\|_{{\mathcal{C}^0}}} \\
\lesssim & {\mu _{q + 1}} {\left| {t - {t_l}} \right|^{\frac{1}{2} - \alpha }}{\kappa _{q - 1}}{\left\| {{w_{q - 1}}} \right\|_{{\mathcal{C}^0}}} \sum\limits_{j \leqslant q - 1} {\lambda _j^{\frac{1}{3}  }}  {\left[ {{\mathfrak{B}_j}} \right]_{\frac{1}{2} - \alpha }}  \\
& +  {\kappa _{q - 1}}\left( {{{\left\| {{\mathfrak{D}_{t,q - 1}}{w_{q - 1}}} \right\|}_{{\mathcal{C}^0}}} + {{\left\| {{w_{q - 1}}} \right\|}_{{\mathcal{C}^1}}}} \right)\sum\limits_{j \leqslant q - 1} {\lambda _j^{\frac{1}{3} }\mathop {\sup }\limits_{\tau  \in \left[ {s \wedge \mathfrak{t},t \wedge \mathfrak{t}} \right]} \left| {{\mathfrak{B}_j}\left( \tau  \right)} \right|} \\
\lesssim & C\left( \omega  \right)  \mu _{q + 1}^{\frac{1}{2} + \alpha} {\kappa _{q - 1}}r_q^{ - 1}\delta _q^{\frac{1}{2}}\lambda _{q - 1}^{\frac{1}{3} } + {\kappa _{q - 1}} \delta _q  {\delta _{q + 1}}\lambda _{q - 1}^{\frac{1}{3}  }\lambda _q^{\frac{2}{3} - 2 \alpha }\left( {{\mu _q} + {\lambda _q}} \right) \\
\lesssim & \left( {C\left( \omega  \right) + 1} \right) \delta _q^{\frac{1}{2}}{\delta _{q + 1}}\lambda _q^{1 - \alpha } .
\end{align*}
Notice that
\begin{align*}
{\mu _{q + 1}}{\left\| { {{\mathfrak{I} _{t,l,q - 1}}} R_{q - 1}^{rem}} \right\|_{{\mathcal{C}^0}}} \lesssim &  {\left\| {{\partial _t}\left( {{v_{q - 1}} - {{\bar v}_{q - 1}}} \right)} \right\|_{{\mathcal{C}^0}}}{\left\| {{w_{q - 1}}} \right\|_{{\mathcal{C}^0}}} + {\left\| {\int_{{t_l}}^t  {{w_{q - 1}} \cdot \mathrm{d}\left( {{\mathcal{W}_{q - 1}} - {\mathcal{\bar W}_{q - 1}}} \right)} } \right\|_{{\mathcal{C}^0}}}  \\
& +  {\left\| {u_{q - 1} - \bar u_{q - 1}} \right\|_{{\mathcal{C}^0}}}\left( {{{\left\| {{\mathfrak{D}_{t,q - 1}}{w_{q - 1}}} \right\|}_{{\mathcal{C}^0}}} + {{\left\| {\bar u_{q} - \bar u_{q - 1}} \right\|}_{{\mathcal{C}^0}}}{{\left\| {{w_{q - 1}}} \right\|}_{{\mathcal{C}^1}}}} \right) \\
& +  {\left\| {{\mathfrak{D}_{t,q - 1}}{w_{c,q - 1}}} \right\|_{{\mathcal{C}^0}}}{\left\| {{w_{q - 1}}} \right\|_{{\mathcal{C}^0}}} +  {\left\| {{\mathfrak{D}_{t,q - 1}}{w_{o,q - 1}}} \right\|_{{\mathcal{C}^0}}}{\left\| {{w_{c,q - 1}}} \right\|_{{\mathcal{C}^0}}} \\
& + {\left\| {{{\bar u}_q} - {{\bar u}_{q - 1}}} \right\|_{{\mathcal{C}^0}}}{\left\| {{w_{q - 1}}} \right\|_{{\mathcal{C}^1}}}{\left\| {{w_{c,q - 1}}} \right\|_{{\mathcal{C}^0}}}  +  {\left\| {{w_{q - 1}}} \right\|_{{\mathcal{C}^0}}}{\left\| {{{\bar u}_q}} \right\|_{{\mathcal{C}^0}}}{\left\| {{u_{q - 1}} - {{\bar u}_{q - 1}}} \right\|_{{\mathcal{C}^1}}} .
\end{align*}
Therefore, employing \eqref{JUIDHI-AS-A}, \eqref{IJUHCJD-D-D-A-1-JU-2} and \eqref{E-OS-SBE-1}-\eqref{E-OS-SBE-4}, we infer that
\begin{equation*}
{\left\| {{\mathfrak{I}_{s,t,q}}R_{q - 1}^{rem}} \right\|_{{\mathcal{C}^0}}} \lesssim \left( {C\left( \omega  \right) + 1} \right) \mu _{q + 1}^{ - 1} \delta _{q + 1}^{\frac{1}{2}}{\delta _{q + 1}}\lambda _q^{1 - \alpha } .
\end{equation*}
\par
\vspace{1em}
{\textbf{Step 4. Estimates of the stochastic error $R_{q - 1}^{sto}$}}. 
\par
Using two identities
\begin{equation*}
{\text{div}}\left( {{w_{q - 1}} \otimes \mathcal{W}_q^{loc}} \right) = \sum\limits_{\left( {{\mathbf{k}},l} \right)} {{\chi _l}\left( {L_{{\mathbf{k}}l}  \cdot \nabla \mathcal{W}_q^{loc}} \right){\phi _{{\mathbf{k}}l}}\exp \left( {{\text{i}}{\lambda _q}{\mathbf{k}} \cdot x} \right)} ,
\end{equation*}
and
\begin{equation*}
{\text{div}}\left( {\mathcal{W}_q^{loc} \otimes {w_{q - 1}}} \right) =  {b_q}{\mathfrak{B}_q}\left( t \right)\lambda _q^{ - \frac{2}{3} + \alpha }\left( {\cos \left( {{\lambda _q}{x_3}} \right),\sin \left( {{\lambda _q}{x_3}} \right),0} \right) \cdot \nabla {w_{q - 1}} ,
\end{equation*}
it follows from \eqref{BCBSJFHIG-SFG}, \eqref{IJUHCJD-D-D-A-2-dfjgtyu-3e}, \eqref{IJUHCJD-D-D-A-4} and \eqref{IJSUHD-A-B-5-proof-a-1} that
\begin{align*}
& {\left\| {\mathcal{R}{\text{div}}\left( {{w_{q - 1}} \otimes \mathcal{W}_q^{loc}} \right)} \right\|_{{\mathcal{C}^0}}} + \lambda _q^{ - 1}{\left\| {\mathcal{R}{\text{div}}\left( {{w_{q - 1}} \otimes \mathcal{W}_q^{loc}} \right)} \right\|_{{\mathcal{C}^1}}}  \\
& + {\left\| {\mathcal{R}{\text{div}}\left( {\mathcal{W}_q^{loc} \otimes {w_{q - 1}}} \right)} \right\|_{{\mathcal{C}^0}}} + \lambda _q^{ - 1}{\left\| {\mathcal{R}{\text{div}}\left( {\mathcal{W}_q^{loc} \otimes {w_{q - 1}}} \right)} \right\|_{{\mathcal{C}^1}}} \\
\lesssim & \lambda _q^{ - 1}{\left\| {L_{{\mathbf{k}}l} {\phi _{{\mathbf{k}}l}}} \right\|_{{\mathcal{C}^0}}}{\left\| {\mathcal{W}_q^{loc}} \right\|_{{\mathcal{C}^1}}} + \lambda _q^{ - 2}{\left\| {{\phi _{{\mathbf{k}}l}}L_{{\mathbf{k}}l}  \cdot \nabla \mathcal{W}_q^{loc}} \right\|_{{\mathcal{C}^1}}} \\
& + \lambda _q^{ - 1}\left( {\lambda _q^{ - \frac{2}{3} + \alpha }{{\left\| {{w_{q - 1}}} \right\|}_{{\mathcal{C}^1}}}} \right) + \lambda _q^{ - 2}\left( {\lambda _q^{\frac{1}{3} + \alpha }{{\left\| {{w_{q - 1}}} \right\|}_{{\mathcal{C}^1}}} + \lambda _q^{ - \frac{2}{3} + \alpha }{{\left\| {{w_{q - 1}}} \right\|}_{{\mathcal{C}^2}}}} \right) \\
\lesssim & {\delta _{q + 1}}\lambda _q^{ - \alpha } .
\end{align*}
Moreover, we derive from \eqref{IJSUHD-A-1-KDJI} and \eqref{IJSUHD-A-3} that 
\begin{align*}
& {\left\| {\mathcal{R}{\text{div}}\left( {{v_{q - 1}} \otimes \mathcal{W}_q^{loc} + \mathcal{W}_q^{loc} \otimes {v_{q - 1}}} \right)} \right\|_{{\mathcal{C}^0}}} + \lambda _q^{ - 1}{\left\| {\mathcal{R}{\text{div}}\left( {{v_{q - 1}} \otimes \mathcal{W}_q^{loc} + \mathcal{W}_q^{loc} \otimes {v_{q - 1}}} \right)} \right\|_{{\mathcal{C}^1}}}  \\
\lesssim & {\left\| {{v_{q - 1}}} \right\|_{{\mathcal{C}^0}}}{\left\| {\mathcal{W}_q^{loc}} \right\|_{{\mathcal{C}^0}}} + \lambda _q^{ - 1}\left( {{{\left\| {{v_{q - 1}}} \right\|}_{{\mathcal{C}^1}}}{{\left\| {\mathcal{W}_q^{loc}} \right\|}_{{\mathcal{C}^0}}} + {{\left\| {{v_{q - 1}}} \right\|}_{{\mathcal{C}^0}}}{{\left\| {\mathcal{W}_q^{loc}} \right\|}_{{\mathcal{C}^1}}}} \right) \\
& + \lambda _q^{ - 2}\left( {{{\left\| {{v_{q - 1}}} \right\|}_{{\mathcal{C}^1}}}{{\left\| {\mathcal{W}_q^{loc}} \right\|}_{{\mathcal{C}^1}}} + {{\left\| {{v_{q - 1}}} \right\|}_{{\mathcal{C}^0}}}{{\left\| {\mathcal{W}_q^{loc}} \right\|}_{{\mathcal{C}^2}}} + {{\left\| {{v_{q - 1}}} \right\|}_{{\mathcal{C}^2}}}{{\left\| {\mathcal{W}_q^{loc}} \right\|}_{{\mathcal{C}^0}}}} \right) \\
\lesssim & {\delta _{q + 1}}\lambda _q^{ - \alpha } .
\end{align*}
Similarly, we also have
\begin{align*}
{\left\| {\mathcal{R}{\text{div}}\tilde F } \right\|_{{\mathcal{C}^2}}}\lesssim  {\delta _{q + 1}}\lambda _q^{2 - \alpha } ,
\end{align*}
where $\tilde F \triangleq {w_{q - 1}} \otimes \mathcal{W}_q^{loc} + \mathcal{W}_q^{loc} \otimes {w_{q - 1}} + {v_{q - 1}} \otimes \mathcal{W}_q^{loc} + \mathcal{W}_q^{loc} \otimes {v_{q - 1}}$. Therefore,
\begin{equation*}
{\left\| { R_{q - 1}^{sto} } \right\|_{{\mathcal{C}^0}}} + \lambda _q^{ - 1}{\left\| { R_{q - 1}^{sto} } \right\|_{{\mathcal{C}^1}}}  + \lambda _q^{ - 2}{\left\| { R_{q - 1}^{sto} } \right\|_{{\mathcal{C}^2}}} \lesssim {\delta _{q + 1}}\lambda _q^{ - \alpha } .  
\end{equation*}
\par
Applying the material integration operator to $R_{q - 1}^{sto}$, we have
\begin{align*}
{{\mathfrak{I} _{t,l,q}}}\mathcal{R}{\text{div}}\tilde F = & \mathcal{R}{\text{div}}  {{\mathfrak{I} _{t,l,q}}}  \tilde F + \int_{{t_l}}^t {\left( {{{\bar u}_q} - {{\bar u}_{q - 1}}} \right) \cdot \nabla \mathcal{R}{\text{div}}\tilde F} \, \mathrm{d}\tau \\
& + \int_{{t_l}}^t {\left[ {{{\bar u}_{q - 1}} \cdot \nabla ,\mathcal{R}{\text{div}}} \right]\tilde F} \, \mathrm{d}\tau  .
\end{align*}
It follows from \eqref{BCBSJFHIG-SFG}, \eqref{IJSUHD-A-1-KDJI}, \eqref{P-Ie-AAA-F-2}, \eqref{E-OS-SBE-1} and \eqref{E-OS-SBE-2} that
\begin{equation*}
{\left\| {\tilde F} \right\|_{{\mathcal{C}^0}}} + \lambda _q^{ - 1}{\left\| {\tilde F} \right\|_{{\mathcal{C}^1}}} + \lambda _q^{ - 2}{\left\| {\tilde F} \right\|_{{\mathcal{C}^2}}} \lesssim \delta _q^{\frac{1}{2}}{\delta _{q + 1}}\lambda _q^{ - 2\alpha }  .
\end{equation*}
Owing to ${\mathfrak{I} _{t,l,q}}\tilde F = \int_{{t_l}}^t {{\text{d}}\tilde F}  + \int_{{t_l}}^t {{{\bar u}_{q - 1}} \cdot \nabla \tilde F} {\text{d}}\tau $, we deduce that
\begin{align}
{\mu _{q + 1}} {\left\| {\mathcal{R}{\text{div}}\int_{{t_l}}^t {{\text{d}}\tilde F} } \right\|_{{\mathcal{C}^\alpha }}} \lesssim & {\left\| {{\partial _t}\left( {{v_{q - 1}} + {w_{q - 1}}} \right) \cdot W_q^{loc}} \right\|_{{\mathcal{C}^\alpha }}} + \mu _q^{ - 1}{\left\| {{{\bar u}_{q - 1}}} \right\|_{{\mathcal{C}^0}}}\left\| {\tilde F} \right\|_{{\mathcal{C}^1}}^{1 - \alpha }\left\| {\tilde F} \right\|_{{\mathcal{C}^2}}^\alpha  \nonumber \\
& + \left\| {{{\bar u}_{q - 1}}} \right\|_{{\mathcal{C}^0}}^{1 - \alpha }\left\| {{{\bar u}_{q - 1}}} \right\|_{{\mathcal{C}^1}}^\alpha {\left\| {\tilde F} \right\|_{{\mathcal{C}^1}}} + \mu _{q + 1}^{\frac{1}{2} + \alpha}\lambda _q^{ - \frac{2}{3}}{\left\| {{v_{q - 1}} + {w_{q - 1}}} \right\|_{{\mathcal{C}^\alpha }}}{\left[ {{\mathfrak{B}_q}} \right]_{\frac{1}{2} - \alpha }} \nonumber \\
\lesssim & \delta _q^{\frac{1}{2}}{\delta _{q + 1}}\lambda _q^{ - \alpha }\left( {\delta _{q - 1}^{\frac{1}{2}}{\lambda _{q - 1}} + {\lambda _q}} \right) + C\left( \omega  \right) \mu _{q + 1}^{\frac{1}{2} + \alpha}  r_q^{ - 1}\lambda _q^{ - \frac{2}{3}}\delta _{q - 1}^{\frac{1}{2}\alpha }\lambda _{q - 1}^\alpha  \nonumber  \\
\lesssim & \left( {C\left( \omega  \right) + 1} \right) \delta _q^{\frac{1}{2}}{\delta _{q + 1}}\lambda _q^{1 - \alpha } .  \label{XNUWEI-ASD-1}
\end{align}
In addition, we infer that
\begin{align}
{\mu _{q + 1}}{\left\| {\int_{{t_l}}^t {\left( {{{\bar u}_q} - {{\bar u}_{q - 1}}} \right) \cdot \nabla \mathcal{R}{\text{div}}\tilde F} \, {\text{d}}\tau } \right\|_{{\mathcal{C}^0}}} \lesssim & {\left\| {{{\bar u}_q} - {{\bar u}_{q - 1}}} \right\|_{{\mathcal{C}^0}}}{\left\| {\tilde F} \right\|_{{\mathcal{C}^{1 + \alpha }}}} \nonumber \\
\lesssim &  {\left\| {{{\bar u}_q} - {{\bar u}_{q - 1}}} \right\|_{{\mathcal{C}^0}}}\left\| {\tilde F} \right\|_{{\mathcal{C}^1}}^{1 - \alpha }\left\| {\tilde F} \right\|_{{\mathcal{C}^2}}^\alpha   \nonumber \\
\lesssim &  \delta _q^{\frac{1}{2}}{\delta _{q + 1}}\lambda _q^{ 1 - \alpha }  .  \label{XNUWEI-ASD-2}
\end{align}
For the commutator, we derive that
\begin{align}
{\mu _{q + 1}} {\left\| {\int_{{t_l}}^t {\left[ {{{\bar u}_{q - 1}} \cdot \nabla ,\mathcal{R}{\text{div}}} \right]\tilde F} \, {\text{d}}\tau } \right\|_{{\mathcal{C}^\alpha }}} \lesssim &  {\left\| {{{\bar u}_{q - 1}}} \right\|_{{\mathcal{C}^{1 + \alpha }}}}\left\| {\tilde F} \right\|_{{\mathcal{C}^0}}^{1 - \alpha }\left\| {\tilde F} \right\|_{{\mathcal{C}^1}}^\alpha  \nonumber \\
\lesssim &  \delta _q^{\frac{1}{2}}{\delta _{q + 1}}\lambda _q^{1 - \alpha }  . \label{XNUWEI-ASD-3}
\end{align}
Therefore, combining \eqref{XNUWEI-ASD-1}, \eqref{XNUWEI-ASD-2} and \eqref{XNUWEI-ASD-3}, we obtain
\begin{equation*}
{\left\| { {\mathfrak{I}_{t,l,q}} R_{q - 1}^{sto} } \right\|_{{\mathcal{C}^0}}} \lesssim \left( {C\left( \omega  \right) + 1} \right) \mu _{q + 1}^{ - 1} \delta _q^{\frac{1}{2}}{\delta _{q + 1}}\lambda _q^{1 - \alpha }   .
\end{equation*}
\par
\vspace{1em}
{\textbf{Step 5. Estimates of the compensative corrector $R_{q - 1}^{com}$}}. 
\par
Let ${{X}_l}$ be the Lagrangian flow satisfying
\begin{equation}\label{KDKJFIG-DNGJU-1-1}
\left\{ {\begin{array}{*{20}{l}}
\mathrm{d} {{ X}_l}\left( {t,x} \right) = {{\bar u}_{q - 1}}\left( {t, {{ X}_l}\left( {t,x} \right)} \right) \mathrm{d} t,\\
{{X}_l}\left( {t_l ,x} \right) = x .
\end{array}} \right.
\end{equation}
By It\^{o}--Wentzell formula, we deduce that
\begin{align*}
\mathrm{d}\left( {{{\bar R}_{q - 1}}\left( {t,{X_l}} \right) - {{\bar R}_l}\left( {t,{X_l}} \right)} \right) = & {\mathrm{d}_t}\left( {{{\bar R}_{q - 1}}\left( {t,{X_l}} \right) - {{\bar R}_l}\left( {t,{X_l}} \right)} \right)  \\
& + {{\bar u}_{q - 1}}\left( {t,{X_l}} \right) \cdot \nabla \left( {{{\bar R}_{q - 1}}\left( {t,{X_l}} \right) - {{\bar R}_l}\left( {t,{X_l}} \right)} \right)\mathrm{d}t .
\end{align*}
Recall that ${\mathrm{d}_t}{{\bar R}_l}\left( {t,{X_l}} \right) + {{\bar u}_{q - 1}}\left( {t,{X_l}} \right) \cdot \nabla {{\bar R}_l}\left( {t,{X_l}} \right)\mathrm{d}t = 0$. We then get
\begin{equation*}
{{\bar R}_{q - 1}}\left( {t,{X_l}} \right) - {{\bar R}_l}\left( {t,{X_l}} \right) = {\mathfrak{I} _{t,l,q - 1}}{{\bar R}_{q - 1}}\left( {t,{X_l}} \right), 
\end{equation*}
where $\mathfrak{I}_{t,l,{q - 1}} \left(  \cdot  \right) \triangleq \int_{t_l}^{t } {\mathrm{d}\left(  \cdot  \right)}  + \int_{t_l}^{t } {{\bar u_{q - 1}} \cdot \nabla \left(  \cdot  \right)\mathrm{d}t}  $, $t \in {{\rm T}_l}$, is the localized material integration corresponding to ${\bar u_{q - 1}}$.
\par
By virtue of the homeomorphic property, we obtain
\begin{equation*}
{\left\| {\bar R_{q - 1} - {{\bar R}_l}} \right\|_{{\mathcal{C}^0}}} = \left\| {{{\bar R}_{q - 1}}\left( {t,{X_l}} \right) - {{\bar R}_l}\left( {t,{X_l}} \right)} \right\|_{{\mathcal{C}^0}} \lesssim {\left\| {{\mathfrak{I}_{t,l,q - 1}}\bar R_{q - 1} \left( {\tau,{X_l}} \right) } \right\|_{{\mathcal{C}^0}}} .
\end{equation*}
Therefore, using the inductive estimate \eqref{IJSUHD-A-5}, the fact ${\left\| {\mathfrak{D}{X_l}} \right\|_{{\mathcal{C}^N}}} \lesssim \lambda _{q - 1}^N$ and Lemma \ref{INSKDINEQU-afski-a-1}, we infer that
\begin{align*}
{\left\| {{\mathfrak{I} _{t,l,q - 1}}{{\bar R}_{q - 1}} \left( {\tau,{X_l}} \right)} \right\|_{{\mathcal{C}^0}}} \lesssim  C\left( \omega  \right) \mu _q^{ - 1} \delta _{q - 1}^{\frac{1}{2}}{\delta _q}\lambda _{q - 1}^{1 - \alpha } ,
\end{align*}
which implies that
\begin{equation*}
{\left\| {\bar R_{q - 1} - {{\bar R}_l}} \right\|_{{\mathcal{C}^0}}} \lesssim C\left( \omega  \right) \mu _q^{ - 1} \delta _{q - 1}^{\frac{1}{2}}{\delta _q}\lambda _{q - 1}^{1 - \alpha } .
\end{equation*}
In addition, it follows from \eqref{IJSUHD-A-4} and \eqref{IJSUHD-A-5} that
\begin{equation*}
{\left\| {\mathring{R}_{q - 1} - \bar R_{q - 1}} \right\|_{{\mathcal{C}^N}}} \lesssim {\kappa _{q - 1}}{\left\| {\mathring{R}_{q - 1}} \right\|_{{\mathcal{C}^{N + 1}}}} \lesssim {\kappa _{q - 1}}{\delta _q}\lambda _{q - 1}^{N + 1 - \alpha } .
\end{equation*}
Since ${\kappa _{q - 1}}{\delta _q}\lambda _{q - 1}^{1 - \alpha } + \mu _q^{ - 1}\delta _{q - 1}^{\frac{1}{2}}{\delta _q}\lambda _{q - 1}^{1 - \alpha } \lesssim {\delta _{q + 1}}\lambda _q^{ - \alpha }$, we deduce that
\begin{equation*}
{\left\| { R_{q - 1}^{com} } \right\|_{{\mathcal{C}^0}}} + \lambda _q^{ - 1}{\left\| { R_{q - 1}^{com} } \right\|_{{\mathcal{C}^1}}}  + \lambda _q^{ - 2}{\left\| { R_{q - 1}^{com} } \right\|_{{\mathcal{C}^2}}}  \lesssim {\delta _{q + 1}}\lambda _q^{ - \alpha } . 
\end{equation*}
\par
For the material integration of ${{\mathring{R}_{q - 1}} - {{\bar R}_{q - 1}}}$, we have
\begin{align*}
{\mathfrak{I}_{t,l,q}}\left( {{\mathring{R}_{q - 1}} - {{\bar R}_{q - 1}}} \right) = & \int_{{t_l}}^t  {\mathrm{d}\left( {{\mathring{R}_{q - 1}} - {{\bar R}_{q - 1}}} \right)}  + \int_{{t_l}}^t  {{{\bar u}_{q - 1}} \cdot \nabla \left( {{\mathring{R}_{q - 1}} - {{\bar R}_{q - 1}}} \right)} \, \mathrm{d}\tau  \\
& + \int_{{t_l}}^t  {\left( {{{\bar u}_q} - {{\bar u}_{q - 1}}} \right) \cdot \nabla \left( {{\mathring{R}_{q - 1}} - {{\bar R}_{q - 1}}} \right)} \, \mathrm{d}\tau .
\end{align*}
Owing to Lemma \ref{Ie-SBE}, we rewrite the Reynolds stress ${\mathring{R}_{q - 1}}$ as a semimartingale decomposition
\begin{equation*}
{\text{d}}{\mathring{R}_{q - 1}} = {\mathcal{A}_R}{\text{d}}t + \sum\limits_{j \leqslant q - 1} {{{\tilde H}_j}} {\mkern 1mu} {\text{d}}{\mathfrak{B}_j} ,
\end{equation*}
where ${\left\| {{\mathcal{A}_R}} \right\|_{{\mathcal{C}^0}}} \lesssim {\left\| {{\mathring{R}_{q - 1}}} \right\|_{{\mathcal{C}^1}}} \lesssim {\delta _q}\lambda _{q - 1}^{1 - \alpha }$ and ${\left\| {{{\tilde H}_j}} \right\|_{{\mathcal{C}^0}}} \lesssim \delta _{j + 1}^{\frac{1}{2}}\lambda _j^{\frac{1}{3}}$. We then otain
\begin{equation*}
\int_{{t_l}}^t {{\text{d}}\left( {{\mathring{R}_{q - 1}} - {{\bar R}_{q - 1}}} \right)}  = \int_{{t_l}}^t {{\mathcal{A}_R} - {{ \mathcal{\bar A}}_R}} {\text{d}}\tau  + \int_{{t_l}}^t {\sum\limits_{jq - 1} {\left( {{{\tilde H}_j} - {{\tilde H}_j}*{\varsigma _{{\kappa _{q - 1}}}}} \right)} {\text{d}}{\mathfrak{B}_j}}  .
\end{equation*}
In the preceding analysis, we have rigorously analyzed the contribution of It\^{o} integrals to the material integration, emphasizing that the associated time-scale factors must be retained in the local temporal scale $\left| {t - {t_l}} \right| \lesssim \mu _{q + 1}^{ - 1}$. We employ the same argument as \eqref{IJSUHD-A-B-5-prosdafof-a-4} to handle the It\^{o} integral $\int_{{t_l}}^t {\sum\limits_{j \leqslant q - 1} {\left( {{{\tilde H}_j} - {{\hat H}_j}} \right)} {\text{d}}{\mathfrak{B}_j}}$. Therefore, it is easy to obtain
\begin{equation*}
{\mu _{q + 1}} {\left\| {{\mathfrak{I}_{t,l,q}}\left( {{\mathring{R}_{q - 1}} - {{\bar R}_{q - 1}}} \right)} \right\|_{{\mathcal{C}^0}}} \lesssim  \delta _q^{\frac{1}{2}}{\delta _{q + 1}}\lambda _q^{1 - \alpha } . 
\end{equation*}
Moreover, using the fact that ${\mathfrak{D}_{t,q - 1}}{{\bar R}_{l}} = 0$, we have
\begin{align*}
{\left\| {{{\mathfrak{I} _{t,l,q}}}\left( {{{\bar R}_{q - 1}} - {{\bar R}_l}} \right)} \right\|_{{\mathcal{C}^0}}} \lesssim & {\left\| { {{\mathfrak{I} _{t,l,q}}} {{\bar R}_{q - 1}}} \right\|_{{\mathcal{C}^0}}} + {\left\| {\int_{{t_l}}^t {\left( {{{\bar u}_q} - {{\bar u}_{q - 1}}} \right) \cdot \nabla {{\bar R}_l}} d\tau } \right\|_{{\mathcal{C}^0}}} \\
\lesssim & \mu _{q + 1}^{ - 1}\delta _q^{\frac{1}{2}}{\delta _{q + 1}}\lambda _q^{1 - \alpha } .
\end{align*}
Consequently,
\begin{equation*}
{\left\| { {{\mathfrak{I} _{t,l,q}}} R_{q - 1}^{com} } \right\|_{{\mathcal{C}^0}}} \lesssim \mu _{q + 1}^{ - 1}  \delta _q^{\frac{1}{2}}{\delta _{q + 1}}\lambda _q^{1 - \alpha } . 
\end{equation*}
The proof is complete. \qed
\par
\begin{proposition}\label{RSRI-ESI-GAP}
For any $t \in \left[ {0, T \wedge \mathfrak{t}} \right]$, the $q$-step energy gap $\mathfrak{E}_{q}$ defined by \eqref{SJKFJdafj-sagfg} satisfies
\begin{equation}\label{RSRI-P-mene-GAP}
0 < {\rho _{q,l}}  \sim {\delta _{q + 1}} . 
\end{equation}
\end{proposition}
\par
\noindent{\textbf{Proof}}.
It follows from Lemma \ref{BFL-1} and Lemma \ref{BFL-2} that
\begin{align*}
\int_{{\mathbb{T}^3}} {{w_{o,q - 1}}\left( x \right) \otimes {w_{o,q - 1}}\left( x \right)} \, {\mathrm{d}}x = & \left| {{\mathbb{T}^3}} \right| {\rho _{q - 1,l}} I - \sum\limits_{l \in \left[ {0,{\mu _q}} \right]} {\chi _l^2\int_{{\mathbb{T}^3}} {{{\bar R}_l}\left( x \right)} \, {\mathrm{d}}x}  \\
& + \sum\limits_{\left( {\mathbf{k},l} \right),\left( {\mathbf{k}',l'} \right),\mathbf{k} + \mathbf{k}' \ne 0} {{\chi _l}{\chi _{l'}}\int_{{\mathbb{T}^3}} {w_{{\mathbf{k}}l}\left( x \right) \otimes w_{{\mathbf{k}}'l'}^u\left( x \right)} \, {\mathrm{d}}x} . 
\end{align*}
Then
\begin{equation*}
\left\| {{w_{o,q - 1}}} \right\|_{{L^2}}^2 = 3 \left| {{\mathbb{T}^3}} \right|{\rho _{q - 1,l}} + \mathrm{Tr}\, \mathfrak{E}_{q - 1}^{Rey} ,
\end{equation*}
where
\begin{equation}
\mathfrak{E}_{q - 1}^{Rey} = \sum\limits_{\left( {\mathbf{k},l} \right),\left( {\mathbf{k}',l'} \right),\mathbf{k} + \mathbf{k}' \ne 0} {{\chi _l}{\chi _{l'}}\int_{{\mathbb{T}^3}} {w_{{\mathbf{k}}l}\left( x \right) \otimes w_{{\mathbf{k}}'l'}\left( x \right)} \, {\mathrm{d}}x}  - \sum\limits_{l \in \left[ {0,{\mu _q}} \right]} {\chi _l^2\int_{{\mathbb{T}^3}} {{{\bar R}_l}\left( x \right)} \, {\mathrm{d}}x}   .  \label{RSRI-P-mene-GAP-proof-1}
\end{equation}
Notice that
\begin{equation*}
{{\hat u}_q} = {{\hat u}_{q - 1}} + w_{q - 1} * {\varsigma _{{{\hat \kappa }_{q}}}} + \mathcal{\hat W}_q^{loc} + {u_{q - 1}} * \left( {{\varsigma _{{{\hat \kappa }_q}}} - {\varsigma _{{{\hat \kappa }_{q - 1}}}}} \right) ,
\end{equation*} 
\begin{equation*}
{{\hat v}_q} = {{\hat v}_{q - 1}} +  w_{q - 1} * {\varsigma _{{{\hat \kappa }_{q}}}} + {v_{q - 1}} * \left( {{\varsigma _{{{\hat \kappa }_q}}} - {\varsigma _{{{\hat \kappa }_{q - 1}}}}} \right)  ,
\end{equation*}
\begin{equation*}
\mathcal{\hat W}_q = \mathcal{\hat W}_{q - 1} + \mathcal{\hat W}_q^{loc} + {\mathcal{W}_{q - 1}} * \left( {{\varsigma _{{{\hat \kappa }_q}}} - {\varsigma _{{{\hat \kappa }_{q - 1}}}}} \right)  ,
\end{equation*}
where $\mathcal{\hat W}_q^{loc} = \mathcal{W}_q^{loc} * {\varsigma _{{{\hat \kappa }_{q}}}}$. We then have
\begin{align*}
3 \left| {{\mathbb{T}^3}} \right| {\rho _{q,l}}  = & {\mathfrak{e} - \mathbb{E} \left\| {{{\hat u}_q}} \right\|_{{L^2}}^2 + \mathbb{E} \left\| {{{ \mathcal{\hat W}}_q}} \right\|_{{L^2}}^2 + 2 \mathbb{E} \int_{0}^{t_l \wedge  \mathfrak{t}} {{{\left( {{{\hat v}_q},{\text{d}}{\mathcal{\hat W}_q}} \right)}_{{\mathbb{T}^3}}}} } - \frac{1}{2} \mathfrak{M} {\delta _{q + 2}} \\
\triangleq & \mathbb{E} {\mathfrak{E}_{q - 1}^{osc} + \mathbb{E} \mathfrak{E}_{q - 1}^{sto} + \mathbb{E} \mathfrak{E}_{q - 1}^{rem}} - \mathbb{E} \mathrm{Tr}\, \mathfrak{E}_{q - 1}^{Rey} ,
\end{align*}
where
\begin{align*}
\mathfrak{E}_{q - 1}^{osc} = & \left\| {{w_{o,q - 1}}} \right\|_{{L^2}}^2 - \left\| {{w_{o,q - 1}} * {\varsigma _{{{\hat \kappa }_{q}}}}} \right\|_{{L^2}}^2 - \left( {{w_{o,q - 1}} * {\varsigma _{{{\hat \kappa }_{q}}}},{w_{c,q - 1}} * {\varsigma _{{{\hat \kappa }_{q}}}}} \right)_{\mathbb{T}^3} \\
& - \left( {{w_{c,q - 1}} * {\varsigma _{{{\hat \kappa }_{q }}}}, w_{q - 1}  * {\varsigma _{{{\hat \kappa }_{q }}}}} \right)_{\mathbb{T}^3} - 2\left( { w_{q - 1}  * {\varsigma _{{{\hat \kappa }_{q }}}},{u_{q - 1}} * {\varsigma _{{{\hat \kappa }_q}}}} \right)_{\mathbb{T}^3}  ,
\end{align*}
\begin{align*}
\mathfrak{E}_{q - 1}^{sto} = & 2\left( {\mathcal{\hat W}_q^{loc},\mathcal{\hat W}_{q - 1} - {{\hat u}_{q - 1}} - w_{q - 1}  * {\varsigma _{{{\hat \kappa }_{q }}}} - {u_{q - 1}} * \left( {{\varsigma _{{{\hat \kappa }_q}}} - {\varsigma _{{{\hat \kappa }_{q - 1}}}}} \right)} \right)_{\mathbb{T}^3}  \\
& + \left\| {{\mathcal{W}_{q - 1}} * \left( {{\varsigma _{{{\hat \kappa }_q}}} - {\varsigma _{{{\hat \kappa }_{q - 1}}}}} \right)} \right\|_{{L^2}}^2 + 2\left( {{\mathcal{W}_{q - 1}} * \left( {{\varsigma _{{{\hat \kappa }_q}}} - {\varsigma _{{{\hat \kappa }_{q - 1}}}}} \right),\mathcal{\hat W}_{q - 1} + \mathcal{\hat W}_q^{loc}} \right)_{\mathbb{T}^3}  \\
& + 2 \int_{0}^{t_l \wedge  \mathfrak{t}} {\left( {{{\hat v}_{q - 1}},{\mathrm{d}}\mathcal{\hat W}_q^{loc} + {\mathrm{d}} {\mathcal{W}_{q - 1}} * \left( {{\varsigma _{{{\hat \kappa }_q}}} - {\varsigma _{{{\hat \kappa }_{q - 1}}}}} \right)} \right)_{\mathbb{T}^3}} \\
& + 2 \int_{0}^{t_l \wedge  \mathfrak{t}} {\left( { w_{q - 1}  * {\varsigma _{{{\hat \kappa }_{q }}}}  + {v_{q - 1}} * \left( {{\varsigma _{{{\hat \kappa }_q}}} - {\varsigma _{{{\hat \kappa }_{q - 1}}}}} \right),{\mathrm{d}}{\mathcal{\hat W}_{q}}} \right)_{\mathbb{T}^3}}  ,
\end{align*}
and
\begin{equation*}
\mathfrak{E}_{q - 1}^{rem} = \frac{1}{2}\left( {{\delta _{q + 1}} - {\delta _{q + 2}}} \right)\mathfrak{M} - \left( {{u_{q - 1}} * \left( {{\varsigma _{{{\hat \kappa }_q}}} - {\varsigma _{{{\hat \kappa }_{q - 1}}}}} \right),{u_{q - 1}} * \left( {{\varsigma _{{{\hat \kappa }_q}}} + {\varsigma _{{{\hat \kappa }_{q - 1}}}}} \right)} \right)_{\mathbb{T}^3} .
\end{equation*}
\par
It follows from \eqref{SJKFHASJJFN} and \eqref{E-OS-SBE-1} that 
\begin{align*}
\left| {\left\| {{w_{o,q - 1}}} \right\|_{{L^2}}^2 - \left\| {{w_{o,q - 1}} * {\varsigma _{{{\hat \kappa }_{q}}}}} \right\|_{{L^2}}^2} \right| = & \left| {{{\left( {{w_{o,q - 1}} - {w_{o,q - 1}} * {\varsigma _{{{\hat \kappa }_{q}}}},{w_{o,q - 1}} + {w_{o,q - 1}} * {\varsigma _{{{\hat \kappa }_{q}}}}} \right) }_{{\mathbb{T}^3}}}} \right|  \\
\lesssim & {{\hat \kappa }_{q}}{\left\| {{w_{o,q - 1}}} \right\|_{{\mathcal{C}^0}}}{\left\| {{w_{o,q - 1}}} \right\|_{{\mathcal{C}^1}}} \\
\lesssim & {\delta _{q + 1}} .
\end{align*}
Employing \eqref{E-OS-SBE-1}, \eqref{E-OS-SBE-2} and the stationary phase lemma, there exists a real number $\alpha \in \left( {0, 1 } \right)$ such that
\begin{align*}
& \left| {\left( {{w_{o,q - 1}} * {\varsigma _{{{\hat \kappa }_{q }}}},{w_{c,q - 1}} * {\varsigma _{{{\hat \kappa }_{q}}}}} \right)_{\mathbb{T}^3} + \left( {{w_{c,{q - 1}}} * {\varsigma _{{{\hat \kappa }_{q }}}}, w_{q - 1} * {\varsigma _{{{\hat \kappa }_{q }}}}} \right)_{\mathbb{T}^3} + \left( {w_{q - 1} * {\varsigma _{{{\hat \kappa }_{q }}}} ,{u_{q - 1}} * {\varsigma _{{{\hat \kappa }_q}}}} \right)_{\mathbb{T}^3}} \right| \\
\lesssim & \left| {\left( {\mathcal{R}{w_{o,q - 1}} * {\varsigma _{{{\hat \kappa }_{q }}}},\nabla \left( {{w_{c,q - 1}} * {\varsigma _{{{\hat \kappa }_{q }}}}} \right)} \right)_{\mathbb{T}^3}} \right| + \left\| {{w_{c,q - 1}} * {\varsigma _{{{\hat \kappa }_{q}}}}} \right\|_{{L^2}}^2 + \left| {\left( {\mathcal{R} \left( { w_{q - 1} * {\varsigma _{{{\hat \kappa }_{q }}}}} \right),\nabla \left( {{u_{q - 1}} * {\varsigma _{{{\hat \kappa }_q}}}} \right)} \right)_{\mathbb{T}^3}} \right| \\
\lesssim & \lambda _q^{\alpha  - 1}  {\left\| {{w_{o,q - 1}}} \right\|_{{\mathcal{C}^0}}}{\left\| {{w_{c,q - 1}}} \right\|_{{\mathcal{C}^1}}} + \left\| {{w_{c,q - 1}}} \right\|_{{\mathcal{C}^0}}^2  + \lambda _q^{\alpha  - 1} {\left\| {{w_{q - 1}}} \right\|_{{\mathcal{C}^0}}}{{\left\| {{u_{q - 1}}} \right\|}_{{\mathcal{C}^1}}} \\
\lesssim & {\delta _{q + 1}}  .
\end{align*}
Thus
\begin{equation}
\left| {\mathfrak{E}_{q - 1}^{osc}} \right| \lesssim {\delta _{q + 1}} .  \label{RSRI-P-mene-GAP-proof-2}
\end{equation}
\par
It follows from \eqref{OS-SBE-C-2}, \eqref{BCBSJFHIG-SFG}, \eqref{E-OS-SBE-1}, \eqref{E-OS-SBE-2} and \eqref{IJSUHD-A-B-2} that
\begin{align*}
& \left| {\left( {\mathcal{\hat W}_q^{loc},\mathcal{\hat W}_{q - 1} - {{\hat u}_{q - 1}} -  w_{q - 1} * {\varsigma _{{{\hat \kappa }_q}}} - {u_{q - 1}} * \left( {{\varsigma _{{{\hat \kappa }_q}}} - {\varsigma _{{{\hat \kappa }_{q - 1}}}}} \right)} \right)_{\mathbb{T}^3}} \right|  \\
\lesssim & \left| {{{\left( {\nabla \left( {\mathcal{\hat W}_{q - 1} - {{\hat u}_{q - 1}}} \right),\mathcal{R}\mathcal{\hat W}_q^{loc}} \right) }_{{\mathbb{T}^3}}}} \right| + \left| {{{\left( {\mathcal{R}\left( { w_{q - 1} * {\varsigma _{{{\hat \kappa }_{q }}}}} \right),\nabla \mathcal{\hat W}_q^{loc}} \right) }_{{\mathbb{T}^3}}}} \right| + {{\hat \kappa }_{q - 1}}{\left\| {\mathcal{\hat W}_q^{loc}} \right\|_{{\mathcal{C}^0}}}{\left\| {{u_{q - 1}}} \right\|_{{\mathcal{C}^1}}}  \\
\lesssim & \lambda _q^{\alpha  - 1}\left( {{{\left\| {\mathcal{\hat W}_q^{loc}} \right\|}_{{\mathcal{C}^0}}} + {{\left\| {\left( { w_{q - 1} * {\varsigma _{{{\hat \kappa }_{q }}}}} \right)} \right\|}_{{\mathcal{C}^0}}}} \right)\left( {{{\left\| {\mathcal{\hat W}_{q - 1}} \right\|}_{{\mathcal{C}^1}}} + {{\left\| {{{\hat u}_{q - 1}}} \right\|}_{{\mathcal{C}^1}}}} \right) + {{\hat \kappa }_{q - 1}}{\left\| {\mathcal{\hat W}_q^{loc}} \right\|_{{\mathcal{C}^0}}}{\left\| {{u_{q - 1}}} \right\|_{{\mathcal{C}^1}}} \\
\lesssim & {\delta _{q + 1}}  ,
\end{align*}
and
\begin{align*}
& \left| {\left\| {{\mathcal{W}_{q - 1}} * \left( {{\varsigma _{{{\hat \kappa }_q}}} - {\varsigma _{{{\hat \kappa }_{q - 1}}}}} \right)} \right\|_{{L^2}}^2 + \left( {{\mathcal{W}_{q - 1}} * \left( {{\varsigma _{{{\hat \kappa }_q}}} - {\varsigma _{{{\hat \kappa }_{q - 1}}}}} \right),\mathcal{\hat W}_{q - 1} + \mathcal{\hat W}_q^{loc}} \right)_{\mathbb{T}^3}} \right| \\
\lesssim & \hat \kappa _{q - 1}^2\left\| {{\mathcal{W}_{q - 1}}} \right\|_{{\mathcal{C}^1}}^2 + {{\hat \kappa }_{q - 1}}\left( {{{\left\| {\mathcal{\hat W}_{q - 1}} \right\|}_{{\mathcal{C}^0}}} + {{\left\| {\mathcal{\hat W}_q^{loc}} \right\|}_{{\mathcal{C}^0}}}} \right){\left\| {{\mathcal{W}_{q - 1}}} \right\|_{{\mathcal{C}^1}}} \\
\lesssim & {\delta _{q + 1}}  .
\end{align*}
Furthermore, we have
\begin{align*}
&\mathbb{E} {\left| {\int_{0}^{t_l \wedge  \mathfrak{t}} {\left( {{{\hat v}_{q - 1}},{\mathrm{d}}\mathcal{\hat W}_q^{loc} + {\mathrm{d}} {\mathcal{W}_{q - 1}} * \left( {{\varsigma _{{{\hat \kappa }_q}}} - {\varsigma _{{{\hat \kappa }_{q - 1}}}}} \right)} \right)_{\mathbb{T}^3}} } \right|^2} \\
& + \mathbb{E}{\left| {\int_{0}^{t_l \wedge  \mathfrak{t}} {\left( { w_{q - 1}  * {\varsigma _{{{\hat \kappa }_{q }}}}  + {v_{q - 1}} * \left( {{\varsigma _{{{\hat \kappa }_q}}} - {\varsigma _{{{\hat \kappa }_{q - 1}}}}} \right),{\mathrm{d}}{\mathcal{\hat W}_{q}}} \right)_{\mathbb{T}^3}} } \right|^2}  \\
\lesssim & \lambda _q^{ - \frac{4}{3} + 2\alpha }  \left\| {\mathcal{R}\left( {\cos \left( {{\lambda _q}{x_3}} \right),\sin \left( {{\lambda _q}{x_3}} \right),0} \right)} \right\|_{{\mathcal{C}^0}}^2\left\| {{{\hat v}_{q - 1}}} \right\|_{{\mathcal{C}^1}}^2  + \hat \kappa _{q - 1}^2{\left( {\sum\limits_{j \leqslant q - 1} {\lambda _j^{\frac{1}{3} + \alpha }} } \right)^2}\left\| {{v_{q - 1}}} \right\|_{{\mathcal{H}^0}}^2 \\
& + \left\| {\mathcal{R}\hat w_{q - 1}} \right\|_{\mathcal{H}}^2\sum\limits_{j \leqslant q} { \lambda _j^{ - \frac{4}{3} + 2\alpha }\left\| {\left( {\cos \left( {{\lambda _{j }}{x_3}} \right),\sin \left( {{\lambda _{j }}{x_3}} \right),0} \right)} \right\|_{\mathcal{H}^1}^2}  + \hat \kappa _{q - 1}^2\left\| {{v_{q - 1}}} \right\|_{\mathcal{H}^1}^2  \\
\lesssim & {\delta _q}\lambda _q^{ - 2} + {\delta _{q - 1}}\lambda _{q - 1}^2\left( {\lambda _q^{ - \frac{{10}}{3} + 2\alpha } + \hat \kappa _{q - 1}^2} \right) .
\end{align*}
Therefore,
\begin{equation}
\left| {\mathbb{E} \mathfrak{E}_{q - 1}^{sto}} \right| \lesssim {\delta _{q + 1}}. \label{RSRI-P-mene-GAP-proof-3}
\end{equation}
\par
Owing to \eqref{IJSUHD-A-B-6}, we obtain
\begin{equation*}
\left| {\left( {{u_{q - 1}} * \left( {{\varsigma _{{{\hat \kappa }_q}}} - {\varsigma _{{{\hat \kappa }_{q - 1}}}}} \right),{u_{q - 1}} * \left( {{\varsigma _{{{\hat \kappa }_q}}} + {\varsigma _{{{\hat \kappa }_{q - 1}}}}} \right)} \right)_{\mathbb{T}^3}} \right| \lesssim {{\hat \kappa }_{q - 1}}{\left\| {{u_{q - 1}}} \right\|_{{\mathcal{C}^1}}} \lesssim {\delta _{q + 1}}  ,
\end{equation*}
which implies that
\begin{equation}
\left| {\mathfrak{E}_{q - 1}^{rem}} \right| \lesssim {\delta _{q + 1}}  . \label{RSRI-P-mene-GAP-proof-4}
\end{equation}
By virtue of the stationary phase lemma and \eqref{IJUHCJD-D-D-A-3}, we deduce that
\begin{equation*}
\left| {\mathrm{Tr} \left( {\sum\limits_{{\mathbf{k}} \ne  - {\mathbf{k}}'} {{\chi _l}{\chi _{l'}}\int_{{\mathbb{T}^3}} {w_{{\mathbf{k}}l}\left( x \right) \otimes w_{{\mathbf{k}}'l'}^u\left( x \right)} \, {\mathrm{d}}x} } \right)} \right| \lesssim \lambda _q^{ - 2}\left\| {a_{{\mathbf{k}}l}^u} \right\|_{{\mathcal{C}^1}}^2 \lesssim {\delta _{q + 1}}  .
\end{equation*}
Since ${{\bar R}_l}$ belongs to the traceless class, we have $\mathrm{Tr} \left( {\sum\limits_{l \in \left[ {0,{\mu _q}} \right]} {\chi _l^2\int_{{\mathbb{T}^3}} {{{\bar R}_l}\left( x \right)} \, {\mathrm{d}}x} } \right) = 0$. Then
\begin{equation}
\left| {\mathrm{Tr}\, \mathfrak{E}_{q - 1}^{Rey}} \right| \lesssim {\delta _{q + 1}} .  \label{RSRI-P-mene-GAP-proof-5}
\end{equation}
\par
Recall that ${\delta _{q + 2}} \sim {a^{ - {b^{q + 1}}\left( {b - 1} \right)}}{\delta _{q + 1}} \leqslant \sigma _q  {\delta _{q + 1}}$ with $1 > \mathop {\lim }\limits_{q \to  + \infty } {\sigma _q} = 0 $. Combining \eqref{RSRI-P-mene-GAP-proof-2}-\eqref{RSRI-P-mene-GAP-proof-5}, we infer that
\begin{equation*}
\left[ {\frac{1}{2}\left( {1 - \sigma _ q} \right)\mathfrak{M} - {C}} \right]{\delta _{q + 1}} \leqslant {\rho _{q,l}}  \leqslant \left( {C + \frac{\mathfrak{M}}{2}} \right){\delta _{q + 1}}  . 
\end{equation*}
We choose a large parameter $\mathfrak{M}$ such that
\begin{equation*}
0 < \frac{1}{2}\left( {1 - \sigma _ q} \right)\mathfrak{M} - {C} \leqslant \left( {C + \frac{\mathfrak{M}}{2}} \right) ,
\end{equation*}
which completes the proof. \qed
%%%%%%%%%%%%%%%%%%%%%%%%%%%%%%%%%%%%%%%%%%%%%%%%%%%%%%%%%%%%%%%%
%%%%%%%%%%%%%%%%%%%%%%%%%%%%%%%%%%%%%%%%%%%%%%%%%%%%%%%%%%%%%%%%
%%%%%%%%%%%%%%%%%%%%%%%%%%%%%%%%%%%%%%%%%%%%%%%%%%%%%%%%%%%%%%%%
\section{The existence of dissipative weak solutions}
In this section, we prove Theorem \ref{CTSE-O-C} in two steps. The first step is to construct the probabilistically strong and analytically weak solutions to the three-dimensional stochastic Euler equations \eqref{BouE1} and establish their uniform boundedness. The second step is to verify the strict energy inequality \eqref{SHUDJ-NENH-B-SIDIS}.
\par
\vspace{1em}
{\textbf{Step 1. The existence of probabilistically strong and analytically weak solutions}}. 
\par
It follows from \eqref{E-OS-SBE-1}, \eqref{E-OS-SBE-2} and \eqref{IJSUHD-A-B-2} that
\begin{equation*}
{\left\| {{w_{q - 1}} + \mathcal{W}_q^{loc}} \right\|_{{\mathcal{C}^0}}} + \lambda _q^{ - 1}{\left\| {{w_{q - 1}} + \mathcal{W}_q^{loc}} \right\|_{{\mathcal{C}^1}}} \lesssim  \delta _q^{\frac{1}{2}} .
\end{equation*}
Then, employing the interpolation inequality, we obtain
\begin{equation*}
{\left\| {{u_q} - {u_{q - 1}}} \right\|_{{\mathcal{C}^\gamma }}} \lesssim  \left\| {{w_{q - 1}} + \mathcal{W}_q^{loc}} \right\|_{{\mathcal{C}^0}}^{1 - \gamma }\left\| {{w_{q - 1}} + \mathcal{W}_q^{loc}} \right\|_{{\mathcal{C}^1}}^\gamma  \lesssim  \delta _q^{\frac{1}{2}}\lambda _q^\gamma  .
\end{equation*}
Choosing $b \to 1$ and $c \to \frac{5}{2}$ so that ${2 cb\gamma  < 1}$ for any $\gamma  \in \left[ {0,\frac{1}{5}} \right)$, we have
\begin{equation*}
\delta _q^{\frac{1}{2}}\lambda _q^\gamma  \lesssim 2{a^{\left( {cb\gamma  - \frac{1}{2}} \right){b^q}}}   \to 0 , \ \ \mbox{as}\ q \to \infty .
\end{equation*}
Therefore, the sequences ${\left\{ {{u_q}} \right\}_{q \in {Z^ + }}}$ is a Cauchy sequence in ${\mathcal{C}^\gamma }\left( {{\mathbb{T}^3}} \right)$. There exists a function $u\left( {t,x} \right) \in {\mathcal{C}^\gamma }\left( {{\mathbb{T}^3}} \right)$ such that
\begin{equation}
{u_q}\left( {t,x} \right) \to u\left( {t,x} \right) , \ \ \mbox{in}\ {\mathcal{C}^\gamma }\left( {{\mathbb{T}^3}} \right) , \ \ \forall \, t \in \left[ {0,T \wedge \mathfrak{t}} \right], \   \mbox{as}\ q \to  \infty .  \label{OSC-SBE-A-proof-a-3}
\end{equation}
\par
Using the fact that $\left( {v \cdot \nabla u,\psi } \right)_{\mathbb{T}^3} = - \left( {v \cdot \nabla \psi ,u} \right)_{\mathbb{T}^3}$ for any $\psi \left( x \right) \in {\mathcal{C}^\infty }\left( {{\mathbb{T}^3};\mathbb{R}} \right) \cap \mathcal{H}$, we derive that
\begin{align*}
& \left| {\int_0^{t \wedge  \mathfrak{t}} {\left( {{u_q}\left( \tau  \right) \cdot \nabla {u_q}\left( \tau  \right) - u\left( \tau  \right) \cdot \nabla u\left( \tau  \right),\psi } \right)_{\mathbb{T}^3}} \, \mathrm{d}\tau } \right| \\
\leqslant & \left| {\int_0^{t \wedge  \mathfrak{t}} {\left( {\left( {{u_q}\left( \tau  \right) - u\left( \tau  \right)} \right) \cdot \nabla {u_q}\left( \tau  \right),\psi } \right)_{\mathbb{T}^3}} \, \mathrm{d}\tau } \right| + \left| {\int_0^{t \wedge  \mathfrak{t}} {\left( {u\left( \tau  \right) \cdot \nabla \left( {{u_q}\left( \tau  \right) - u\left( \tau  \right)} \right),\psi } \right)_{\mathbb{T}^3}} \, \mathrm{d}\tau } \right| \\
\leqslant & \left| {\int_0^{t \wedge  \mathfrak{t}} {\left( {\left( {{u_q}\left( \tau  \right) - u\left( \tau  \right)} \right) \cdot \nabla \psi ,{u_q}\left( \tau  \right)} \right)_{\mathbb{T}^3}} \, \mathrm{d}\tau } \right| + \left| {\int_0^{t \wedge  \mathfrak{t}} {\left( {u\left( \tau  \right) \cdot \nabla \psi ,\left( {{u_q}\left( \tau  \right) - u\left( \tau  \right)} \right)} \right)_{\mathbb{T}^3}} \, \mathrm{d}\tau } \right|  \\
\leqslant & {\left\| \psi  \right\|_{\mathcal{H}^1}}\int_0^{t \wedge  \mathfrak{t}} {{{\left\| {{u_q}\left( \tau  \right) - u\left( \tau  \right)} \right\|}_{{\mathcal{C}^\gamma }}}\left( {{{\left\| {{u_q}\left( \tau  \right)} \right\|}_{\mathcal{H}^0}} + {{\left\| {u\left( \tau  \right)} \right\|}_{\mathcal{H}^0}}} \right)} \, \mathrm{d}\tau  .
\end{align*}
Hence, it follows from \eqref{OSC-SBE-A-proof-a-3} that
\begin{equation}
\left| {\int_0^{t \wedge  \mathfrak{t}} {\left( {{u_q}\left( \tau  \right) \cdot \nabla {u_q}\left( \tau  \right) - u\left( \tau  \right) \cdot \nabla u\left( \tau  \right),\psi } \right)_{\mathbb{T}^3}} \, \mathrm{d}\tau } \right| \to 0 ,\ \ \mbox{as}\ q \to  \infty .   \label{OSC-SBE-A-proof-a-5}
\end{equation}
Invoking the inductive estimates \eqref{IJSUHD-A-B-4} and the interpolation inequality, we obtain
\begin{equation*}
{\left\| {\mathring{R}_q} \right\|_{{\mathcal{C}^\gamma }}}  \lesssim \left\| {\mathring{R}_q} \right\|_{{\mathcal{C}^0}}^{1 - \gamma }\left\| {\mathring{R}_q} \right\|_{{\mathcal{C}^1}}^\gamma   \lesssim {\delta _{q + 1}}\lambda _q^{\gamma  - \alpha } .
\end{equation*}
Due to $c \leqslant \frac{1}{\gamma }$, we have $c \left( {\gamma - \alpha } \right) - 1 < 0$. Then ${\delta _{q + 1}}\lambda _q^{\gamma  - \alpha } \lesssim {a^{{b^{q + 1}}\left( {c \left( {\gamma - \alpha } \right) - 1} \right)}} \to 0$ as $q \to \infty$, which implies that
\begin{equation}
\mathring{R}_q  \left( t  \right) \to 0  , \ \ \mbox{in}\ {\mathcal{C}^\gamma }\left( {{\mathbb{T}^3}} \right) , \ \forall \, t \in \left[ {0,T \wedge \mathfrak{t}} \right] ,\ \mbox{as}\ q \to  \infty .  \label{OSC-SBE-A-proof-a-7}
\end{equation}
\par
It follows from the definition \eqref{Ie-SBE-C-1}, inductive estimates \eqref{IJUHCJD-D-D-A-1}, \eqref{JUIDHI-AS-A}, \eqref{E-OS-SBE-1} and \eqref{E-OS-SBE-2} that
\begin{align*}
& {\left\| {{p_q} - {p_{q - 1}}} \right\|_{{\mathcal{C}^0}}} + \lambda _q^{ - 1}{\left\| {{p_q} - {p_{q - 1}}} \right\|_{{\mathcal{C}^1}}} \\
\lesssim & {\left\| {{w_{o,q - 1}}} \right\|_{{\mathcal{C}^0}}}{\left\| {{w_{c,q - 1}}} \right\|_{{\mathcal{C}^0}}} + {\left\| {{u_{q - 1}} - {{\bar u}_{q - 1}}} \right\|_{{\mathcal{C}^0}}}{\left\| {{w_{q - 1}}} \right\|_{{\mathcal{C}^0}}} + \left\| {{w_{o,q - 1}}} \right\|_{{\mathcal{C}^0}}^2 + \left\| {{w_{c,q - 1}}} \right\|_{{\mathcal{C}^0}}^2 \\
& + \lambda _q^{ - 1}\left( {{{\left\| {{w_{o,q - 1}}} \right\|}_{{\mathcal{C}^0}}}{{\left\| {{w_{c,q - 1}}} \right\|}_{{\mathcal{C}^1}}} + {{\left\| {{w_{c,q - 1}}} \right\|}_{{\mathcal{C}^0}}}{{\left\| {{w_{o,q - 1}}} \right\|}_{{\mathcal{C}^1}}} + {{\left\| {{w_{o,q - 1}}} \right\|}_{{\mathcal{C}^0}}}{{\left\| {{w_{o,q - 1}}} \right\|}_{{\mathcal{C}^1}}}} \right) \\
& + \lambda _q^{ - 1}\left( {{{\left\| {{u_{q - 1}} - {{\bar u}_{q - 1}}} \right\|}_{{\mathcal{C}^0}}}{{\left\| {{w_{q - 1}}} \right\|}_{{\mathcal{C}^1}}} + {{\left\| {{w_{q - 1}}} \right\|}_{{\mathcal{C}^0}}}{{\left\| {{u_{q - 1}} - {{\bar u}_{q - 1}}} \right\|}_{{\mathcal{C}^1}}} + {{\left\| {{w_{c,q - 1}}} \right\|}_{{\mathcal{C}^0}}}{{\left\| {{w_{c,q - 1}}} \right\|}_{{\mathcal{C}^1}}}} \right) \\
\lesssim & \delta _q ,
\end{align*}
which implies that
\begin{equation*}
{\left\| {{p_q} - {p_{q - 1}}} \right\|_{{\mathcal{C}^\gamma }}} \lesssim \left\| {{p_q} - {p_{q - 1}}} \right\|_{{\mathcal{C}^0}}^{1 - \gamma }\left\| {{p_q} - {p_{q - 1}}} \right\|_{{\mathcal{C}^1}}^\gamma  \lesssim \delta _q \lambda _q^\gamma  \to 0 , \ \ \mbox{as}\ q \to \infty .
\end{equation*}
Then, there exists a non-vanishing pressure $p\left( t,x \right) $ such that
\begin{equation*}
{p_q}\left( {t,x} \right) \to p\left( {t,x} \right), \ \ \text{as}\ q \to  \infty ,
\end{equation*}
which implies that
\begin{equation}
\left( {\nabla {p_q} \left( t  \right)  - \nabla p\left( t  \right) ,\psi } \right) _{\mathbb{T}^3}  \to 0, \ \ \forall \, \psi \left( x \right) \in {\mathcal{C}^\infty }\left( {{\mathbb{T}^3}} \right) , \ t \in \left[ {0,T \wedge \mathfrak{t}} \right] , \ \  \text{as}\ q \to  \infty  .   \label{OSC-SBE-A-proof-a-9}
\end{equation}
\par
Similarly, it follows from \eqref{IJSUHD-A-B-2} and the interpolation inequality that there exists a Wiener process $ {\mathcal{W}}\left( {t,x} \right) \in \mathscr{W}$ such that
\begin{equation}
{\mathcal{W}_q}\left( {t,x} \right) \to {\mathcal{W}}\left( {t,x} \right) , \ \ \mbox{in}\ {\mathcal{C}^\gamma }\left( {{\mathbb{T}^3}} \right),\ \forall\, t \in \left[ {0,T \wedge \mathfrak{t}} \right], \ \text{as}\ q \to  \infty .  \label{OSC-SBE-A-proof-a-10}
\end{equation}
Consequently, in view of \eqref{OSC-SBE-A-proof-a-3}-\eqref{OSC-SBE-A-proof-a-10}, the limit $\left( {u , p,\mathcal{W}} \right)$ fulfills the three-dimensional stochastic Euler equations \eqref{BouE1} in distribution.
\par
For any $\xi = \frac{1}{5} - \gamma \in \left( {0,\frac{1}{5}} \right]$, it follows from \eqref{OSC-SBE-A-proof-a-3} that
\begin{align*}
\mathop {\sup }\limits_{t \in \left[ {0,T \wedge \mathfrak{t}} \right]} \left\| {{u}\left( t  \right)} \right\|_{{\mathcal{C}^{\frac{1}{5} - \xi }}} & \lesssim  \mathop {\sup }\limits_{t \in \left[ {0,T \wedge \mathfrak{t}} \right]}  {\mathop {\sup }\limits_{q \in {\mathbb{Z}^ + }} \left\| {{u_q}\left( t  \right)} \right\|_{{\mathcal{C}^{\frac{1}{5} - \xi }}}}    \\
& \leqslant  C\left( {{{ \lambda }_1}} \right) + C \cdot {\left( {\mathop {\sup }\limits_{q \in {\mathbb{Z}^ + }} \sum\limits_{j = 1}^{q - 1} {\delta _{j + 1}^{\frac{1}{2}}\lambda _{j + 1}^{\frac{1}{5} - \xi }} } \right)}   \\
& \leqslant  C\left( {{{ \lambda }_1}} \right)  + {\mathop {\sup }\limits_{q \in {\mathbb{Z}^ + }} \sum\limits_{j = 1}^{q - 1} {{a^{{b^{j + 1}}\left( {\left( {\frac{1}{5} - \xi } \right)bc - \frac{1}{2}} \right)}}} } . 
\end{align*}
Due to $\xi  > \frac{1}{5} - \frac{1}{{2bc}} > 0$, there exists a positive constant ${C\left( {\xi } \right)}$ such that
\begin{equation}
\mathop {\sup }\limits_{t \in \left[ {0,T \wedge \mathfrak{t}} \right]} \left\| {{u}\left( t  \right)} \right\|_{{\mathcal{C}^{\frac{1}{5} - \xi }}}  \leqslant C\left( {{{ \lambda }_1}} \right) +  {\mathop {\sup }\limits_{q \in {\mathbb{Z}^ + }} \sum\limits_{j = 1}^{q - 1} {{a^{ - C\left( {\xi } \right){b^{j + 1}}}}} } .  \label{djufhx-asda-1}
\end{equation}
Since $b > 1$, there exists an integer ${N^*} > 1$ such that ${b^{j + 1}} \geqslant j + 1$ for any $j \geqslant {N^*}$. One then has ${a^{ - C\left( {\xi } \right){b^{j + 1}}}} \leqslant {a^{ - C\left( {\xi } \right)\left( {j + 1} \right)}}$ as $j \geqslant {N^*}$ and $a > 1$. When $q > {N^*} + 1$, we derive that
\begin{align}
\mathop {\lim }\limits_{q \to  + \infty } \sum\limits_{j = {N^*}}^{q - 1} {{a^{ - C\left( {\xi } \right){b^{j + 1}}}}} & \leqslant  \mathop {\lim }\limits_{q \to  + \infty } \sum\limits_{j = {N^*}}^{q - 1} {{a^{ - C\left( {\xi } \right)\left( {j + 1} \right)}}} \nonumber \\
& = \mathop {\lim }\limits_{q \to  + \infty } \frac{{{a^{ - C\left( {\xi } \right)\left( {{N^*} + 1} \right)}}}}{{1 - {a^{ - C\left( {\xi } \right)}}}}\left( {1 - {a^{ - C\left( {\xi } \right)\left( {q - {N^*}} \right)}}} \right) \nonumber \\
& = \frac{{{a^{ - C\left( {\xi } \right)\left( {{N^*} + 1} \right)}}}}{{1 - {a^{ - C\left( {\xi } \right)}}}} . \label{dasdasgkhic}
\end{align}
Hence, it follows from \eqref{djufhx-asda-1} and \eqref{dasdasgkhic} that
\begin{equation*}
\mathop {\sup }\limits_{t \in \left[ {0,T \wedge \mathfrak{t}} \right]} \left\| {{u}\left( t  \right)} \right\|_{{\mathcal{C}^{\gamma }}} \leqslant  C\left( {{{ \lambda }_1}} \right) +  {{\sum\limits_{j = 1}^{N - 1} {{a^{ - C\left( {\xi } \right){b^{j + 1}}}}} } + \frac{{{a^{ - C\left( {\xi } \right)\left( {N + 1} \right)}}}}{{1 - {a^{ - C\left( {\xi } \right)}}}}} .
\end{equation*}
\par
\vspace{1em}
{\textbf{Step 2. Dissipative property}}. 
\par
Denote $\mathcal{\tilde E}_{q}^v \left( t \right) \triangleq \left\| {v_{q}\left( t \right)} \right\|_{\mathcal{H}}^2 $. We have $\mathcal{\tilde E}_q^v\left( t \right) - \mathcal{\tilde E}_{q - 1}^v\left( t \right) = \left\| {{w_{q - 1}}\left( t \right)} \right\|_{\mathcal{H}}^2 + 2\left( {{v_{q - 1}},{w_{q - 1}}} \right)_{\mathbb{T}^3}$. It follows from \eqref{E-OS-SBE-1}, \eqref{E-OS-SBE-2} and \eqref{IJSUHD-A-B-6} that
\begin{equation*}
\mathop {\sup }\limits_{t \in \left[ {0,T \wedge \mathfrak{t}} \right]} \left( {\left\| {{w_{q - 1}}\left( t \right)} \right\|_{\mathcal{H}}^2  + 2\left| {\left( {{v_{q - 1}},{w_{q - 1}}} \right)_{\mathbb{T}^3}} \right| } \right) \lesssim \delta _q^{\frac{1}{2}} ,
\end{equation*}
which implies that
\begin{equation}\label{DIFFFISJA-A-1}
\left| {\mathcal{\tilde E}_q^v\left( t \right) - \mathcal{\tilde E}_{q - 1}^v\left( t \right)} \right|  \to 0 , \ \ \mbox{as}\ q \to \infty .
\end{equation}
Using the same argument as that for \eqref{OSC-SBE-A-proof-a-3}, we deduce that
\begin{equation}\label{DIFFFISJA-A-2}
{{v_{q}}\left( t,x  \right)} \to  {v\left( t,x  \right)},\ \ \mbox{in}\ {\mathcal{C}^\gamma }\left( {{\mathbb{T}^3}} \right),\ \ \forall\, t \in \left[ {0,T \wedge \mathfrak{t}} \right], \ \text{as}\ q \to  \infty ,
\end{equation}
where ${{v}\left( t,x  \right)}  = {{u}\left( t,x  \right)}  - {\mathcal{W}}\left( {t,x} \right) $. Combining \eqref{SJKFJdafj-sagfg}, \eqref{RSRI-P-mene-GAP}, \eqref{DIFFFISJA-A-1} and \eqref{DIFFFISJA-A-2}, there exits a function $\mathcal{\tilde E}^ v \left( t \right) = \left\| {v\left( t \right)} \right\|_{\mathcal{H}}^2$ such that
\begin{equation*}
{\mathfrak{e}} \left( t \right) < {\mathfrak{e}} \left( t_l \right)  \lesssim \mathcal{\tilde E}_q^v \left( t_l \right)  + {\delta _{q + 1}} \to {{ \mathcal{\tilde E}}^v} \left( t_l \right)  < +  \infty  , \ \ \forall \, t \in \left[ {0,T} \right] ,\   \mbox{as}\ q \to  \infty  .
\end{equation*}
This shows that the function ${\mathfrak{e}}\left( t \right)$ is well-defined as $q \to  \infty$.
\par
Employing the It\^{o} formula, we obtain
\begin{align*}
\mathbb{E}\left\| {{{\hat u}_{q - 1}}\left( {t \wedge \mathfrak{t}} \right)} \right\|_{{L^2}}^2 = & \mathbb{E}\left\| {{{\hat u}_{q - 1}}\left( {{t_l} \wedge \mathfrak{t}} \right)} \right\|_{{L^2}}^2 + 2\mathbb{E}\int_{{t_l} \wedge \mathfrak{t}}^{t \wedge \mathfrak{t}} {{{\left( {{{\hat u}_{q - 1}},{A_u}} \right)}_{{\mathbb{T}^3}}}} {\text{d}}\tau   \\
&  - 2\mathbb{E}\int_{{t_l} \wedge \mathfrak{t}}^{t \wedge \mathfrak{t}} {{{\left( {\nabla {{\hat u}_{q - 1}},{{\hat R}_{q - 1}}} \right)}_{{\mathbb{T}^3}}}} {\text{d}}\tau  + \mathbb{E}\left( {t \wedge \mathfrak{t} - {t_l} \wedge \mathfrak{t}} \right) ,
\end{align*}
where ${A_u} \triangleq {\text{div}}\left( {{{\hat u}_{q - 1}} \otimes {{\hat u}_{q - 1}}} \right) - {\text{div}}\left( {{u_{q - 1}} \otimes {u_{q - 1}}} \right)*{\varsigma _{{{\hat \kappa }_{q - 1}}}}$. Then
\begin{align*}
& \mathbb{E}\left\| {{{\hat u}_{q - 1}}\left( {t \wedge \mathfrak{t}} \right)} \right\|_{{L^2}}^2 - \mathbb{E}\left\| {{{\hat u}_{q - 1}}\left( {{t_l} \wedge \mathfrak{t}} \right)} \right\|_{{L^2}}^2 + \mathbb{E}\left\| {{{ \mathcal{\hat W}}_{q - 1}}\left( {{t_l} \wedge \mathfrak{t}} \right)} \right\|_{{L^2}}^2 - \mathbb{E}\left\| {{{ \mathcal{\hat W}}_{q - 1}}\left( {t \wedge \mathfrak{t}} \right)} \right\|_{{L^2}}^2 \\
= & 2\mathbb{E}\int_{{t_l} \wedge \mathfrak{t}}^{t \wedge \mathfrak{t}} {{{\left( {{{\hat u}_{q - 1}},{A_u}} \right)}_{{\mathbb{T}^3}}}} {\text{d}}\tau  - 2\mathbb{E}\int_{{t_l} \wedge \mathfrak{t}}^{t \wedge \mathfrak{t}} {{{\left( {\nabla {{\hat u}_{q - 1}},{{\hat R}_{q - 1}}} \right)}_{{\mathbb{T}^3}}}} {\text{d}}\tau  .
\end{align*}
Owing to ${\left\| {{A_u}} \right\|_{{\mathcal{C}^0}}} \lesssim \hat \kappa _q^2\left\| {{{\hat u}_q}} \right\|_{{\mathcal{C}^1}}^2 \lesssim \delta _q^{\frac{1}{2}}$, we have
\begin{equation*}
\left| {\sum\limits_{l \in \left[ {0,\left\lceil {{\mu _q}T} \right\rceil } \right]} {\chi _l^2\mathbb{E}\int_{{t_l} \wedge \mathfrak{t}}^{t \wedge \mathfrak{t}} {{{\left( {{{\hat u}_q},{A_u}} \right)}_{{\mathbb{T}^3}}}} {\text{d}}\tau } } \right| \lesssim \delta _q^{\frac{1}{2}}  .
\end{equation*}
Notice that
\begin{equation*}
\left| {\sum\limits_{l { \in \left[ {0,\left\lceil {{\mu _q}T} \right\rceil } \right]}} {\chi _l^2 \mathbb{E} \int_{{t_l} \wedge \mathfrak{t}}^{t \wedge \mathfrak{t}} {{{\left( {\nabla {{\hat u}_{q - 1}},{{\hat R}_{q - 1}}} \right)}_{{\mathbb{T}^3}}}} {\text{d}}\tau } } \right| \lesssim \mu _q^{ - 1}\delta _{q - 1}^{\frac{1}{2}}{\delta _q}\lambda _{q - 1}^{1 - \alpha }  .
\end{equation*}
We then infer that
\begin{align*}
& \left| {\sum\limits_{l { \in \left[ {0,\left\lceil {{\mu _q}T} \right\rceil } \right]}} {\chi _l^2  \mathbb{E}\left\| {{{\hat u}_{q - 1}}\left( {t \wedge \mathfrak{t}} \right)} \right\|_{{L^2}}^2 - \mathbb{E}\left\| {{{\hat u}_{q - 1}}\left( {{t_l} \wedge \mathfrak{t}} \right)} \right\|_{{L^2}}^2 + \mathbb{E}\left\| {{{ \mathcal{\hat W}}_{q - 1}}\left( {{t_l} \wedge \mathfrak{t}} \right)} \right\|_{{L^2}}^2 - \mathbb{E}\left\| {{{ \mathcal{\hat W}}_{q - 1}}\left( {t \wedge \mathfrak{t}} \right)} \right\|_{{L^2}}^2}} \right|  \\
\lesssim & \delta _q^{\frac{1}{2}} + \mu _q^{ - 1}\delta _{q - 1}^{\frac{1}{2}}{\delta _q}\lambda _{q - 1}^{1 - \alpha }  .
\end{align*}
Denote ${\rho _{q - 1}} = 3\left| {{\mathbb{T}^3}} \right|\sum\limits_{l { \in \left[ {0,\left\lceil {{\mu _q}T} \right\rceil } \right]}} {\chi _l^2{\rho _{q - 1,l}}} $. Therefore, using the fact that
\begin{equation*}
\left| {\sum\limits_{l { \in \left[ {0,\left\lceil {{\mu _q}T} \right\rceil } \right]}} {\chi _l^2\left( {\mathfrak{e}\left( t \right) - \mathfrak{e}\left( {{t_l}} \right)} \right)} } \right| \lesssim \mu _q^{ - 1}  ,
\end{equation*}
we obtain
\begin{equation}
\left| {\mathfrak{e} - \left( {{\rho _{q - 1}} + \mathbb{E}\left\| {{{\hat u}_{q - 1}}} \right\|_{{L^2}}^2 - \mathbb{E}\left\| {{{ \mathcal{\hat W}}_{q - 1}}} \right\|_{{L^2}}^2} \right)} \right| \lesssim  \delta _q^{\frac{1}{2}} + \mu _q^{ - 1} + \mu _q^{ - 1}\delta _{q - 1}^{\frac{1}{2}}{\delta _q}\lambda _{q - 1}^{1 - \alpha } . \label{KSJFKAH-FAG-AF-1}
\end{equation}
\par
Since
\begin{equation*}
\left| {\mathbb{E}\left\| {{{\hat u}_q}} \right\|_{{L^2}}^2 - \mathbb{E}\left\| {{{\hat u}_{q - 1}}} \right\|_{{L^2}}^2 + \mathbb{E}\left\| {{{ \mathcal{\hat W}}_{q - 1}}} \right\|_{{L^2}}^2 - \mathbb{E}\left\| {{{ \mathcal{\hat W}}_q}} \right\|_{{L^2}}^2} \right| \lesssim \mathbb{E}{\left\| {{w_q}} \right\|_{{L^2}}} + \mathbb{E}{\left\| {\mathcal{W}_q^{loc}} \right\|_{{L^2}}} ,
\end{equation*}
it follows from \eqref{E-OS-SBE-1}, \eqref{E-OS-SBE-2} and \eqref{RSRI-P-mene-GAP} that
\begin{equation}
\left| {\mathbb{E}\left\| {{{\hat u}_q}} \right\|_{{L^2}}^2 - \mathbb{E}\left\| {{{ \mathcal{\hat W}}_q}} \right\|_{{L^2}}^2 - \left( {{\rho _{q - 1}} + \mathbb{E}\left\| {{{\hat u}_{q - 1}}} \right\|_{{L^2}}^2 - \mathbb{E}\left\| {{{ \mathcal{\hat W}}_{q - 1}}} \right\|_{{L^2}}^2} \right)} \right| \lesssim  \delta _q^{\frac{1}{2}}  .  \label{KSJFKAH-FAG-AF-2}
\end{equation}
Combining \eqref{KSJFKAH-FAG-AF-1}, \eqref{KSJFKAH-FAG-AF-2} and the fact that 
\begin{equation*}
\left| {\mathbb{E}\left\| {{{\hat u}_q}} \right\|_{{L^2}}^2 - \mathbb{E}\left\| {{u_q}} \right\|_{{L^2}}^2 + \mathbb{E}\left\| {{\mathcal{W}_q}} \right\|_{{L^2}}^2 - \mathbb{E}\left\| {{{ \mathcal{\hat W}}_q}} \right\|_{{L^2}}^2} \right| \lesssim \delta _q^{\frac{1}{2}} ,
\end{equation*}
we derive that
\begin{align*}
& \left| {\mathfrak{e} - \mathbb{E}\left\| {{u_q}} \right\|_{{L^2}}^2 + \mathbb{E}\left\| {{\mathcal{W}_q}} \right\|_{{L^2}}^2} \right| \\
\lesssim & \left| {\mathfrak{e} - \left( {{\rho _{q - 1}} + \mathbb{E}\left\| {{{\hat u}_{q - 1}}} \right\|_{{L^2}}^2 - \mathbb{E}\left\| {{{\mathcal{\hat W}}_{q - 1}}} \right\|_{{L^2}}^2} \right)} \right| \\
& + \left| {\mathbb{E}\left\| {{{\hat u}_q}} \right\|_{{L^2}}^2 - \mathbb{E}\left\| {{{\mathcal{\hat W}}_q}} \right\|_{{L^2}}^2 - \left( {{\rho _{q - 1}} + \mathbb{E}\left\| {{{\hat u}_{q - 1}}} \right\|_{{L^2}}^2 - \mathbb{E}\left\| {{{\mathcal{\hat W}}_{q - 1}}} \right\|_{{L^2}}^2} \right)} \right|  \\
& + \left| {\mathbb{E}\left\| {{{\hat u}_q}} \right\|_{{L^2}}^2 - \mathbb{E}\left\| {{u_q}} \right\|_{{L^2}}^2 + \mathbb{E}\left\| {{\mathcal{W}_q}} \right\|_{{L^2}}^2 - \mathbb{E}\left\| {{{\mathcal{\hat W}}_q}} \right\|_{{L^2}}^2} \right|  \\
\lesssim &  \delta _q^{\frac{1}{2}} +  \mu _q^{ - 1} + \mu _q^{ - 1}\delta _{q - 1}^{\frac{1}{2}}{\delta _q}\lambda _{q - 1}^{1 - \alpha } ,
\end{align*}
which implies that
\begin{equation}
{\mathfrak{e}}\left( t \right) - \mathbb{E} \left\| {u\left( t \wedge  \mathfrak{t} \right)} \right\|_{{L^2}}^2 + {\mathrm{Tr}} \,  Q  \cdot \mathbb{E}   \left( {t \wedge  \mathfrak{t} } \right) = 0 , \ \ \forall \, t \in \left[ {0,T} \right] .  \label{xiuh-IJHYDJFNVHF}
\end{equation}
Consequently, it follows from \eqref{xiuh-IJHYDJFNVHF} that the strict energy inequality \eqref{SHUDJ-NENH-B-SIDIS} holds. The proof is complete. \qed
%%%%%%%%%%%%%%%%%%%%%%%%%%%%%%%%%%%%%%%%%%%%%%
%%%%%%%%%%%%%%%%%%%%%%%%%%%%%%%%%%%%%%%%%%%%%%
%%%%%%%%%%%%%%%%%%%%%%%%%%%%%%%%%%%%%%%%%%%%%%
%%%%%%%%%%%%%%%%%%%%%%%%%%%%%%%%%%%%%%%%%%%%%%
\section{Regularity threshold for dissipative weak solutions}
In this section, we focus on the regularity threshold for dissipative weak solutions to the three-dimensional stochastic Euler equations \eqref{BouE1}.
%%%%%%%%%%%%%%%%%%%%%%%%%%%%%%%%%%%%%%%%%%%%%%%%%%%
\subsection{Kolmogorov 4/5 law}
For the self-similar, homogeneous and isotropic turbulence, Kolmogorov \cite{Kol1941a,Kol1941b,Kol1941c,Kol1941d} sparked considerable interest in the functional relationship between the $p$-th order moment of increment $u\left( {t,x + h} \right) - u\left( {t,x} \right)$ and the shift length $\left| h \right|$, and originally proposed the Kolmogorov's theory of turbulence. The most impressive one, called Kolmogorov 4/5 law, provides the exact equality (also see \cite{TurbulenceFrisch})
\begin{equation*}
\mathbb{E}{\left( {\left( {u\left( {x + h} \right) - u\left( x \right)} \right) \cdot \frac{h}{{\left| h \right|}}} \right)^3} =  - \frac{4}{5}\varepsilon \left| h \right| , \ \ \left| h \right| \in \left[ {{l_D},{l_I}} \right] , 
\end{equation*}
where $\varepsilon > 0$ is the dissipation rate of energy, and $\left[ {{l_D},{l_I}} \right]  \subset \mathbb{R}^+$ constitutes the inertial range. The Kolmogorov 4/5 law has been rigorously established for the stochastic Navier--Stokes and the deterministic Euler equations (see \cite{MR3943488,Zbl07892740,EyinkLocal,arXiv:2304.14470,Zbl1553.76041,MR4260784}). However, the rigorous proof of the Kolmogorov 4/5 law for the three-dimensional stochastic Euler equations \eqref{BouE1} remains an open problem. Therefore, we first establish the Kolmogorov 4/5 law of dissipative solutions for the three-dimensional stochastic Euler equations \eqref{BouE1}.
\par
For the dissipative, probabilistically strong and analytically weak solution ${u}\left( {t,x} \right)$ of Eq.\eqref{BouE1}, we denote ${\mathscr{T}_h}u  \left( {t,x} \right) \triangleq u  \left( {t,x + h} \right)$,
\begin{equation*}
\Gamma  \left( {t,h} \right) \triangleq \int_{{\mathbb{T}^3}} {u  \left( {t,x} \right) \otimes {\mathscr{T}_h}u  \left( {t,x} \right)} \, \mathrm{d}x,
\end{equation*}
\begin{align*}
{\mathscr{D}_i}  \left( {t,h} \right) \triangleq  \int_{{\mathbb{T}^3}} & \left[ { \left( {{\mathscr{T}_h}{u}\left( {t,x} \right) - {u}\left( {t,x} \right)} \right) \otimes \left( {{\mathscr{T}_h}{u}\left( {t,x} \right) - {u}\left( {t,x} \right)} \right)} \right] \\
& \left( {{\mathscr{T}_h}u_i \left( {t,x} \right) - u_i \left( {t,x} \right)} \right) \mathrm{d}x,\ \ i = 1, 2 , 3 ,
\end{align*}
and
\begin{equation*}
\mathscr{K}\left( h \right) \triangleq \int_{{\mathbb{T}^3}} {\sum\limits_{j \in \mathbb{Z}^+}  {{b_j^2} \lambda _j^{ - \frac{4}{3} } \left( {{\mathcal{G}_j}\left( x \right) \otimes {\mathscr{T}_h}{\mathcal{G}_j}\left( x \right)} \right)} } \, \mathrm{d}x ,
\end{equation*}
where ${\mathcal{G}_j}\left( x \right) \triangleq \left( {\cos \left( {{\lambda _j} {x_3}} \right),\sin \left( {{\lambda _j} {x_3}} \right),0} \right)$. Define the real matrix inner product
\begin{equation*}
{\left( {{a_{ij}}} \right)_{i,j \in {\mathbb{Z}^ + }}}:{\left( {{b_{ij}}} \right)_{i,j \in {\mathbb{Z}^ + }}} \triangleq \sum\limits_{i,j \in {\mathbb{Z}^ + }} {{a_{ij}}{b_{ij}}} , \ \  {for}\ {a_{ij}}\in {\mathbb{R}^1}, \ {b_{ij}} \in {\mathbb{R}^1}. 
\end{equation*}
Employing the same argument as Bedrossian, Coti Zelati, Punshon--Smith and Weber \cite[Proposition 3.1]{MR3943488}, it is straightforward to establish the K\'{a}rm\'{a}n-Howarth-Monin relationship as follows.
\par
\begin{lemma}\label{KHMR-LHS}
Let $\phi \left( {l} \right)$ and $\varphi \left( {l} \right)$ be smooth and compactly supported on $\left( {0,+ \infty } \right)$ satisfying
\begin{equation*}%\label{CSFOR3T}
\int_{0}^\infty {{l^2}\phi \left( l \right)} \mathrm{d}l < + \infty ,\ \ \ \  \int_{0}^\infty {{l^2}\varphi \left( l \right)} \mathrm{d}l < + \infty .
\end{equation*}
Suppose that
\begin{equation*}
\eta \left( h \right) = \phi \left( {\left| h \right|} \right)I + \varphi \left( {\left| h \right|} \right)\left( {\tilde h \otimes \tilde h} \right), \ \ \ \ \tilde h = {\left| h \right|^{ - 1}}h, \ h \in \mathbb{R}^3 .
\end{equation*}
Then, we have
\begin{align}
& \mathbb{E} \int_{{\mathbb{R}^3}} {\eta \left( h \right):\left( {\Gamma  \left( {t,h} \right) - \Gamma  \left( {s,h} \right)} \right)} \mathrm{d}h  + \frac{1}{2} \mathbb{E} \int_{s}^t {\int_{{\mathbb{R}^3}} {\sum\limits_{i = 1}^3 {\left[ {\frac{\partial }{{\partial {h_i}}}\eta \left( h \right):{\mathscr{D}_i } \left( {\tau ,h} \right)} \right]} } \mathrm{d}h} \, \mathrm{d}\tau   \nonumber \\
 = &  \int_{s}^t {\int_{{\mathbb{R}^3}} {\eta \left( h \right):\mathscr{K}\left( h \right)} \mathrm{d}h} \, \mathrm{d}\tau  ,\ \  \forall \, 0 \leqslant s < t \leqslant T . \label{KHMR-LHS-1}
\end{align}
\end{lemma}
\par
Since the regularity \eqref{djfkshhgss} implies the integrability condition $\mathbb{E}\int_{s \wedge  \mathfrak{t}}^{t \wedge  \mathfrak{t} } {\left\| {u\left( \tau  \right)} \right\|_{{L^3}}^3} \, \mathrm{d}\tau  <  + \infty $, we can define the averaged third-order longitudinal structure function
\begin{equation*}
{S_{\shortparallel}}\left( {t,l} \right) \triangleq \frac{1}{{4\pi }}\int_{{\mathbb{S}^2}} {\int_{{\mathbb{T}^3}} {{{\left( {\left( {u  \left( {t,x + {l \mathbf{n}}} \right) - u\left( {t,x} \right)} \right) \cdot \mathbf{n}} \right)}^3}} \mathrm{d}x} \mathrm{d}S\left( {\mathbf{n}} \right),
\end{equation*}
where ${\mathbb{S}^2}$ denotes the unit sphere with the outward normal vector ${\mathbf{n}}$, and the Kolmogorov scale $l $ is positive. The following result reveals that the strict energy inequality \eqref{SHUDJ-NENH-B-SIDIS} and the integrability condition provide the sufficient condition for establishing the Kolmogorov 4/5 law of the three-dimensional stochastic Euler equations \eqref{BouE1}.
\par
\begin{proposition}\label{LHS-3-4-and-5-4-L-L}
Suppose that the probabilistically strong and analytically weak solution ${u}\left( {t,x} \right)$ of Eq.\eqref{BouE1} satisfies the strict energy inequality \eqref{SHUDJ-NENH-B-SIDIS}. Then, we have
\begin{equation}\label{LHS-5-4-L-L-L}
\mathop {\lim }\limits_{l \to 0}  \mathbb{E} \int_{s \wedge  \mathfrak{t}}^{t \wedge  \mathfrak{t}} {\frac{{S_{\shortparallel}\left( {\tau ,l} \right)}}{l}} \, \mathrm{d}\tau   = - \frac{4}{5}\varepsilon ,\ \  \forall \, 0 \leqslant s < t \leqslant T   ,
\end{equation}
where $2 \varepsilon  = \mathbb{E}\left\| {u \left( s \wedge  \mathfrak{t} \right)} \right\|_{{L^2}}^2 - \mathbb{E}\left\| {u\left( t \wedge  \mathfrak{t} \right)} \right\|_{{L^2}}^2 + {\mathrm{Tr}} \,  Q \cdot \mathbb{E} \left( {t \wedge  \mathfrak{t} - s \wedge  \mathfrak{t}} \right) > 0$.
\end{proposition}
\par
{\noindent}{\bf{Proof}}. 
We first establish the classical Kolmogorov 4/3 law. The Kolmogorov 4/5 law is then derived from the Kolmogorov 4/3 law.
\par
\vspace{1em}
{\textbf{Step 1. Deriving the Kolmogorov 4/3 law}}.
\par
Define the averaged third-order longitudinal structure function
\begin{equation*}
{S_0}\left( {t,l} \right) \triangleq  \frac{1}{{4\pi }}\int_{{\mathbb{S}^2}} \int_{{\mathbb{T}^3}} {{\left| {u  \left( {t,x + {l \mathbf{n}}} \right) - u\left( {t,x} \right)} \right|}^2}  \left( {u  \left( {t,x + {l \mathbf{n}}} \right) - u\left( {t,x} \right)} \right) \cdot \mathbf{n} \mathrm{d}x \mathrm{d}S\left( {\mathbf{n}} \right) .
\end{equation*}
Then, choosing $\eta \left( h \right) = \phi \left( {\left| h \right|} \right)I$ in the K\'{a}rm\'{a}n-Howarth-Monin relationship \eqref{KHMR-LHS-1}, we deduce that
\begin{align}
& \mathbb{E} \int_{{\mathbb{R}^3}} {\phi \left( {\left| h \right|} \right)I :\left( {\Gamma \left( {t \wedge  \mathfrak{t} ,h} \right) - \Gamma   \left( {s \wedge  \mathfrak{t} ,h} \right)} \right)} \mathrm{d}h  + \frac{1}{2} \mathbb{E} \int_{s \wedge  \mathfrak{t}}^{t \wedge  \mathfrak{t}}  {\int_{{\mathbb{R}^3}} {\sum\limits_{i = 1}^3 {\left[ {\frac{\partial }{{\partial {h_i}}}\phi \left( {\left| h \right|} \right)I :{\mathscr{D}_i }\left( {\tau ,h} \right)} \right]} } \, \mathrm{d}h} \mathrm{d}\tau \nonumber \\
= & \mathbb{E} \int_{s \wedge  \mathfrak{t}}^{t \wedge  \mathfrak{t}}  {\int_{{\mathbb{R}^3}} {\phi \left( {\left| h \right|} \right)I:\mathscr{K}\left( h \right)} \, \mathrm{d}h} \mathrm{d}\tau . \label{LHS-3-4-L-L-1}
\end{align}
Denote
\begin{equation*}
{\bar \Gamma} \left( {t,l} \right) \triangleq \frac{1}{{4\pi }}\int_{{\mathbb{S}^2}} {I:{\Gamma }\left( {t,l {\mathbf{n}}} \right)} \mathrm{d}S\left( {\mathbf{n}} \right),
\end{equation*}
and
\begin{equation*}
\mathscr{\bar K}\left( l \right) \triangleq \frac{1}{{4\pi }}\int_{{\mathbb{S}^2}} {\sum\limits_{j \in \mathbb{Z}^+}  {{b_j^2} \lambda _j^{ - \frac{4}{3}  }\left( {{\mathcal{G}_j}\left( x \right),{\mathscr{T}_{l\mathbf{n}}}{\mathcal{G}_j}\left( x \right)} \right)} } \mathrm{d}S\left( {\mathbf{n}} \right).
\end{equation*}
It follows from \eqref{LHS-3-4-L-L-1} that
\begin{align*}
&2 \mathbb{E} \int_{{\mathbb{R}^ + }} {{l^2}\phi \left( l \right)\left( {{{\bar \Gamma} }\left( {t \wedge  \mathfrak{t} ,l} \right) - {{\bar \Gamma} }\left( {s \wedge  \mathfrak{t} ,l} \right)} \right)} \mathrm{d}l + \mathbb{E} \int_{s \wedge  \mathfrak{t}}^{t \wedge  \mathfrak{t}}  {\int_{{\mathbb{R}^ + }} {{l^2}{\partial _l}\phi \left( l \right)S_0\left( {\tau ,l} \right)} } \mathrm{d}l\mathrm{d}\tau   \\
= & 2 \mathbb{E} \int_{s \wedge  \mathfrak{t}}^{t \wedge  \mathfrak{t}}  {\int_{{\mathbb{R}^ + }} {{l^2}\phi \left( l \right)\mathscr{\bar K}\left( l \right)} } \mathrm{d}l\mathrm{d}\tau .
\end{align*}
Thus, the following identity holds in the sense of distributions
\begin{equation*}
2{l^2} \mathbb{E} \left( {{{\bar \Gamma} }\left( {t \wedge  \mathfrak{t} ,l} \right) - {{\bar \Gamma} }\left( {s \wedge  \mathfrak{t} ,l} \right)} \right) - \mathbb{E} \int_{s \wedge  \mathfrak{t}}^{t \wedge  \mathfrak{t}}  {{\partial _l}\left( {{l^2}S_0\left( {\tau ,l} \right)} \right)} \, \mathrm{d}\tau =  2{l^2} \mathbb{E} \int_{s \wedge  \mathfrak{t}}^{t \wedge  \mathfrak{t}}  { \mathscr{\bar K}\left( l \right)} \, \mathrm{d}\tau ,
\end{equation*}
which implies that
\begin{equation}
\mathbb{E} \int_{s \wedge  \mathfrak{t}}^{t \wedge  \mathfrak{t}}  {\frac{{S_0\left( {\tau ,l} \right)}}{l}} \, \mathrm{d}\tau   =   \frac{2}{l^3}\int_0^l {{r^2}\mathbb{E}{\left( {{{\bar \Gamma} }\left( {t \wedge  \mathfrak{t} ,r} \right) - {{\bar \Gamma} }\left( {s \wedge  \mathfrak{t} ,r} \right)} \right)}} \mathrm{d}r  - \frac{2}{l^3} \mathbb{E} \int_{s \wedge  \mathfrak{t}}^{t \wedge  \mathfrak{t}}  {\int_0^l {{r^2}\mathscr{\bar K}\left( r \right)} \mathrm{d}r} \, \mathrm{d}\tau . \label{LHS-3-4-L-L-4}
\end{equation}
\par
Consider the two-point correlation function of the velocity. Since 
\begin{equation*}
\mathbb{E}\left\| {u\left( t \wedge  \mathfrak{t} \right)} \right\|_\mathcal{H}^2 < \mathbb{E}\left\| {u\left( 0 \right)} \right\|_\mathcal{H}^2 + \mathrm{Tr} \, Q  \cdot \mathbb{E} \left( {t \wedge  \mathfrak{t}} \right) < + \infty ,
\end{equation*}
it follows from the H\"{o}lder inequality that
\begin{align}
\frac{2}{{l^3}} \mathbb{E} \left| \int_0^l {{r^2}{\left( {{{\bar \Gamma} }\left( {t \wedge  \mathfrak{t},r} \right) - {{\bar \Gamma} }\left( {t \wedge  \mathfrak{t},0} \right)} \right)}} \mathrm{d}r \right| \lesssim & \frac{2}{l^3}\int_0^{{l}} {{r^2}\mathbb{E}\left| {\left( {{u}\left( t \wedge  \mathfrak{t} \right),{\mathscr{T}_{r\mathbf{n}}}{u} \left( t \wedge  \mathfrak{t} \right) - {u}\left( t \wedge  \mathfrak{t} \right)} \right)} \right|} \mathrm{d}r  \nonumber \\
\lesssim & \frac{2}{{l^3}}\int_0^l {{r^2}{{\left( {\mathbb{E}\left\| {{u}\left( t \wedge  \mathfrak{t} \right) - {\mathscr{T}_{r\mathbf{n}}}{u}\left( t \wedge  \mathfrak{t} \right)} \right\|_{\mathcal{H}}^2} \right)}^{\frac{1}{2}}}} \mathrm{d}r    . \label{MNCYDHFU-1}
\end{align}
Using the translation continuity of ${u} \in \mathcal{H}$, we have
\begin{equation*}
{\left\| {u \left( {t \wedge \mathfrak{t}} \right) - {\mathscr{T}_{r{\mathbf{n}}}}u \left( {t \wedge \mathfrak{t}} \right)} \right\|_\mathcal{H}} \to 0 , \ \ \mbox{as}\ r \to 0 , \ \mathbb{P}\mbox{-a.s.}
\end{equation*}
Applying the dominated convergence theorem, we conclude that
\begin{equation}
\mathbb{E}\left\| {u \left( {t \wedge \mathfrak{t}} \right) - {\mathscr{T}_{r{\mathbf{n}}}}u \left( {t \wedge \mathfrak{t}} \right)} \right\|_\mathcal{H}^2 \to 0 , \ \ \mbox{as}\ r \to 0. \label{MNCYDHFU-2}
\end{equation}
Combining \eqref{MNCYDHFU-1} and \eqref{MNCYDHFU-2}, we obtain
\begin{align}
& \mathop {\lim }\limits_{l \to 0}\frac{2}{{l^3}} \mathbb{E} \left| {\int_0^l {{r^2}{\left( {{{\bar \Gamma} }\left( {t \wedge  \mathfrak{t} ,r} \right) - {{\bar \Gamma} }\left( {t \wedge  \mathfrak{t},0} \right)} \right)}} \mathrm{d}r} \right| \nonumber \\
\lesssim & \mathop {\lim }\limits_{l \to 0} \mathop {\sup }\limits_{r \in \left[ {0,l} \right]} {\left( {\mathbb{E}\left\| {u \left( {t \wedge \mathfrak{t}} \right) - {\mathscr{T}_{r{\mathbf{n}}}}u \left( {t \wedge \mathfrak{t}} \right)} \right\|_\mathcal{H}^2} \right)^{\frac{1}{2}}} = 0   , \ \ \forall \, t \in \left[ {0,T} \right] .  \label{MNCYDHFU-3}
\end{align}
\par
Thanks to \eqref{MNCYDHFU-3}, we have
\begin{align}
\mathop {\lim }\limits_{l \to 0} \frac{2}{{l^3}}\int_0^l {{r^2}\mathbb{E}\left( {{{\bar \Gamma} }\left( {t \wedge  \mathfrak{t},r} \right) - {{\bar \Gamma} }\left( {s \wedge  \mathfrak{t},r} \right)} \right)} \mathrm{d}r = & \mathop {\lim }\limits_{l \to 0} \frac{2}{{l^3}}\int_0^l {{r^2}\mathbb{E}\left( {{{\bar \Gamma} }\left( {t \wedge  \mathfrak{t},r} \right) - {{\bar \Gamma} }\left( {t \wedge  \mathfrak{t},0} \right)} \right)} \mathrm{d}r \nonumber \\
& + \mathop {\lim }\limits_{l \to 0} \frac{2}{{l^3}}\int_0^l {{r^2}\mathbb{E}\left( {{{\bar \Gamma} }\left( {t \wedge  \mathfrak{t},0} \right) - {{\bar \Gamma} }\left( {s \wedge  \mathfrak{t},0} \right)} \right)} \mathrm{d}r  \nonumber \\
& +  \mathop {\lim }\limits_{l \to 0} \frac{2}{{l^3}}\int_0^l {{r^2}\mathbb{E}\left( {{{\bar \Gamma} }\left( {s \wedge  \mathfrak{t},0} \right) - {{\bar \Gamma} }\left( {s \wedge  \mathfrak{t},r} \right)} \right)} \mathrm{d}r   \nonumber \\
= & \frac{2}{{3}}\left( { \mathbb{E}\left\| {{u}\left( t \wedge  \mathfrak{t} \right)} \right\|_{\mathcal{H}}^2  -  \mathbb{E}\left\| {{u\left( s \wedge  \mathfrak{t} \right)}} \right\|_{\mathcal{H}}^2} \right). \label{LHS-3-4-L-L-9} 
\end{align}
Using the same argument as \eqref{LHS-3-4-L-L-9}, it is straightforward to derive that
\begin{equation}\label{LHS-3-4-L-L-7sdf5s}
\mathop {\lim }\limits_{l \to 0} \frac{2}{{{l^3}}} \mathbb{E} \int_{s \wedge  \mathfrak{t}}^{t \wedge  \mathfrak{t}}  {\int_0^l {{r^2} \mathcal{\bar K}\left( r \right)} {\text{d}}r} \, {\text{d}}\tau  = \frac{2}{3}  {\mathrm{Tr}} \,  Q \cdot \mathbb{E} \left( {t \wedge  \mathfrak{t} - s \wedge  \mathfrak{t}} \right) .
\end{equation}
Therefore, it follows from \eqref{LHS-3-4-L-L-4}, \eqref{LHS-3-4-L-L-9}, \eqref{LHS-3-4-L-L-7sdf5s} and the strict energy inequality \eqref{SHUDJ-NENH-B-SIDIS} that
\begin{align}
\mathop {\lim }\limits_{l \to 0} \mathbb{E} \int_{s \wedge  \mathfrak{t}}^{t \wedge  \mathfrak{t}}  {\frac{{S_0\left( {\tau ,l} \right)}}{l}} \, \mathrm{d}\tau  = &  - \frac{4}{{3}}\left( {\frac{1}{2}\mathbb{E}\left\| {{u}\left( s \wedge  \mathfrak{t} \right)} \right\|_{\mathcal{H}}^2 - \frac{1}{2}\mathbb{E}\left\| {{u}\left( t \wedge  \mathfrak{t} \right)} \right\|_{\mathcal{H}}^2  + \frac{1}{2}  {\mathrm{Tr}} \,  Q \cdot \mathbb{E} \left( {t \wedge  \mathfrak{t} - s \wedge  \mathfrak{t}} \right) } \right) \nonumber \\
= & - \frac{4}{3}{\varepsilon } .  \label{LHS-3-4-L-L-L}
\end{align}
\par
\vspace{1em}
{\textbf{Step 2. Verification of the Kolmogorov 4/5 law \eqref{LHS-5-4-L-L-L}}}. 
\par
Picking $\eta \left( h \right) = \varphi \left( {\left| h \right|} \right){\mathbf{n} \otimes \mathbf{n}}$ in the K\'{a}rm\'{a}n-Howarth-Monin relationship \eqref{KHMR-LHS-1}, we have
\begin{align*}
& \mathbb{E} \int_{{\mathbb{R}^3}} {\left( {\varphi \left( {\left| h \right|} \right){\mathbf{n} \otimes \mathbf{n}}} \right):\left( {{\Gamma }\left( {t \wedge  \mathfrak{t},h} \right) - {\Gamma }\left( {s \wedge  \mathfrak{t},h} \right)} \right)} \mathrm{d}h   \\
& + \frac{1}{2}\mathbb{E} \int_{s \wedge  \mathfrak{t}}^{t \wedge  \mathfrak{t}}  {\int_{{\mathbb{R}^3}} {\sum\limits_{i = 1}^3 {{\partial _{{h_i}}}\left( {\varphi \left( {\left| h \right|} \right){\mathbf{n} \otimes \mathbf{n}}} \right):\mathscr{D}_i^{{u}}} } \mathrm{d}h} \, \mathrm{d}\tau   \\
= & \mathbb{E} \int_{s \wedge  \mathfrak{t}}^{t \wedge  \mathfrak{t}}  {\int_{{\mathbb{R}^3}} {\left( {\varphi \left( {\left| h \right|} \right){\mathbf{n} \otimes \mathbf{n}}} \right):\mathscr{K}\left( h \right)} \mathrm{d}h} \, \mathrm{d}\tau . 
\end{align*}
Denote
\begin{equation*}
{\mathfrak{J}}\left( {t,l} \right) \triangleq \int_{{\mathbb{T}^3}} {\int_{{\mathbb{S}^2}} {\left( {\mathbf{n} \otimes \mathbf{n}} \right):\left( {{u}\left( {t,x} \right) \otimes {\mathscr{T}_{l\mathbf{n}}}{u}\left( {t,x} \right)} \right)} \mathrm{d}S\left( {\mathbf{n}} \right)} \, \mathrm{d}x ,
\end{equation*}
and
\begin{equation*}
\mathscr{\tilde K}\left( l \right) \triangleq \frac{1}{{4\pi }}\int_{{\mathbb{T}^3}} {\int_{{\mathbb{S}^2}} {\sum\limits_{j \in \mathbb{Z}^+}  {{b_j^2} \lambda _j^{ - \frac{4}{3} }\left( {\mathbf{n} \otimes \mathbf{n}} \right):{\left( {{\mathcal{G}_j}\left( x \right) \otimes {\mathscr{T}_{l\mathbf{n}}}{\mathcal{G}_j}\left( x \right)} \right)}} } \mathrm{d}S\left( {\mathbf{n}} \right)} \, \mathrm{d}x.
\end{equation*}
We derive that
\begin{align*}
& \mathbb{E} \int_{s \wedge  \mathfrak{t}}^{t \wedge  \mathfrak{t}}  {\int_{{\mathbb{R}^ + }} {\left( {{l^2}{\partial _l}\varphi \left( l \right) - 2l\varphi \left( l \right)} \right)S_{\shortparallel}\left( {\tau ,l} \right)} \mathrm{d}l} \, \mathrm{d}\tau \\
= & 2\mathbb{E} \int_{s \wedge  \mathfrak{t}}^{t \wedge  \mathfrak{t}}  {\int_{{\mathbb{R}^ + }} {{l^2}\varphi \left( l \right) \mathscr{\tilde K}\left( l \right)} \mathrm{d}l} \, \mathrm{d}\tau - 2\mathbb{E} \int_{s \wedge  \mathfrak{t}}^{t \wedge  \mathfrak{t}}  {\int_{{\mathbb{R}^ + }} {l\varphi \left( l \right)S_0\left( {\tau ,l} \right)} \mathrm{d}l} \, \mathrm{d}\tau \\
&  - \frac{1}{{2\pi }} \mathbb{E} \int_{{\mathbb{R}^ + }} {{l^2}{\varphi \left( l \right)}\left( {{\mathfrak{J}}\left( {t \wedge  \mathfrak{t},l} \right) - {\mathfrak{J}}\left( {s \wedge  \mathfrak{t},l} \right)} \right)} \mathrm{d}l .
\end{align*}
Notice that ${l^2}{\partial _l}\varphi \left( l \right) - 2l\varphi \left( l \right) = {l^4}{\partial _l}\left( {{l^{ - 2}}\varphi \left( l \right)} \right)$. Thus, the following identity holds in the sense of distributions
\begin{align*}
& - \mathbb{E} \int_{s \wedge  \mathfrak{t}}^{t \wedge  \mathfrak{t}}  {{l^{ - 2}}{\partial _l}\left( {{l^4} S_{\shortparallel}\left( {\tau ,l} \right)} \right)} \, \mathrm{d}\tau + 2 \mathbb{E} \int_{s \wedge  \mathfrak{t}}^{t \wedge  \mathfrak{t}}  {lS_0\left( {\tau ,l} \right)} \, \mathrm{d}\tau  \\
= &  2 \mathbb{E} \int_{s \wedge  \mathfrak{t}}^{t \wedge  \mathfrak{t}}  {{l^2}\mathscr{\tilde K}\left( l \right)} \, \mathrm{d}\tau - \frac{{l^2}}{{2\pi }}\mathbb{E} \left( {{\mathfrak{J}}\left( {t \wedge  \mathfrak{t},l} \right) - {\mathfrak{J}}\left( {s \wedge  \mathfrak{t},l} \right)} \right)  .
\end{align*}
Consequently,
\begin{align}
\mathbb{E} \int_{s \wedge  \mathfrak{t}}^{t \wedge  \mathfrak{t}}  {\frac{{S_{\shortparallel}\left( {\tau ,l} \right)}}{l}} \, \mathrm{d}\tau  = & \frac{1}{{2\pi {l^5}}}\int_0^l {{r^4}\mathbb{E}\left( {{\mathfrak{J}}\left( {t \wedge  \mathfrak{t},r} \right) - {\mathfrak{J}}\left( {s \wedge  \mathfrak{t},r} \right)} \right)} \mathrm{d}r \nonumber \\
& + \frac{2}{{l^5}}\mathbb{E} \int_{s \wedge  \mathfrak{t}}^{t \wedge  \mathfrak{t}}  {\int_0^l {{r^3}S_0\left( {\tau ,r} \right)} \mathrm{d}r} \, \mathrm{d}\tau - \frac{2}{{l^5}} \mathbb{E} \int_{s \wedge  \mathfrak{t}}^{t \wedge  \mathfrak{t}}  {\int_0^l {{r^4}\mathscr{\tilde K}\left( r \right)} \mathrm{d}r} \, \mathrm{d}\tau . \label{LHS-5-4-L-L-P-3}
\end{align}
\par
Using the same argument as \eqref{MNCYDHFU-3}, we obtain
\begin{equation*}
\mathop {\lim }\limits_{l \to 0} \mathbb{E} \left| {\frac{1}{{2\pi {l^5}}}\int_0^l {{r^4}\left( {{\mathfrak{J}}\left( {t \wedge  \mathfrak{t},0} \right) - {\mathfrak{J}}\left( {t \wedge  \mathfrak{t},r} \right)} \right)} \mathrm{d}r} \right| = 0 , \ \ \forall \, t \in \left[ {0,T } \right].
\end{equation*}
Then
\begin{align}
\mathop {\lim }\limits_{l \to 0} \frac{1}{{2\pi {l^5}}}\int_0^l {{r^4}\mathbb{E}\left( {{\mathfrak{J}}\left( {t \wedge  \mathfrak{t} ,r} \right) - {\mathfrak{J}}\left( {s \wedge  \mathfrak{t},r} \right)} \right)} \mathrm{d}r  = & \mathop {\lim }\limits_{l \to 0} \frac{1}{{2\pi {l^5}}}\int_0^l {{r^4}\mathbb{E}\left( {{\mathfrak{J}}\left( {t \wedge  \mathfrak{t},r} \right) - {\mathfrak{J}}\left( {t \wedge  \mathfrak{t},0} \right)} \right)} \mathrm{d}r  \nonumber \\
& + \mathop {\lim }\limits_{l \to 0} \frac{1}{{2\pi {l^5}}}\int_0^l {{r^4}\mathbb{E}\left( {{\mathfrak{J}}\left( {t \wedge  \mathfrak{t},0} \right) - {\mathfrak{J}}\left( {s \wedge  \mathfrak{t},0} \right)} \right)} \mathrm{d}r \nonumber \\
& + \mathop {\lim }\limits_{l \to 0} \frac{1}{{2\pi {l^5}}}\int_0^l {{r^4}\mathbb{E}\left( {{\mathfrak{J}}\left( {s \wedge  \mathfrak{t},0} \right) - {\mathfrak{J}}\left( {s \wedge  \mathfrak{t},r} \right)} \right)} \mathrm{d}r  \nonumber \\
= & \frac{2}{{15}}\left( {\mathbb{E}\left\| {{u}\left( t \wedge  \mathfrak{t} \right)} \right\|_{\mathcal{H}}^2 - \mathbb{E}\left\| {{u } \left( s \wedge  \mathfrak{t} \right)} \right\|_{\mathcal{H}}^2} \right) . \label{LHS-5-4-L-L-P-7} 
\end{align}
Invoking the Kolmogorov 4/3 law \eqref{LHS-3-4-L-L-L}, we have
\begin{equation}\label{LHS-5-4-L-L-P-8}
\mathop {\lim }\limits_{l \to 0} \frac{2}{{l^5}}\mathbb{E} \int_{s \wedge  \mathfrak{t}}^{t \wedge  \mathfrak{t}}  {\int_0^l {{r^3}S_0\left( {\tau ,r} \right)} \mathrm{d}r} \, \mathrm{d}\tau  = - \frac{8}{{15}}\varepsilon   .
\end{equation}
Furthermore, 
\begin{equation}\label{LHS-5-4-La1245fasd}
\mathop {\lim }\limits_{l \to 0} \frac{2}{{{l^5}}}\mathbb{E}  \int_{s \wedge  \mathfrak{t}}^{t \wedge  \mathfrak{t}}  {\int_0^l {{r^4} \mathcal{\tilde K}\left( r \right)} {\text{d}}r} \, {\text{d}}\tau  = \frac{2}{{15}} {\mathrm{Tr}} \,  Q \cdot \mathbb{E} \left( {t \wedge  \mathfrak{t} - s \wedge  \mathfrak{t}} \right)  .
\end{equation}
Consequently, it follows from \eqref{LHS-5-4-L-L-P-3}-\eqref{LHS-5-4-La1245fasd} and the strict energy inequality \eqref{SHUDJ-NENH-B-SIDIS} that
\begin{align*}
\mathop {\lim }\limits_{l \to 0} \mathbb{E} \int_{s \wedge  \mathfrak{t}}^{t \wedge  \mathfrak{t}}  {\frac{{S_{\shortparallel}\left( {\tau ,l} \right)}}{l}} \, \mathrm{d}\tau  = & - \frac{4}{5}\left( {\frac{1}{2}\mathbb{E}\left\| {{u}\left( s \wedge  \mathfrak{t} \right)} \right\|_{\mathcal{H}}^2 - \frac{1}{2}\mathbb{E}\left\| {{u}\left( t \wedge  \mathfrak{t} \right)} \right\|_{\mathcal{H}}^2  + \frac{1}{2} {\mathrm{Tr}} \,  Q \cdot \mathbb{E} \left( {t \wedge  \mathfrak{t} - s \wedge  \mathfrak{t}} \right) } \right) \\
= & - \frac{4}{5}\varepsilon .
\end{align*}
The proof is complete. $\hfill\square$
%%%%%%%%%%%%%%%%%%%%%%%%%%%%%%%%%%%%%%%%%%%%%
\subsection{The proof of nonexistence in Theorem \ref{CTSE-O-C}}
Our proof proceeds by contradiction. We assume that Eq.\eqref{BouE1} possess analytically weak solutions satisfying the regularity \eqref{djfkshhgss} and the strict energy inequality \eqref{SHUDJ-NENH-B-SIDIS}. Denote
\begin{equation*}
{{\tilde S}} \left( {t,l} \right) \triangleq  \int_{{\mathbb{S}^2}} {\int_{{\mathbb{T}^3}} {{{\left| {\left( {{\mathscr{T}_{l\mathbf{n}}}{u}\left( {t ,x} \right) - {u}\left( {t ,x} \right)} \right) } \right|}^3}} \, \mathrm{d}x} \mathrm{d}S\left( {\mathbf{n}} \right) .
\end{equation*}
It follows from the Kolmogorov 4/5 law \eqref{LHS-5-4-L-L-L} that
\begin{equation*}
\mathop {\lim }\limits_{l \to 0}  \frac{1}{l}\mathbb{E} \int_{s \wedge  \mathfrak{t}}^{t \wedge  \mathfrak{t}}  {S_{\shortparallel}\left( {\tau ,l} \right)} \, \mathrm{d}\tau  = - \frac{4}{5} \varepsilon  ,
\end{equation*}
which implies that
\begin{equation*}
\mathop {\lim \inf}\limits_{l \to 0} \frac{1}{l}\mathbb{E} \int_{s \wedge  \mathfrak{t}}^{t \wedge  \mathfrak{t}}  {\left| {{{\tilde S}} \left( {\tau ,l} \right)} \right|} \, \mathrm{d}\tau \geqslant C \varepsilon  .
\end{equation*}
Hence,
\begin{equation}
\mathop {\lim \inf}\limits_{l \to 0}   \ln \left( {\frac{1}{l} \mathbb{E} \int_{s \wedge  \mathfrak{t}}^{t \wedge  \mathfrak{t}}  {{{\tilde S}} \left( {\tau,l} \right)} \, \mathrm{d}\tau} \right) \geqslant   \ln \left( {C \varepsilon } \right). \label{xinshdkj-4}
\end{equation}
\par
Owing to \eqref{djfkshhgss}, the velocity field ${u}\left( {t,x} \right)$ is bounded in $C\left( {\left[ {0,T \wedge {\mathfrak{t}}} \right];{C^\gamma }\left( {{\mathbb{T}^3}} \right)} \right)$ with probability $1$. Then, we have
\begin{equation*}
\mathop {\sup }\limits_{x \in {\mathbb{T}^3}} \int_{s \wedge  \mathfrak{t}}^{t \wedge  \mathfrak{t}}  {\frac{{{{\left| {{\mathscr{T}_{l\mathbf{n}}}{u}\left( {\tau,x} \right) - {u}\left( {\tau ,x} \right)} \right|}^3}}}{{{l^{3\gamma }}}}} \, \mathrm{d}\tau  \leqslant C \int_{s \wedge  \mathfrak{t}}^{t \wedge  \mathfrak{t}}  {\left\| {{u}\left( \tau  \right)} \right\|_{{\mathcal{C}^\gamma }}^3} \, \mathrm{d}\tau  \leqslant C  , \ \  \mathbb{P}\mbox{-a.s.},
\end{equation*}
which implies that
\begin{equation*}
\mathbb{E} \int_{s \wedge  \mathfrak{t}}^{t \wedge  \mathfrak{t}}  {{{\tilde S}} \left( {\tau,l} \right)} \, \mathrm{d}\tau  \leqslant C  {l^{3\gamma }} ,
\end{equation*}
Therefore,
\begin{equation}
 \ln \left( {\frac{1}{l} \mathbb{E} \int_{s \wedge  \mathfrak{t}}^{t \wedge  \mathfrak{t}}  { {{\tilde S}} \left( {\tau,l} \right)} \, \mathrm{d}\tau} \right) \leqslant \left( {3\gamma  - 1} \right) \ln l + \ln \left( {C \left| {{\mathbb{T}^3}} \right|} \right) ,\ \ \forall \ l \in \left( {0,1} \right) .\label{CTSE-O-C-p-2}
\end{equation}
It follows from \eqref{xinshdkj-4} and \eqref{CTSE-O-C-p-2} that
\begin{equation*}
\gamma  \leqslant \frac{1}{3} ,
\end{equation*}
which contradicts the condition $\gamma  > \frac{1}{3} $. The proof is complete.  $\hfill\square$
%%%%%%%%%%%%%%%%%%%%%%%%%%%%%%%%%%%%%%%%%%%%%%
%%%%%%%%%%%%%%%%%%%%%%%%%%%%%%%%%%%%%%%%%%%%%%
%%%%%%%%%%%%%%%%%%%%%%%%%%%%%%%%%%%%%%%%%%%%%%
%%%%%%%%%%%%%%%%%%%%%%%%%%%%%%%%%%%%%%%%%%%%%%
\section*{Declarations}
\noindent \textbf{Acknowledgements}: This work is supported by NSFC (Grants 12531009 and 12571183).

\vspace{2mm}
%\noindent \textbf{Author Contributions}:

%{\textbf{Zhenxin Liu}}: equal contribution; editing, review, methodology, investigation, suggestions and improvements. 

%{\textbf{Wenhu Zhong}}: equal contribution; writing—original draft, methodology, conceptualization, investigation and formal analysis.

\noindent \textbf{Conflict of Interest Statement}: The authors have no conflicts to disclose.

\vspace{2mm}

\noindent \textbf{Data Availability}: No data was used for the research described in the article.

%%%%%%%%%%%%%%%%%%%%%%%%%%%%%%%%%%%%%%%%%%%%%%
%%%%%%%%%%%%%%%%%%%%%%%%%%%%%%%%%%%%%%%%%%%%%%
%%%%%%%%%%%%%%%%%%%%%%%%%%%%%%%%%%%%%%%%%%%%%%

\end{document}